\documentclass[11pt,a4paper]{article} 
\usepackage[latin1]{inputenc}
\usepackage{indentfirst}
\usepackage{amsmath}
\usepackage{amsfonts}
\usepackage{amsthm}
\usepackage{amssymb}
\usepackage{graphics}
\usepackage{hyperref}
\usepackage{graphicx}
\usepackage{multicol}
\usepackage{color} 
\usepackage{authblk}
\usepackage{epsfig}
\usepackage{subfigure}
\usepackage{hyperref}
\hypersetup{colorlinks=true,breaklinks=true,urlcolor=blue,linkcolor=blue,bookmarksopen=true}

\def \i{\mathbf{i}}

\newcommand{\cp}{{\mathcal P}}

\def\EE{\mathcal{EE}}

\newcommand{\Co}{\mathcal{C}}

\def\N{\mathbb{N}} 

\def\P{\mathbb{P}}

\def\E{\mathbb{E}}

\def\EE{\mathcal{E}}

\def\R{\mathbb{R}}

\def\t{\textrm}

\newcommand{\expp}[1]{\mathop {\mathrm{e}^{ #1}}}
\def\f{\bold{f}}
\def\w{\widetilde}

\def\ind{{\mathchoice {\rm 1\mskip-4mu l} {\rm 1\mskip-4mu l}
{\rm 1\mskip-4.5mu l} {\rm 1\mskip-5mu l}}}
\def\ud{d}

\newcommand{\be} {\begin{equation}}
\newcommand{\ee} {\end{equation}}
\newcommand{\bea} {\begin{eqnarray}}
\newcommand{\eea} {\end{eqnarray}}
\newcommand{\Bea} {\begin{eqnarray*}}
\newcommand{\Eea} {\end{eqnarray*}}

\newcommand{\linf}[1]{\underset{#1\to+\infty}\longrightarrow}

\numberwithin{equation}{section}

\def\be{\begin{eqnarray}}
\def\ee{\end{eqnarray}}
\def\ben{\begin{eqnarray*}}
\def\een{\end{eqnarray*}}

\def\me{\medskip \noindent}
\def\bi{\bigskip \noindent}
\def\E{\mathbb{E}}
\def\P{\mathbb{P}}

\def\f{{\cal F}}

\let\text=\textstyle

\def\One{{\mathbf 1}}   % Indicatrice

\newcommand{\loi}{{\cal L}}

\newcommand{\sm}{{s-}}
\newcommand{\ds}{\displaystyle}
\newcommand{\intot}{\displaystyle \int _0^t }

\newcommand{\indiq}{{\bf 1}}

\newcommand{\e}{{\epsilon}}

\newcommand{\intrd}{\ds{\int_{\rit^d}}}
\newcommand{\rit}{\mathbb{R}}
\newcommand{\nit}{\mathbb{N}}
\newcommand{\dit}{\mathbb{D}}

\def\me{\medskip\noindent}
\def\bi{\bigskip\noindent}

\newtheorem{ann}{Assumption}[section]
\newtheorem{Thm}{Theorem}[section]
\newtheorem{Lem}[Thm]{Lemma}

\newtheorem{Prop}[Thm]{Proposition}
\newtheorem{Cor}[Thm]{Corollary}

\newtheorem{Def}[Thm]{Definition}
\theoremstyle{definition} \newtheorem{Notation}{Notation}
\theoremstyle{definition} \newtheorem*{key}{Key words}
\theoremstyle{exercice} \newtheorem*{exo}{Exercise}
\theoremstyle{definition} \newtheorem*{ams}{A.M.S. Classification}

\theoremstyle{remark}

\newtheorem{rema}[Thm]{Remark}

\begin{document}
\title{ Some stochastic models for structured populations : \\
scaling limits and long time  behavior}
\author{  Vincent Bansaye\thanks{CMAP, Ecole Polytechnique, CNRS, route de
    Saclay, 91128 Palaiseau Cedex-France; E-mail: \emph{vincent.bansaye@polytechnique.edu}}
   \  \& \
    Sylvie M\'el\'eard\thanks{CMAP, Ecole Polytechnique, CNRS, route de
    Saclay, 91128 Palaiseau Cedex-France; E-mail: \emph{sylvie.meleard@polytechnique.edu}} 
 }
\maketitle \vspace{1.5cm}
\begin{abstract}

The first chapter  concerns monotype population models. We first  study  general birth and death processes and we give non-explosion and extinction criteria, moment computations and a pathwise representation.  We then show how different scales  may lead to different qualitative approximations, either ODEs or SDEs. The prototypes of these equations are the logistic (deterministic) equation and the logistic Feller diffusion process. The convergence in law of the sequence of processes is proved by tightness-uniqueness argument.  In these large population  approximations, the competition between individuals leads to nonlinear drift terms. 

We then focus on models without interaction but including exceptional events due either to demographic stochasticity or to environmental stochasticity. In the first case, an individual may have  a large number of offspring and we introduce  the class of continuous state branching processes. In the second case, catastrophes may occur and kill a random fraction of the population and  the process enjoys a quenched branching property.  We emphasize on the study  of the  Laplace transform, which allows us to classify  the long time behavior of these processes.

In the second chapter, we model structured populations by measure-valued stochastic differential equations. Our approach is based on the individual dynamics.  The individuals are characterized by parameters which have an influence on their survival or reproduction ability. Some of these parameters can be genetic and are  inheritable except when mutations occur, but they can also be a space location or a quantity of parasites.  The individuals compete for resources or other environmental constraints.
We describe the population by a point measure-valued Markov process. We study macroscopic approximations of this process depending on the interplay between different scalings and obtain in the limit either   integro-differential equations or reaction-diffusion equations or nonlinear super-processes. In each case, we insist on the specific techniques  for the proof of convergence and for the study of  the limiting model. The limiting processes offer different models of mutation-selection dynamics.

Then, we study two-level models motivated  by cell division dynamics, where the cell population is  discrete  and characterized by a trait, which may be continuous. In particular, we finely study
a process for parasite infection and the trait is  the parasite load. The latter grows following a Feller diffusion and  is randomly shared in the two daughter cells when the cell divides. Finally, we focus on  the neutral case when the rate of division of cells is constant but
the trait evolves following a general Markov process and may split in a random number of cells. The long time behavior  of the structured population is then  linked and derived from the  behavior a well chosen SDE (monotype population).
\end{abstract}

\begin{key}
Population models - Birth and death processes - large population approximations - Continuous state branching processes - Branching processes in random environment - Measure-valued Markov processes - Martingale properties - Two-level models - Cell division dynamics. 
\end{key}

\begin{ams} 60J80, 60J75, 60G57, 60H10, 92D25, 92D15.
\end{ams}

\tableofcontents
\section{Introduction}

This course concerns the stochastic modeling of population dynamics. In the first part, we focus on monotypic populations described by one dimensional stochastic differential equations with jumps. 
We consider their scaling limits for large populations and
study the long time behavior
of  the limiting processes. It is achieved thanks to  martingale properties, Poisson measure representations and stochastic calculus.
These tools and results  will be used and extended to  measure-valued processes in the second part. 
The latter is dedicated to structured populations, where individuals are characterized by a trait  belonging to  a continuum.  
  
In the first section, we  define birth and death processes  with   rates  depending on the state of the population and recall some long time properties based on recursion equations. A pathwise representation of the processes using Poisson point measures is introduced,  from which we deduce some martingale properties.  We represent  the carrying capacity of the underlying environment through a scaling  parameter $K\in \mathbb{N}$  and state our results in the limit of large $K$.   Depending on the  demographic rates,   the    population size  renormalized by  $K$  is  approximated either by  the solution of an ordinary differential equation or by  the solution of a stochastic differential equation. The proofs are based on   martingale properties and tightness-uniqueness arguments. When the  per individual  death rate is an affine function of  the population size,  in the limit we obtain  either a so called  logistic equation or     a logistic Feller diffusion process. The long time  behavior of these limiting dynamics is studied. Assuming a constant  per capita death rate leads to a Feller diffusion  which  satisfies  the branching property:  two disjoint subpopulations evolve independently. In that case, specific tools using Laplace transforms can be used. 
We extend this class of processes 
by  adding jumps, which may be due either to demographic stochasticity or to environmental stochasticity.
We consider  them separately and  we characterize their finite dimensional laws and  long time behavior using the branching property, the generator and martingale properties. First, we focus on Continuous State Branching Processes which arise as  scaling limits of branching processes when the individuals may have a very large number
of offspring. This gives rise to a jump term whose rate is proportional to the size of the population. Using the Lamperti transform, we can then both describe their scaling limits 
and classify the long time behavior : extinction, absorption at 
$0$ or exponential growth to infinity.
The second class  of   jump processes models random environmental catastrophes, which kill a random fraction of the population.  The continuous state process  can  be obtained as a limit of discrete particle systems, where the demographic dynamics of the population and the  environmental catastrophes occur on different timescales. Now, only the {\it quenched} branching property holds and the long time behavior of the Laplace exponent is more subtle.
We recover the three usual regimes, subcritical, critical and supercritical but the subcritical case is split  in three sub-cases leading  to  different  asymptotics for the survival probability. 

 The second part concerns structured populations whose  individuals are  characterized by a type taking values in a continuum. Two main examples are developed. The first one  models Darwinian evolution where the type is an heritable trait subject to to selection and mutation. The second case describes  cell division with  parasite infection  and the type of a cell is the amount of parasites it carries.  In both cases, the mathematical model is a  measure-valued Markov process with jumps. Therefore, we develop some stochastic tools for such processes and use a pathwise representation driven by Poisson point measures to obtain martingale properties. We consider different approximations of the  process, depending on the parameter $K$, which as before scales the 
population size but now also the demographic rates. The limiting theorems are proved using compactness-uniqueness arguments and the semimartingale decomposition of the measure-valued processes. \\
 In the first  two subsections, the population model includes mutations which may occur during each birth event with some positive probability. The mutant inherits a random perturbation of the ancestor's trait. The individuals compete for resources and the  individual death rate depends on the  whole population trait distribution, leading to nonlinearities in the limit. In the large population case, the limiting equation is a nonlinear integro-differential equation. In the allometric case, when the demographic rates are much larger but the mutation amplitude very small in an appropriate scale, the limiting object can be either a nonlinear reaction-diffusion equation or a nonlinear super-process. The latter is a continuous measure-valued  process whose law is characterized by martingale properties. Simulations show the qualitative differences between the
  trait supports for  these different asymptotics. It means that a change of scales in the parameters  leads to quite different evolutive scenarios. 
Let us point out  that  the classical models for population dynamics in  an ecological or mutational framework can thus be explained  from the  birth and death processes describing the evolution of the population at the level of the individuals.\\
 In the last two subsections, we describe  two-level models motivated by cell division dynamics.
First, we consider a finite population of dividing cells. The cells are infected by parasites which  may influence their division rates.
 The parasites  are more abundant and reproduce and die faster than the cells and their growth is  modeled
 by a Feller diffusion. When the cells divide, the parasite load is randomly shared in the two daughter cells. Following a random cell lineage (by keeping one daughter cell at random at each division)
makes appear a Feller diffusion with catastrophes. 
When studying the number of infected cells for large times, we obtain different regimes depending on the positivity or not of a parameter based on the division rate, the parasite splitting law and the parasite growth rate. Finally, we consider the long time behavior of a structured population when the genealogical tree is a branching process. It allows multiple offspring and  deaths. Between the branching events, the individual  traits evolve independently following a Markov process. 
The ergodicity of a well chosen one dimensional auxiliary Markov process
allows to prove the convergence of the trait distribution  within the population when time goes to infinity.

\bi
{\bf Notation}

\noindent
For a Polish space $E$, ${\cal P}(E)$ denotes the space of probability measures on $E$.

\me
The spaces $C^2_{b}(\mathbb{R})$,  $C^2_{b}(\mathbb{R}_{+})$,  $C^2_{b}(\mathbb{R}^d)$ are the spaces of bounded continuous functions whose first and second derivatives are bounded and continuous,  resp. on $\mathbb{R}$,  $\mathbb{R}_{+}$, $\mathbb{R}^d$. 

\me
In all what follows, $C$  denotes a constant real number whose value can change from one line to the other.

\bi
{\bf Acknowledgment}

\noindent The authors wish to warmly thank Amandine V\'eber for the  reading of the manuscript and her suggestions.

\noindent

\part{Discrete Monotype Population Models and One-dimensional Stochastic Differential Equations}

In the first chapter, we concentrate on one-dimensional models for population dynamics. After recalling the main properties of the birth and death processes, we study different scaling limits using a martingale approach. 
Then we investigate the long time behavior of some classes of limiting processes, in the case of large reproduction events or random environment using the branching property. 

\section{Birth and Death Processes}

\subsection{Definition and non-explosion criterion}

\begin{Def} A \index{birth and death process}  \textbf{birth and death process} is a pure jump Markov process whose  jumps steps are equal to   $\pm 1$. The transition rates are as follows: \ben
\left\{\begin{array}{ccc} i \to i+1 &\hbox{ at rate } &
\lambda_i\\i \to i-1 &\hbox{ at rate } & \mu_i,\end{array}\right.
\een  $(\lambda_i)_{i\in \mathbb{N}^*}$ and $(\mu_i)_{i\in \mathbb{N}^*}$  being two sequences of positive real numbers and 
 $\lambda_0=\mu_0=0$. \end{Def} 
 
\me In this case, the infinitesimal generator is the matrix $(Q_{i,j})$ defined on $\mathbb{N}\times \mathbb{N}$ by
$$Q_{i,i+1}=\lambda_{i}\ ,\ Q_{i,i-1}=\mu_{i}\ , \ Q_{i,i} = -(\lambda_{i}+\mu_{i})\,, \ \ Q_{i,j}=0 \hbox{ otherwise}.$$
%\ben Q = 
%\left(\begin{array}{cccccc}\mu_{1} & -(\lambda_{1}+\mu_{1}) & \lambda_{1}&0&0&\cdots \\
%0 & \mu_{2} & -(\lambda_{2}+\mu_{2}) &\lambda_{2}&0& \cdots\\
%0 & 0 & \mu_{3}& -(\lambda_{3}+\mu_{3}) &\lambda_{3}&\cdots\\
%\cdots & & & & & \\
%\cdots & & & & & \end{array}\right).\een
The global jump rate for a population with size  $i\geq 1$ is $\ \lambda_{i}+\mu_{i}$. After a random time distributed according an exponential law with parameter  $\ \lambda_{i}+\mu_{i}$, the process increases by $1$  with probability ${\lambda_{i}\over \lambda_{i}+\mu_{i}}$  and decreases by $-1$ with probability ${\mu_{i}\over \lambda_{i}+\mu_{i}}$. If $\ \lambda_{i}+\mu_{i} = 0$, the process is absorbed at $i$.

\me Recall that if $(P_{i,j}(t)  : t\in \mathbb{R}_{+})$ denotes the transition semigroup of the process, then
\ben
P_{i,i+1}(h) = \lambda_{i} \, h + o(h)\ ;\ 
P_{i,i-1}(h) = \mu_{i} \, h + o(h)\ ;\ 
P_{i,i}(h) = 1 - (\lambda_{i}+\mu_{i}) \, h + o(h).
\een

\me \textbf{Examples:} The constant  numbers $\lambda$, $\mu$, $\rho$, $c$ are positive. 

1) The \index{ Yule process} Yule process corresponds to the case $\lambda_i=i \lambda$, 
$\ \mu_i=0$.

2) The branching process or \index{linear birth and death process } linear birth and death process : $\lambda_i=i \lambda,\ \mu_i=i \mu$.

3) The \index{ birth and death process  with immigration} birth and death process  with immigration :
$\lambda_i=i \lambda + \rho,\ \mu_i=i \mu$.

4) The \index{logistic birth and death process} logistic birth and death process : 
 $\lambda_i=i \lambda,\ \mu_i=i \mu +
c\ i(i-1)$.

\bi The following theorem characterizes the  non-explosion  in finite time of the process. In this case,  the  process will be defined and will have a.s. finite value  at any time $t\in\mathbb{R}_{+}$.

\begin{Thm}
\label{nm-vie} Suppose that $\lambda_i>0$ for all $i\geq 1$.
Then the  birth and death process has almost surely  an infinite life time 
if and only if the following series diverges: \be
\label{non-exp} \sum_{i\geq 1}
\left({1\over \lambda_i}+{\mu_i\over
\lambda_i\lambda_{i-1}}+\cdots+{\mu_i\cdots
\mu_2\over\lambda_i\cdots\lambda_2\lambda_1}\right) = +\infty.\ee
\end{Thm}

\begin{Cor}
\label{xor-exist}
If for any $i$, $\lambda_{i}\leq \lambda \, i $, with $\lambda>0$,  the process is well defined on $\mathbb{R}_{+}$.
\end{Cor}

\begin{rema}\rm One can check that   the birth and death processes mentioned in the examples above satisfy this property and are well defined on $\mathbb{R}_{+}$.
\end{rema}

\begin{proof}[Proof of Theorem \ref{nm-vie}]
Let $(T_n)_n$ be the sequence of jump times  of the process and $(S_n)_n$ the sequence of the inter-jump times,
$$S_n=T_n-T_{n-1}, \quad  \forall n\geq 1; \quad T_0=0, \quad S_0=0.$$
We define $T_\infty=\lim_n T_n$. The process doesn't explode almost surely and  is well defined on $\mathbb{R}_{+}$ if and only if for any $ i \geq 1$,
$\ \mathbb{P}_i(T_\infty <+\infty) = 0$.

\noindent The proof consists in showing that the process doesn't explode almost surely if and only if  the unique non-negative and bounded solution  $x = (x_i)_{i\in \mathbb{N}}$ of
$\ Q\, x = x\ $
is the null solution. This proof is actually achieved  for any integer valued pure jump Markov process.
We will then see that it is  equivalent to \eqref{non-exp} for birth and death processes. 

\noindent  For any $i\geq 1$, we set $h_i^{(0)}=1$ and for $n\geq 1$, 
$$h_i^{(n)} =\mathbb{E}_i(\exp(-T_n))= \mathbb{E}_i\left(\exp(-\sum_{k=1}^n  S_k)\right).$$
We have
\begin{eqnarray*} \mathbb{E}_i\left(\exp\left(-\sum_{k=1}^{n+1} S_k\right) \big| S_1\right) %=
%\mathbb{E}_i\left(\exp(- S_1)\ \exp(-\sum_{k=2}^{n+1} S_k)\ | S_1\right)
 = \exp(- S_1)\   \mathbb{E}_i\left(\mathbb{E}_{X_{S_1}}\left(\exp(-\sum_{k=1}^n S_k)\right)\right),
\end{eqnarray*} 
by  the Markov property, the independence of $S_{1}$ and $X_{S_1}$ and  since the jump times of the shifted  process   are  $T_n - S_1$. Moreover,
$$\mathbb{E}_{i}\left(\mathbb{E}_{X_{S_1}}\left(\exp(-\sum_{k=1}^n S_k)\right)\right) = \sum_{j\neq i} \mathbb{P}_i(X_{S_1} = j)\ 
\mathbb{E}_j\left(\exp(-\sum_{k=1}^n S_k)\right) =  \sum_{j\neq i} {Q_{i,j}\over q_i}\  h_j^{(n)},$$
where  $q_{i}= \sum_{j\ne i} Q_{i,j}$.
Therefore, 
%$$
% \mathbb{E}_i\left(\exp(-\sum_{k=1}^{n+1} S_k)| S_1\right) = \sum_{j\neq i} {Q_{i,j}\over q_i}\ h_j^{(n)}% \mathbb{E}_j(\exp(-\sum_{k=1}^n S_k))
%\  \mathbb{E}_i(\exp(- S_1))$$
for all $n\geq 0$,
$$h_i^{(n+1)} =  \mathbb{E}_i\left(\mathbb{E}_i\left(\exp(-\sum_{k=1}^{n+1} S_k)\big| S_1\right)\right)= \sum_{j\neq i} {Q_{i,j}\over q_i} \ h_j^{(n)}\ \mathbb{E}_i(\exp(-S_1)).$$
Since $\mathbb{E}_i(\exp(-S_1))=\int_0^\infty q_i e^{-q_i s} e^{-s} ds = {q_i\over 1+ q_i},$
we finally obtain that 
\be
\label{rec-hn}
h_i^{(n+1)} = \sum_{j\neq i} {Q_{i,j}\over 1 + q_i}\  h_j^{(n)}.\ee
Let $(x_i)_i$ be a non-negative   solution of  $Qx=x$ bounded by $1$.  We get $h_{i}^{(0)}=1 \geq x_{i}$ and thanks to the previous formula, we  deduce by induction that for   all $i\geq 1$ and for all $n\in   \mathbb{N}$, 
$h_i^{(n)} \geq x_i\geq 0.$
Indeed if $ h_j^{(n)}\geq x_{j}$, we get $h_i^{(n+1)} \geq  \sum_{j\neq i} {Q_{i,j}\over 1+q_i} x_j$.
As $x$ is  solution of $Qx=x$, it satisfies  $x_i=\sum_{j} Q_{i,j} \ x_j = Q_{i,i} x_i + \sum_{j\neq i} Q_{i,j} x_j  = -q_i  x_i + \sum_{j\neq i} Q_{i,j} x_j,$
thus $ \sum_{j\neq i} {Q_{i,j}\over 1+q_i} x_j  =x_i$ and $h_i^{(n+1)}\geq x_{i}$.  

\noindent If the  process doesn't explode almost surely,  we have $T_\infty =+\infty$ a.s. and $\lim_n h_i^{(n)}=0$. Making $n$ tend to infinity in the previous inequality, we deduce that $x_i=0$. 
Thus, in this case, the unique non-negative and bounded solution of $Qx=x$ is the null solution. 

\noindent Let us now assume that the process explodes with a   positive probability. Let $\ z_i = \mathbb{E}_i(e^{-T_\infty})$.  There exists $i$ such that $\mathbb{P}_i(T_\infty<+\infty)>0$ and for this integer $i$, $z_i>0$.
Going to the limit with  $\ T_\infty =\lim_n T_n$ and $\ T_n=\sum_{k=1}^n S_k$ yields $\ z_j=\lim_n h_j^{(n)}$.  Making $n$ tend to infinity proves that $z$ is a non-negative and bounded solution  of 
$Q z = z$, with $z_i>0$. It ensures that the process   doesn't explode almost surely if and only if  the unique non-negative and bounded solution  $x = (x_i)_{i\in \mathbb{N}}$ of
$\ Q\, x = x\ $ is $x=0$.

\me We  apply this result to the birth and death process. We assume that 
  $\lambda_i>0$ for $i\geq 1$ and $\lambda_0=\mu_0=0$.
Let   $(x_i)_{i\in \mathbb{N}}$ be a non-negative  solution  of the equation $Qx=x$.   
For $n\geq 1$, introduce $\ \Delta_ n= x_n-x_{n-1}$. Equation $Qx=x$ can be written $\ x_0 = 0\ $ and  
$$\lambda_n x_{n+1} -(\lambda_n + \mu_n) x_n + \mu_n x_{n-1} = x_n\ ,\ \forall n \geq 1.$$
Setting $f_n= \displaystyle{1\over \lambda_n} $ and $ g_n= \displaystyle{\mu_n\over \lambda_n} $, we get
$$\Delta_1=  x_1\,;\, \Delta_2= \Delta_1 \,  g_1 + f_1\,  x_1\,;\,  \ldots \, ;\, \Delta_{n+1} = \Delta_n \,  g_n + f_n\,  x_n.$$
Remark that for all $n$, $\Delta_n \geq 0$ and the sequence $(x_n)_n$ is non-decreasing. 
If $x_1=0$, the solution is zero. Otherwise we deduce that $$ \Delta_{n+1} = f_n x_n +  \sum_{k=1}^{n-1} f_k\, g_{k+1}\cdots  g_n\, x_k\, +\,  g_{1}\cdots  g_n \, x_1.$$
Since $(x_k)_k$ is non-decreasing  and defining
$\ \ \displaystyle{r_n = {1\over \lambda_n}  + \sum_{k=1}^{n-1}  {\mu_{k+1} \cdots \mu_n\over \lambda_k  \lambda_{k+1} \cdots \lambda_n} +  {\mu_{1} \cdots \mu_n\over \lambda_1   \cdots \lambda_n}}  $,
 it follows that
$\ r_n \ x_1 \leq \Delta_{n+1} \leq  r_n \ x_n$,
and by iteration
$$x_1(1+ r_1+ \cdots +r_n ) \leq x_{n+1}\leq  x_1 \ \prod_{k=1}^n(1+r_k) .$$

\noindent Therefore we have proved that the boundedness of the sequence  $(x_n)_n$ is equivalent to the convergence of   $\ \sum_{k} r_k$ and Theorem \ref{nm-vie} is proved.
\end{proof}

\subsection{Kolmogorov equations and invariant measure}

\me Let us recall the  Kolmogorov equations, (see for example Karlin-Taylor\cite{Karlin1975}).

\me Forward Kolmogorov equation: for all $i,j \in \mathbb{N}$, 
\be
{dP_{i,j}\over dt}(t) &=& \sum_{k} P_{i,k}(t)\, Q_{k,j} = P_{i,j+1}(t) Q_{j+1,j} + P_{i,j-1}(t) Q_{j-1,j} + P_{i,j}(t) Q_{j,j}\nonumber \\
&=& \mu_{j+1} P_{i,j+1}(t) + \lambda_{j-1} P_{i,j-1}(t) - (\lambda_{j}+ \mu_{j}) P_{i,j}(t). \label{nm-prog}
\ee
Backward Kolmogorov equation: for all $i,j \in \mathbb{N}$, 
\be
{dP_{i,j}\over dt}(t) &=& \sum_{k} Q_{i,k} \, P_{k,j}(t) =  Q_{i,i-1} P_{i-1,j}(t) + Q_{i,i+1} P_{i+1,j}(t)  + Q_{i,i} P_{i,j}(t) \nonumber \\
&=& \mu_{i} P_{i-1,j}(t) + \lambda_{i} P_{i+1,j}(t) - (\lambda_{i}+ \mu_{i}) P_{i,j}(t).  \label{nm-retrog}
\ee
Let us define for all $j\in \mathbb{N}$ the probability measure $$p_{j}(t) = \mathbb{P}( X(t) = j) = \sum_{i} \mathbb{P}( X(t) = j|\, X(0)=i)  \mathbb{P}( X(0) = i) = \sum_{i} \mathbb{P}( X(0) = i) P_{i,j}(t).$$
A straightforward computation  shows that the forward Kolmogorov equation \eqref{nm-prog} reads
\be
\label{kolmo-BD}
{d\, p_{j}\over dt}(t)= \lambda_{j-1}\, p_{j-1}(t) + \mu_{j+1} \, p_{j+1}(t)
-(\lambda_{j} + \mu_{j}) \, p_{j}(t).
\ee
This equation is useful to find an invariant measure, that is a sequence $(q_{j})_{j}$ of nonnegative real numbers  with $\ \sum_{j} q_{j}<+\infty$ and satisfying for all $j$, 
$$\lambda_{j-1}\, q_{j-1} + \mu_{j+1} \, q_{j+1}
-(\lambda_{j} + \mu_{j}) \, q_{j} = 0.
$$

\subsection{Extinction criterion - Extinction time}

\me Some of the following  computation can be found in \cite{Karlin1975} or in \cite{Allen2011}, but they are finely developed in \cite{BansayeMeleardRichard}.

\me Let $T_{0}$ denote the extinction time and  $\ u_i = \mathbb{P}_{i}(T_{0}<\infty)\ $ the probability to see extinction in finite time starting from state $i$.

 \noindent Conditioning by the first jump $X_{T_{1}}\in \{-1,+1\}$,  we get the following recurrence property: for all $i\geq 1$,
 \be
\lambda_i u_{i+1}-(\lambda_i+\mu_i)u_i +\mu_i u_{i-1} =0 \label{ext1}
\ee

\noindent This equation can also be easily obtained from the backward Kolmogorov equation \eqref{nm-retrog}. Indeed
$$u_{i} = \mathbb{P}_{i}(\exists t>0, X_{t}=0) =\mathbb{P}_{i}(\cup_{t} \{X_{t}=0\}) = \lim_{t\to \infty} P_{i,0}(t),$$
and
$${dP_{i,0}\over dt}(t) =  \mu_{i} P_{i-1,0}(t) + \lambda_{i} P_{i+1,0}(t) - (\lambda_{i}+ \mu_{i}) P_{i,0}(t).$$

\me
Let us solve \eqref{ext1}.  We know that   $u_{0}=1$. Let us first assume that for a state  $N$,  $\lambda_{N}=0$ and $\lambda_{i}>0$ for  $i<N$.  Define
$\ u_i^{(N)}=\mathbb{P}_i(T_{0}<T_N)$,
where $T_{N}$ is the hitting time of $N$.  Thus  $u_{0}^N =1$ et $u_{N}^N =0$.
Setting $$U_N =\sum_{k=1}^{N-1} {\mu_1\cdots \mu_k\over
\lambda_1\cdots\lambda_k},$$
straightforward computations  using \eqref{ext1} yield that  for $i\in
\{1,\cdots,N-1\}$ \ben u_i^{(N)} = (1+U_N)^{-1} \sum_{k=i}^{N-1}
{\mu_1\cdots \mu_k\over \lambda_1\cdots\lambda_k} \quad \emph{and in particular }  \
u_1^{(N)} = {U_N\over 1+U_N}.\een 
 
 \noindent For the general case,  let   $N$ tend to infinity. We observe that extinction will happen (or not) almost surely in finite time depending on the convergence of the series  $\ \displaystyle{\sum_{k=1}^{\infty} {\mu_1\cdots \mu_k\over
\lambda_1\cdots\lambda_k}}$.

 \begin{Thm}
 \label{ext-bd}
 
 \medskip
 (i) If $\ \displaystyle{\sum_{k=1}^{\infty} {\mu_1\cdots \mu_k\over
\lambda_1\cdots\lambda_k} = +\infty}$, then the
 extinction probabilities $u_{i}$ are equal to  $1$.  Hence we have almost-sure  extinction of the birth and death process for any finite initial condition.
 
 \medskip
 (ii) If $\ \displaystyle{\sum_{k=1}^{\infty} {\mu_1\cdots \mu_k\over
\lambda_1\cdots\lambda_k}} = U_\infty<\infty$, then for $i\geq 1$,
 \ben u_i= (1+U_\infty)^{-1} \sum_{k=i}^\infty {\mu_1\cdots
\mu_k\over \lambda_1\cdots\lambda_k}.\een
There is a  positive probability for the process to survive for any  positive inital condition. 
\end{Thm}

\me
{\bf Application of  Theorem \ref{ext-bd} to the  binary branching process} (linear birth and death process): any  individual gives birth at rate $\lambda$ and dies at rate $\mu$. The population process is a binary branching process and individual life times are exponential variables with parameter $\lambda+\mu$. An individual   either gives birth to  $2$ individuals with probability
${\lambda\over \lambda+\mu}$ or dies with probability
${\mu\over \lambda+\mu}$.

\me Applying the previous results, one gets that when $\lambda\leq \mu$, i.e. when the process is sub-critical or critical, the sequence  $(U_N)_N$ tends to infinity with  $\,N$ and there is extinction with probability  $1$. Conversely, if 
$\lambda> \mu$, the sequence $(U_N)_N$ converges to  ${\mu\over
\lambda-\mu}$ and straightforward computations yield
$u_i=(\mu/\lambda)^i$.
%PlaÃƒÂ§ons-nous dans le cas oÃƒÂ¹ on a
%extinction presque-sÃƒÂ»re: $\lambda\leq \mu$.

\bi {\bf Application of Theorem  \ref{ext-bd} to the logistic birth and death process.} Let us assume that the birth and death rates are given by  \be \label{logistique}
\lambda_i=\lambda\, i\ ;\ \mu_i=\mu\, i+ c\, i(i-1).\ee
The parameter $c$ models the competition pressure between two individuals. 
It's easy to show that in this case, the series $\ \displaystyle{\sum_{k=1}^{\infty} {\mu_1\cdots \mu_k\over
\lambda_1\cdots\lambda_k}}\,$ diverges, leading to the almost sure extinction of the process. Hence  the competition between individuals makes the extinction inevitable.

\bi Let us now come back to the general case and assume that the series $\ \displaystyle{\sum_{k=1}^{\infty} {\mu_1\cdots \mu_k\over
\lambda_1\cdots\lambda_k}}$ diverges. The extinction time $T_{0}$ is well defined and we wish to compute its moments.

\me We use the standard notation $$ \pi_{1}= \frac{1}{\mu_{1}}\ ;\ \pi_{n}= \frac{\lambda_{1}\ldots \lambda_{n-1}}{\mu_{1}\ldots \mu_{n}}\ \quad  \forall n\geq 2.$$

\me
\begin{Prop} Let us assume that 
\be
\label{extinct}\ \sum_{k=1}^{\infty} {\mu_1\cdots \mu_k\over
\lambda_1\cdots\lambda_k} = \sum_{n} \frac{1}{\lambda_{n}\pi_{n}} = + \infty.\ee
Then 

(i)  For any $a>0$ and $n\geq1$,
\be
\label{recurrence_Gn}  G_{n}(a)=\mathbb{E}_{n+1}(\exp(-aT_n))=1+\frac{\mu_n+a}{\lambda_n}-\frac{\mu_n}{\lambda_n}\frac{1}{G_{n-1}(a)}.
\ee

(ii) $\E_{1}(T_{0}) = \sum_{k\geq 1} \pi_{k}$ and for every $n\geq 2$,
$$\E_{n}(T_{0}) = \sum_{k\geq 1} \pi_{k} + \sum_{k= 1}^{n-1}  \frac{1}{\lambda_{k}\pi_{k}} \sum_{i\geq k+1} \pi_{i}= \sum_{k= 1}^{n-1} 
\left(\sum_{i\geq k+1} \frac{\lambda_{k+1} \ldots \lambda_{i-1}}{\mu_{k+1}\ldots \mu_{i}}\right).$$
\end{Prop}

\begin{proof}
\noindent (i)  Let $\tau_n$ be a random variable distributed as $T_n$ under $\P_{n+1}$ and consider the Laplace transform of $\tau_n$. Following  \cite[p. 264]{Anderson1991} and
by the Markov property, we have
$$\tau_{n-1}\overset{(\mathrm{d})}=\One_{\{Y_n=-1\}}{\cal E}_n+\One_{\{Y_n=1\}}\left({\cal E}_n+\tau_{n}+\tau'_{n-1}\right)$$
where $Y_n$, ${\cal E}_n$, $\tau'_{n-1}$ and $\tau_{n}$ are independent random variables,
 ${\cal E}_n$ is an exponential  random variable with parameter $\lambda_n+\mu_n$ and $\tau'_{n-1}$ is distributed as $\tau_{n-1}$ and $\P(Y_n=1)=1-\P(Y_n=-1)=\lambda_n/(\lambda_n+\mu_n)$.
Hence, we get
$$G_{n-1}(a)=\frac{\lambda_n+\mu_n}{a+\lambda_n+\mu_n}\left(G_n(a)G_{n-1}(a)\frac{\lambda_n}{\lambda_n+\mu_n}+\frac{\mu_n}{\lambda_n+\mu_n}\right)$$
and \eqref{recurrence_Gn} follows. \\

 \noindent (ii) Differentiating \eqref{recurrence_Gn}  at $a=0$, we get
\begin{equation*}\label{eq_1750}
\E_{n}(T_{n-1}) = \frac{\lambda_n}{\mu_n} \ \E_{n+1}(T_{n}) +\frac{1}{\mu_{n}}, \quad n\geq1.
 \end{equation*}
Following the proof of Theorem \ref{ext-bd}, we first deal with the particular case when $\lambda_{N}=0$ for some $N>n$, $\ \E_{N}(T_{N-1}) = \frac{1}{\mu_{N}}$ and 
a simple induction gives
$$\E_{n}(T_{n-1}) =  \frac{1}{\mu_{n}} +  \sum_{i=n+1}^N\frac{\lambda_{n}\ldots \lambda_{i-1}}{\mu_{n}\ldots \mu_{i}}.$$

We get $\ \E_{1}(T_{0})  = \sum_{k=1}^N \pi_{k}$ and  writing $\ \E_{n}(T_{0}) = \sum_{k=1}^n \E_{k}(T_{k-1})$,
 we deduce that 
$$\E_{n}(T_{0})  = \sum_{k=1}^N \pi_{k} + \sum_{k=1}^{n-1}  \frac{1}{\lambda_{k}\pi_{k}} \sum_{i=k+1}^N \pi_{i}.$$
In the general case, let $N>n$. Thanks to \eqref{extinct}, $T_{0}$ is finite and the process a.s. does not explode in finite time for any initial condition. Then  $T_{N}\to \infty$ $\P_{n}$-a.s., where we use the convention $\{T_{N}=+\infty\}$ on the event where the process does not attain $N$. The monotone convergence theorem yields
$$\E_{n}(T_0 ;  T_0\leq T_N )\linf{N}
\E_{n}(T_0).$$ Let us  consider a birth and death process $X^N$ with birth and death rates $(\lambda^N_k,\mu^N_k : k\geq 0)$ such that $(\lambda^N_k,\mu^N_k)= ({\lambda}_k,{\mu}_k)$ for $k\ne N$ and ${\lambda}^N_N=0, {\mu}^N_N=\mu_N$. \\
 Since
$(X_t :  t\leq T_N)$ and $({X}^N_t : t\leq {T}^N_N)$ have the same distribution under $\P_{n}$, we get 
$$\E_{n}\left(T_0 ;  T_0\leq T_N \right)=\E_{n}\left(T_0^N;  T_0^N\leq T_N^N \right),$$
which yields $$\E_{n}(T_0)=\lim_{N\rightarrow \infty} \E_{n}\left(T_0^N ;  T_0^N\leq T_N^N \right)\leq \lim_{N\rightarrow \infty} \E_{n}\left(T_0^N \right),$$
where the convergence of the last term is due to the stochastic monotonicity of $T_0^N$  with respect to $N$ under $\P_{n}$.
Using now that  $T_0^N$ is stochastically smaller than $T_0$ under $\P_{n}$, we have also
$$\E_{n}(T_0)\geq \E_{n}(T_0^N).$$
We deduce that
$$\E_{n}(T_0)=\lim_{N\rightarrow \infty}\E_{n}(T_0^N)=\lim_{N\rightarrow \infty} \sum_{k= 1}^N \pi_{k} + \sum_{k= 1}^{n-1}  \frac{1}{\lambda_{k}\pi_{k}} \sum_{i= k+1}^N \pi_{i},$$
which ends up the proof.
\end{proof}

%Remark that in the proof, we have obtained in particular
%$$\E_{n+1}(T_{n})=\frac{1}{\lambda_{n}\pi_{n}} \sum_{i\geq n+1} \pi_{i} .$$

\me
\begin{exo}
 Assume \eqref{extinct}. 
Show that for every $n\geq 0$,
\ben
\label{mom2}
 \E_{n+1}(T_{n}^2)&=&\frac{2}{\lambda_n\pi_n}\sum_{i\geq n} \lambda_i\pi_i\,\E_{i+1}(T_{i})^2 ;\\
 \E_{n+1}(T_{n}^3)&=&\frac{6}{\lambda_n\pi_n}\sum_{i\geq n} \lambda_i\pi_i\,\E_{i+1}(T_{i})\,{\rm Var}_{i+1}(T_{i}).
\een
\end{exo}
\bi

\subsection{Trajectorial representation of  birth and death processes}

\me We consider as previously a birth and death process with birth rates $(\lambda_{n})_{n}$ and death rates $(\mu_{n})_{n}$. We write  $\lambda_{n}=\lambda(n) $ and $\mu_{n}= \mu(n)$, where $\,\lambda(.)$ and $\,\mu(.)$ are two functions defined on $\mathbb{R}_{+}$. We assume further that there exist $\bar \lambda>0$ and $ \bar \mu>0$ such that for any $x\geq 0$,
\be 
\label{hyp-mu} \lambda(x)\leq \bar \lambda\,x\quad ; \quad \mu(x) \leq \bar \mu(1+x^2).
\ee
 This assumption is satisfied for the logistic case where $\lambda(x)=  \lambda\, x$ and $\mu(x) = c x(x-1) +  \mu\, x$.

\me Assumption \eqref{hyp-mu} is a sufficient condition ensuring the existence of the process on $\mathbb{R}_{+}$, as observed in Corollary \ref{xor-exist}.

\begin{Prop}
\label{mom-1}
On the same probability space, we consider a   Poisson point measure $N(ds,du)$   with intensity $\, dsdu$ on $\mathbb{R}_{+} \times \mathbb{R}_{+}$ (see Appendix). We also consider a random variable $Z_{0}$ independent of $N$ and introduce the filtration $(\f_{t})_{t}$ given by $\f_{t}= \sigma(Z_{0}, N((0,s]\times A), s\leq t,  A \in {\cal B}(\mathbb{R}_{+}))$. 

\me  The left-continuous and right-limited non-negative Markov process $(Z_{t})_{t\geq 0}$ defined by
\be
\label{BD-traj}
Z_{t}= Z_{0} + \int_{0}^t\int_{\mathbb{R}_{+}} \left(\One_{\{u\leq \lambda(Z_{s-})\}} - \One_{\{ \lambda(Z_{s-})\leq u \leq  \lambda(Z_{s-}) +  \mu(Z_{s-})\}}\right)\,N(ds,du)
\ee
is a birth and death process with birth (resp. death) rates $(\lambda_{n})_{n}$ (resp. 
$(\mu_{n})_{n}$).

\me If for $p\geq 1$, $\E( Z_{0}^p)<+\infty$, then for any $T>0$,
\be
\label{BD-moment}
\E\big(\sup_{t\leq T}\, Z_{t}^p\big)<+\infty.\ee
\end{Prop}

\begin{proof}
For $n\in \mathbb{N}$, let us introduce the stopping times
$$T_{n}=\inf\{t>0, Z_{t}\geq n\}.$$
For $t\geq 0$, we have
\ben Z_{t\wedge T_{n}}^p = Z_{0}^p && + \int_{0}^{t\wedge T_{n}}  \int_{\mathbb{R}_{+}} \left((Z_{s-}+1)^p
 - Z_{s-}^p\right)\One_{\{u\leq \lambda(Z_{s-})\}}\, N(ds,du) \\
 && +  \int_{0}^{t\wedge T_{n}}  \int_{\mathbb{R}_{+}} \left((Z_{s-}-1)^p
 - Z_{s-}^p\right) \One_{\{ \lambda(Z_{s-})\leq u \leq  \lambda(Z_{s-}) +  \mu(Z_{s-})\}}\,N(ds,du).
\een
The second part of the r.h.s. is non-positive and the first part is increasing in time, yielding the upper bound 
$$\sup_{s\leq t}Z_{s\wedge T_{n}}^p \leq Z_{0}^p + \int_{0}^{t\wedge T_{n}}  \int_{\mathbb{R}_{+}} \left((Z_{s-}+1)^p
 - Z_{s-}^p\right)\One_{\{u\leq \lambda(Z_{s-})\}} \,N(ds,du).$$
 Since  there exists $C>0$ such that
 $\ (1+x)^{p} - x^p
\leq C(1+x^{p-1})$ for any $x\geq 0$ and by  \eqref{hyp-mu},
we get
 \ben
 \E(\sup_{s\leq t}Z_{s\wedge T_{n}}^p )\leq \E(Z_{0}^p) + C\, \bar{\lambda}\, \E\left(\int_{0}^{t\wedge T_{n}} Z_{s} \,(1+  Z_{s}^{p-1})\, ds\right)
\leq  \bar C \left(1 + \int_{0}^{t} \E\big(\sup_{u\leq s\wedge T_{n}}\,Z_{u}^p \big)\, ds\right),\een
where $\bar C$ is a positive number independent of $n$. 
Since the process is bounded by $n$ before $T_{n}$,  Gronwall's Lemma implies the existence (for any $T>0$) of a constant number $C_{T,p}$ independent of $n$ such that
\be
\label{Gronwall}\E\big(\sup_{t\leq T\wedge T_{n}}\,Z_{t}^p\big) \leq C_{T,p}.
\ee
In particular,  the sequence $(T_{n})_{n}$ tends to infinity almost surely. Indeed, otherwise there  would exist $T_{0}>0$ such that $\P(\sup_{n}T_{n} <T_{0})>0$. Hence
$\, \E\big(\sup_{t\leq T_{0}\wedge T_{n}}\,Z_{t}^p\big) \geq n^p \,\P(\sup_{n}T_{n} <T_{0})$,
which contradicts \eqref{Gronwall}. 
Making $n$ tend to infinity in \eqref{Gronwall} and using Fatou's Lemma yield  \eqref{BD-moment}.
\end{proof}

\me Remark that given $Z_{0}$ and $N$, the process defined by \eqref{BD-traj}
 is unique. Indeed it can be inductively constructed. It is thus unique in law.
Let us now recall its infinitesimal generator and give some martingale properties.

\begin{Thm}
\label{BD-mart}
Let us assume that $\E(Z_{0}^p)<\infty$, for $p\geq 2$. 

\me (i) The infinitesimal  generator of the Markov process $Z$  is defined for any bounded measurable  function $\phi$  from $\mathbb{R}_{+}$ into $\mathbb{R}$  by
$$L\phi(z) = \lambda(z)(\phi(z+1)-\phi(z)) + \mu(z)(\phi(z-1)-\phi(z)).$$

\noindent (ii)  For any measurable function $\phi$ such that $|\phi(x)| + |L\phi(x)| \leq C\,(1+x^p)$, the process $M^\phi$ defined by
\be
\label{mphi}
M^\phi_{t} = \phi(Z_{t})-\phi(Z_{0}) - \int_{0}^t L\phi(Z_{s})ds\ee
is a left-limited and right-continous (c\`adl\`ag)  $(\f_{t})_{t}$-martingale.

\me (iii)
The process $M$ defined by
\be
\label{mm} M_{t} = Z_{t} - Z_{0} - \int_{0}^t (\lambda(Z_{s}) - \mu(Z_{s}))ds\ee
is a square-integrable martingale with quadratic variation
\be
\label{varq}\langle M\rangle_{t} = \int_{0}^t (\lambda(Z_{s}) + \mu(Z_{s}))ds.\ee
\end{Thm}
 
 \me  Remark that  the drift term of \eqref{mm} involves the difference between the birth and  death rates (i.e. the growth rate), while \eqref{varq} involves  the sum of both rates.  Indeed the drift term describes the mean behavior whereas the quadratic variation reports  the random fluctuations. 
 %%%%%%%%
 
 \me
  \begin{proof}
 $(i)$ is well known.
 
 \me  $(ii)$  Dynkin's theorem implies that $M^\phi$ is a local martingale.
 By the assumption on  $\ \phi$ and \eqref{BD-moment}, all the terms of the r.h.s. of  \eqref{mphi} are integrable. Therefore $M^\phi$ is a martingale.
 
\me $(iii)$ We first assume that $\E(Z_{0}^3)<+\infty$. By \eqref{hyp-mu}, we may apply $(ii)$ to both functions $\phi_{1}(x)=x$ and  $\phi_{2}(x)=x^2$. Hence 
$\ M_{t} = Z_{t} - Z_{0} - \int_{0}^t (\lambda(Z_{s}) - \mu(Z_{s}))ds\,$ and  $
Z_{t}^2 - Z_{0}^2- \int_{0}^t \big(\lambda(Z_{s})(2Z_{s}+1) - \mu(Z_{s})(1-2Z_{s})\big)ds\,$ are martingales. The process $Z$ is a semi-martingale and  It\^{o}'s formula applied to $Z^2$ gives  that $
Z_{t}^2 - Z_{0}^2- \int_{0}^t 2Z_{s} \big(\lambda(Z_{s}) - \mu(Z_{s})\big)ds - \langle M\rangle_{t}\,$ is a martingale. The uniqueness of the Doob-Meyer decomposition leads to
\eqref{varq}. The general case $\E(Z_{0}^2)<+\infty$ follows by a standard localization argument. 
\end{proof}

\section{Scaling Limits for Birth and Death Processes}
\label{LPA}

\me If the  population is  large, so many birth and death events occur that  the dynamics becomes difficult to describe individual per individual. Living systems need resources in order  to survive and reproduce and the biomass per capita depends on the order of magnitude of these resources. 
 We introduce a  parameter $K\in\mathbb{N}^*=\{1,2,\ldots\}$ scaling either the size of the population or the total amount of resources. We assume that the individuals are weighted by ${1\over K}$. 
 
 \me In this section, we  show that depending on the scaling relations  between the population size and  the demographic parameters, the population size process will be approximate either by a deterministic process or by a stochastic process. These approximations will lead to different long time behaviors. 
 
\me  In the rest of this section, we consider a sequence of  birth and death processes $Z^K$ parametrized by $K$, where the  birth and death  rates for the population state $n\in \mathbb{N}$ are given by $\lambda_{K}(n)$ and $\mu_{K}(n)$. Since the individuals are weighted by ${1\over K}$,   the  population dynamics  is modeled by the  process $(X^K_{t}, t\geq 0)\in \mathbb{D}(\mathbb{R}_{+}, \mathbb{R}_{+})$  with   jump amplitudes  $\pm {1\over K}$ and defined for $t\geq 0$ by
\be
\label{pop-weight}
X^K_{t} = {Z^K_{t}\over K}.
\ee
 This process is a Markov process with generator
 \be
 \label{gen}L_{K}\phi(x)= \lambda_{K}(Kx)\big(\phi(x+{1\over K})-\phi(x)\big) +  \mu_{K}(Kx)\big(\phi(x-{1\over K})-\phi(x)\big).
 \ee
 Therefore, adapting Proposition \ref{mom-1} and Theorem \ref{BD-mart}, one can easily show that if $\,\lambda_{K}(n)\leq \bar \lambda n$ (uniformly in $K$) and if
 \be
 \label{init-1} \sup_{K} \E((X^K_{0})^3)<+\infty,\ee
 then 
 \be
 \label{unif-1}
 \sup_{K} \E(\sup_{t\leq T}(X^K_{t})^3)<+\infty,\ee
  and for any $K\in \mathbb{N}^*$, the process 
  \be
\label{mar-K}\ M^K_{t} = X^K_{t} - X^K_{0} - {1\over K}\int_{0}^t (\lambda_{K}(Z^K_{s}) - \mu_{K}(Z^K_{s}))ds\ee
is a square integrable martingale with quadratic variation 
\be
\label{var-K}
 \langle M^K\rangle_{t} = {1\over K^2}\, \int_{0}^t (\lambda_{K}(Z^K_{s}) + \mu_{K}(Z^K_{s}))ds.\ee

 \subsection{Deterministic approximation  -  Malthusian and logistic Equations }

\me Let us now assume that the birth and death rates satisfy the following assumption:
\be
\label{coeff-1}
&&\lambda_{K}(n) = n \lambda\left({n\over K}\right); \ \mu_{K}(n) = n \mu\left({n\over K}\right), \   \hbox{where the functions } \ \nonumber\\
&&  \lambda\  \hbox{and} \ \mu\  \hbox{ are non negative and Lipschitz continuous on } \ \mathbb{R}_{+},\nonumber\\
&& \lambda(x) \leq \bar \lambda \quad ; \quad \mu(x)\leq \bar \mu (1+x).
\ee

\me We will focus on two particular cases:
\begin{description}
\item  The linear case: $\lambda_{K}(n) = n \lambda$ and $\ \mu_{K}(n) = n \mu$, with $\lambda, \mu >0$. 

\item 
 The logistic case:  $\lambda_{K}(n) = n \lambda$ and $\ \mu_{K}(n) =n (\mu + \displaystyle{c\over K} n)$ with $\lambda, \mu, c>0$. \end{description}
By \eqref{init-1}, the population size is of the order of magnitude of $K$ and the biomass per capita is  of order ${1\over K}$. This explains that  the competition pressure from one individual  to another one in the logistic case is proportional to 
${1\over K}$.

\bi We are interested in the limiting behavior of the process $(X^K_{t}, t\geq 0)$ when $K\to \infty$.

\begin{Thm}
\label{det-1}
Let us assume  \eqref{coeff-1}, \eqref{init-1} and that the sequence $(X^K_{0})_{K}$ converges in law (and in probability)  to a real number $x_{0}$. Then for any $T>0$, the sequence of processes $(X^K_{t}, t\in[0,T])$ converges in law (and hence in probability), in $\mathbb{D}([0,T], \mathbb{R}_{+})$, to the continuous deterministic function $(x(t), t\in[0,T])$ solution of the ordinary differential equation
\be
\label{lim-det}
x'(t) = x(t)(\lambda(x(t)) - \mu(x(t)))\ ; x(0)=x_{0}.
\ee
\end{Thm}

\me In the linear case, the limiting equation is the Malthusian equation
$$x'(t) = x(t)(\lambda-\mu).$$ In the logistic case, one obtains the logistic equation
\be
\label{log}
x'(t)=x(t)(\lambda-\mu-c\,x(t)).\ee These two equations have different long time behaviors. In the Malthusian case, depending on the sign of $\lambda-\mu$, the solution of the equation  tends to $+\infty$ or to $0$ as time goes to infinity, modeling the explosion or extinction of the population.  In the logistic case and  if   the growth rate $\lambda-\mu$ is positive, the solution  converges to the carrying capacity $\displaystyle\ {\lambda-\mu\over c}>0$. The competition between individuals  yields a regulation of the population size. 

\me
\begin{proof}
The proof is based on a compactness-uniqueness argument.  More precisely, the scheme of the proof is the following:

1) Uniqueness of the limit.

2) Uniform estimates on the moments.

3) Tightness of the sequence of laws of $(X^K_{t}, t\in [0,T])$ in the Skorohod space. We will use the Aldous and Rebolledo criterion.

4) Identification of the limit. 

\me 
Thanks to Assumption \eqref{coeff-1}, the uniqueness of the solution of equation \eqref{lim-det} is obvious. 
We also have \eqref{unif-1}. Therefore it remains to prove the  tightness of the sequence of laws and to identify the limit. 
Recall (see for example \cite{EK} or \cite{JM86}) that since the processes $(X^K_{t}= X^K_{0}+ M^K_{t}+ A^K_{t})_{t}$
 are semimartingales,  tightness will be proved as soon as we have

$(i)$ The sequence of laws of $(\sup_{t\leq T}|X^K_{t}|)$ is tight,

$(ii)$ The finite variation processes $\langle M^K\rangle$ and $A^K$ satisfy the Aldous conditions. 

\me  
Let us recall the Aldous condition (see \cite{Al78}): let  $(Y^K)_K$ be a sequence of $\f_{t}$-adapted processes   and  $\tau$  the set of stopping times for the filtration $(\f_{t})_{t}$.  The Aldous  condition can be written: 
$\forall \varepsilon>0$, $\forall \eta>0$, $\exists \delta>0$, $K_{0}$ such that
$$\sup_{K\geq K_{0}}\, \sup_{S,S'\in \tau; S\leq S'\leq (S+\delta)\wedge T}\P(|Y^K_{S'}-Y^K_{S}|>\varepsilon)\leq \eta.$$

Let us show this property for the sequence $(A^K)_{K}$. 
We have
\ben
\E(|A^K_{S'}-A^K_{S}|)&\leq& \E\left(\int_{S}^{S'}X^K_{s}|\lambda(X^K_{s})-\mu(X^K_{s})| ds\right)\\
&\leq& C\E\left(\int_{S}^{S'}(1+(X^K_{s})^2)ds\right)\quad  \ \hbox{by }\  \eqref{coeff-1}\\
&\leq& C\, \delta\,\E\left(\sup_{s\leq T}(1+(X^K_{s})^2)\right)
\een
which tends to $0$ uniformly in $K$ as $\delta$ tends to $0$.
We use a similar argument for $(\langle M^K\rangle)_{K}$ to conclude for the tightness of the laws of $(X^K)_{K}$.  Prokhorov's Theorem implies the relative compactness of this family of laws in the set of probability measures on $\mathbb{D}([0,T],\mathbb{R})$, leading to the existence of a limiting value $\,Q$. 

\me Let us now identify the limit. The jumps of $X^K$ have the amplitude ${1\over K}$. 
Since the mapping $x\to \sup_{t\leq T}|\Delta x(t)|$ is continuous from $\mathbb{D}([0,T],\mathbb{R})$ into $\mathbb{R}_{+}$, then the probability measure  $Q$  only charges  the subset of continuous functions. For any $t>0$, we define on $\mathbb{D}([0,T],\mathbb{R})$ the function
$$\psi_{t}(x) = x_{t}- x_{0} - \int_{0}^t (\lambda(x_{s}) - \mu(x_{s}))\ x_{s} ds. $$
The assumptions yield $$|\psi_{t}(x)| \leq C\, \sup_{t\leq T}(1+(x_{t})^2)$$
and  we deduce the uniform integrability of the sequence $(\psi_{t}(X^K))_{K}$ from \eqref{unif-1}. 
The projection mapping $x \to x_{t}$ isn't continuous on $\mathbb{D}([0,T],\mathbb{R})$ but since $\,Q$ only charges the continuous paths, we deduce that $X \to \psi_{t}(X)$ is $Q$-a.s. continuous, if $X$ denotes the canonical process. Therefore, since $Q$ is the weak limit of a subsequence of $({\cal L}(X^K))_{K}$ (that for simplicity we still denote  ${\cal L}(X^K)$) and using the uniform integrability of $(\psi_{t}(X^K))_{K}$, we get
$$\E_{Q}(|\psi_{t}(X)| ) = \lim_{K} \E(|\psi_{t}(X^K)| ) =\lim_{K} \E(|M^K_{t}| ).$$
But
$$ \E(|M^K_{t}| ) \leq \left(\E(|M^K_{t}| ^2)\right)^{1/2}$$
tends to $0$ by  \eqref{var-K}, \eqref{coeff-1} and  \eqref{unif-1}. Hence the limiting process $X$  is the deterministic solution of the  equation
$$x(t) = x_{0} + \int_{0}^t x_{s} (\lambda(x_{s}) - \mu(x_{s}))  ds.
$$
That ends the proof. 
\end{proof}

\me
\subsection{Stochastic approximation - Feller and logistic Feller diffusions}
\label{ScalingDim1}
Let us now  assume that 
\be
\label{coeff-2}
&&\lambda_{K}(n) = n\left(\gamma K + \lambda\right); \ \mu_{K}(n) = n \left(\gamma K + \mu + {c\over K} n\right),
\ee
where $\gamma, \lambda, \mu, c$ are nonnegative constants and $\lambda>\mu$.
The coefficient $\gamma>0$ is called the allometry coefficient. Such   population model describes the behavior of small individuals  which are born or die very fast. 
As we will see in the next theorem, this assumption changes the qualitative nature of the large population approximation. 
\begin{Thm} \label{Thlimstoc} Assume \eqref{coeff-2} and 
\eqref{init-1} and that the random variables $\,X^K_{0}$ converge in law to a   random variable $X_{0}$. Then   for any $T>0$, the sequence of processes $(X^K_{t}, t\in[0,T])$ converges in law, in $\mathbb{D}([0,T], \mathbb{R}_{+})$, to the continuous diffusion process $(X_{t}, t\in[0,T])$ solution of the stochastic  differential equation
\be
\label{lim-sto}
X_{t} =X_0+\int_0^t \sqrt{2\gamma X_{s}} dB_{s} + \int_0^t X_s(\lambda - \mu - c\,X_{s})ds.
\ee
\end{Thm}

\noindent In this case, the limiting process is stochastic. Indeed there are so many birth and death jump events that the stochasticity cannot completely disappear. 
 Hence the  term $\sqrt{2\gamma X_{t}} dB_{t} $ models  the demographic stochasticity. Its variance is proportional to the renormalized population size. When $c=0$, we get the Feller diffusion equation
\be
\label{Feller}dX_{t} = \sqrt{2\gamma X_{t}} dB_{t} + X_t(\lambda-\mu)dt.\ee If $c\neq 0$, Equation \eqref{lim-sto} is called by extension the logistic Feller diffusion equation  (see Etheridge \cite{Et04} and Lambert \cite{lam05}). 
 
 \me
 \begin{proof}
 Here again the proof is based on a uniqueness-compactness argument.
 
\noindent Let us first prove the uniqueness in law of a solution of \eqref{lim-sto}. We use a general result  concerning one-dimensional stochastic
differential equations (see Ikeda-Watanabe \cite{SDEsApp1} p.448). The diffusion and drift coefficients are of class  $C^1$ and non zero on $(0,+\infty)$ but 
can cancel at $0$. So Equation \eqref{lim-sto} is uniquely defined until the stopping time $T_{e}=T_{0}\wedge T_{\infty}$ where $T_{0}$ is the hitting time 
of $0$ and $T_{\infty}$ the explosion time. Furthermore, $0$ is an absorbing point for the process.  In the particular case where $c=0$ (no interaction), 
the process stays in $(0,\infty)$ or goes to extinction almost surely (see  Subsection 4.1 and Proposition \ref{fel-beh}). When $c>0$,   the process goes to extinction almost surely, as recalled below.
 
 \begin{Lem}
 \label{extinction}
 For any $x>0$, $\P_{x}(T_{e}=T_{0}<+\infty) =1$ if $c>0$. 
 \end{Lem}
 \begin{proof}[Proof of Lemma \ref{extinction}] 
 Recall Ikeda-Watanabe's results in \cite{SDEsApp1} (see also Shreve-Karatzas \cite{Karatzas:98}  Prop. 5.32). Let
 $Y_{t}$ denote the solution of the one-dimensional  stochastic differential equation 
 $dY_{t} = \sigma(Y_{t}) d B_{t} + b(Y_{t}) dt$. Let us  introduce the two functions:
 \ben 
 \Lambda(x) &=& \int_{1}^x \exp\left(-\int_{1}^z \frac{2 b(y)}{\sigma^2(y)} dy \right) dz ;\\
 \kappa(x) &=& \int_{1}^x  \exp\left(-\int_{1}^z \frac{2 b(y)}{\sigma^2(y)} dy \right) \left(\int_{1}^z  \exp\left(\int_{1}^\eta \frac{2 b(y)}{\sigma^2(y)} dy \right){d \eta\over \sigma^2(\eta)} \right) dz.
 \een
 Then there is equivalence between the two following assertions:
 
 (a) For any $y>0$, $\P_{y}( T^Y_{e}=T^Y_{0}<+\infty) =1$. 
 
 (b) $\Lambda(+\infty)=+\infty\ $ and $\ \kappa(0^+)<+\infty$.
 
 \me In our case, straightforward computations allow us to show that (b) is satisfied by the solution of \eqref{lim-sto} as soon as $c\neq 0$. 
 \end{proof}
% \end{proof}
 
 \me 
Let us now prove that there exists a constant $C_{1,T}$ such that 
 \be
\label{inter}
\sup_{t\leq T}\sup_{K} \E((X^K_{t})^3) \leq C_{1,T},\ee 
where $C_{1,T}$ only depends on $T$.
Following the proof of Theorem \ref{BD-mart} $(iii)$, we note that  
 \ben
 &&(X^K_{t})^3 - (X^K_{0})^3  - \int_{0}^t \bigg\{\gamma K^2 X^K_{s} \bigg((
 X^K_{s} + {1\over K})^3 + (
 X^K_{s} - {1\over K})^3 - 2 (
 X^K_{s} )^3\bigg) \\
 &&+ \lambda K   X^K_{s} \bigg((
 X^K_{s} + {1\over K})^3 -  (
 X^K_{s} )^3\bigg) + (\mu K + c X^K_{s} ) X^K_{s} \bigg(
 ( X^K_{s} - {1\over K})^3-  (
 X^K_{s} )^3\bigg)\bigg\} ds
   \een 
   is a local martingale. 
 Therefore, using that $(x+{1\over K})^3
+(x-{1\over K})^3 -2 x^3 = {6\over K^2} x$, a localization argument and  Gronwall's inequality, we get \eqref{inter}.

\noindent  Hence, we may deduce a pathwise second order moment estimate:
\be
\label{borne} \sup_{K} \E\big(\sup_{t\leq T}(X^K_{t})^2\big) \leq C_{2,T},\ee where $C_{2,T}$ only depends on $T$. Indeed, we have
$$X^K_{t} = X^K_{0}+ M^K_{t}+ \int_{0}^t X^K_{s}(\lambda  - \mu  - c X^K_{s}) ds,$$
where $M^K$ is a martingale. Then, there exists $C'_{T}>0$ with
$$\E\big(\sup_{s\leq t }(X^K_{s})^2\big) \leq C'_{T}\left(\E((X^K_{0})^2) +   \sup_{s\leq t}\E((X^K_{s})^2) + \E\big(\sup_{s\leq t }(M^K_{s})^2\big) \right),$$
and by Doob's inequality,
\ben \E\big(\sup_{s\leq t }(M^K_{s})^2\big)&\leq& C \E(\langle M^K \rangle_{t}) =  C\E\left(\int_{0}^t \big(2 \gamma X^K_{s} + {X^K_{s}\over K} ( \lambda+ \mu + c X^K_{s}) \big)ds\right).\een

\noindent  Finally Gronwall's Lemma and  \eqref{inter}  allow to get \eqref{borne}. 
The proof of the tightness follows as in the proof of Theorem \ref{det-1}. 

 \me Let us now identify the limit. We consider a limiting value $Q$. 
Remark once again  that since the mapping $x\to \sup_{t\leq T}|\Delta x(t)|$ is continuous from $\mathbb{D}([0,T],\mathbb{R})$ into $\mathbb{R}_{+}$, then 
 $Q$  charges only the continuous paths. For any $t>0$ and $\phi\in C^2_{b}$, we define on $\mathbb{D}([0,T],\mathbb{R})$ the function
$$\psi_{t}^1(x) = \phi(x_{t})- \phi(x_{0}) - \int_{0}^t L\phi(x_{s})\  ds,$$ where
$\displaystyle{L\phi(x)=\gamma\, x\,\phi''(x) +((\lambda-\mu)x-cx^2)\, \phi'(x)}$.
Note that  $|\psi_{t}^1(x)| \leq C\, \int_{0}^T(1+x_{s}^2) ds$, which implies
 the uniform integrability of the sequence $(\psi_{t}^1(X^K))_{K}$ by \eqref{inter}. 

\noindent Let us  prove that the process $(\psi_{t}^1(X), t\geq 0)$ is a $Q$-martingale. That will be  right as soon as $\E_{Q}(H(X))=0$ for any function $H$ defined as follows:
$$H(X) =g_{1}(X_{s_{1}})\cdots g_{k}(X_{s_{k}})(\psi_{t}^1(X) - \psi_{s}^1(X)  ),$$ for  $0\leq s_{1}\leq \cdots\leq s_{k}\leq s<t$ and $g_{1},\cdots, g_{k}\in C_{b}(\mathbb{R}_{+})$. 

\noindent Now, $\phi(X^K_{t})$ is a semimartingale and 
$$\phi(X^K_{t})=\phi(X^K_{0}) + M^{K,\phi}_{t} + \int_{0}^t L_{K}\phi(X^K_{s}) ds.$$
Moreover, if $\phi \in C^3_b$,   
$\displaystyle{|L_{K}\phi(x) - L\phi(x)| \leq {C\over K} (1+x^2)}\ $
and  \be
\label{diffgen}\E\left(\int_0^T|L_{K}\phi(X^K_{s}) - L\phi(X^K_{s})| ds \right)\leq {C\over K} \E(\sup_{s\leq T} (1+|X^K_{s}|^2)).\ee
We denote by $ \psi_{t}^K(X) $  the similar function as $\psi_{t}^1(X)$ with $L$  replaced by $L_{K}$. 

\me
The function $H_{K}$ will denote the  function similar to $H$  with $\psi_{t}^1$ replaced by $\psi_{t}^K$. We write
 $$\E_{Q}(H(X))= \E_{Q}(H(X))-\E(H(X^K))  + \big(\E_{Q}(H(X^K)) - \E(H_{K}(X^K))\big) +  \E(H_{K}(X^K)).$$
The third term is zero since $\psi_{t}^K(X^K) $ is a martingale. It's easy to prove the convergence to $0$ of the second term using  \eqref{diffgen}. The first
term tends to $0$ by continuity and uniform integrability. 
Hence we have proved that under $Q$ the limiting process satisfies the following martingale problem: for any $\phi\in C^3_{b}$, 
the process $\phi(X_{t}) -\phi(X_{0}) - \int_{0}^t L\phi(X_{s}) ds$ is a martingale. We know further that for any $T>0$, $\E(\sup_{t\leq T}(X_{t})^2)<+\infty$. 
It remains to show that under $Q$, $(X_{t})$ is the unique solution of \eqref{lim-sto}.
 Such point is standard and  can be found for example in Karatzas-Shreve \cite{Karatzas:98} but we give a quick proof.
Applying the martingale problem to $\phi(x)=x$, then $\phi(x)=x^2$, we get that $X_{t}$ is a square integrable semimartingale
 and that $\ X_{t}=X_{0}+ M_{t}+ \int_{0}^t X_{s}(\lambda-\mu-cX_{s})\, ds$ and the martingale part $M_{t}$ has  quadratic variation $\int_{0}^t 2  \gamma X_{s} ds$. Then a representation theorem is used to conclude. Indeed, let us increase the probability space and consider an auxiliary space $(\Omega', {\cal A}', \P')$ and a Brownian motion $W$ defined  on the latter. On $\Omega \times \Omega'$, let us define 
$$B_{t}(\omega, \omega') = \int_{0}^t {1\over \sqrt{2 \gamma X_{s}(\omega)}} \One_{\{X_{s}(\omega)\neq 0\}}dM_{s}(\omega) + \int_{0}^t  \One_{\{X_{s}(\omega) = 0\}} dW_{s}(\omega').$$
It's obvious that the processes $B_{t}$ and $B_{t}^2 -t$ are continuous martingale on the product probability space. Then $B$ is a Brownian motion by the L\'evy's characterization. In addition, we compute
$$\E\left(\left(M_{t}- \int_{0}^t \sqrt{2 \gamma X_{s}}dB_{s}\right)^2\right) = \E\left(\int_{0}^t  \One_{\{X_{s}(\omega) = 0\}}d\langle M\rangle_{s}\right)=0.$$
Thus, $M_{t}= \int_{0}^t \sqrt{2 \gamma X_{s}}dB_{s}$, which ends the proof.
 \end{proof}

\subsection{Selection strategy in random environments}

\me In \eqref{lim-sto}, the stochastic term   is demographic in the sense that,  as seen in the previous section, it comes from a  very high number of births and deaths. Another stochastic term can be added to the deterministic equation to model a random environment. In Evans, Hening and Schreiber \cite{evans}, the authors consider the population abundance process $(Y_{t}, t\geq 0)$  governed by the stochastic differential equation
\be
\label{rand-env}dY_{t} = Y_{t}(r-c Y_{t}) + \sigma Y_{t} dW_{t}\ ,\ Y_{0}>0,\ee
where $(W_{t})_{t\geq 0}$ is a standard Brownian motion. 

\noindent   The growth rate has a stochastic component whose $\sigma^2$ is the infinitesimal variance.
The process is well defined
 and has an explicit form which can be checked using It\^o's formula:
$$Y_{t}=\frac{Y_{0}\exp\big((r-{\sigma^2\over 2})t + \sigma W_{t}\big)}{1+Y_{0} {r\over c}\,\int_{0}^t \exp\big((r-{\sigma^2\over 2})s + \sigma W_{s}\big)\,ds}.
$$
Then $Y_{t}\geq 0$ for all $t\geq 0$ almost surely. 

\me The authors deduce the long time behavior of the process depending on the sign of $r - {\sigma^2\over 2}$. We refer to  \cite{evans} for the proof.

\begin{Prop}
\begin{enumerate}
\item If $r - {\sigma^2\over 2}<0$, then $\lim_{t\to \infty} Y_{t} = 0$ almost surely.

\item If $r - {\sigma^2\over 2}=0$, then $\liminf_{t\to \infty} Y_{t} = 0$ almost surely, $\limsup_{t\to \infty} Y_{t} = \infty$ almost surely and $\lim_{t\to \infty} {1\over t}\int_{0}^t Y_{s}ds = 0$ almost surely.

\item If $r - {\sigma^2\over 2}>0$, then $(Y_{t})_{t}$ has a unique stationary distribution which is the  law $\Gamma({2 r\over \sigma^2}-1,{\sigma^2\over 2 c})=\Gamma(k,\theta)$, with density $x \longrightarrow {1\over \Gamma(k) \theta^k}x^{k-1}e^{-{x\over \theta}}$.
\end{enumerate}
\end{Prop}

\me Of course, a challenge is to consider a mixed model with  demographic stochasticity and random environment, consisting in adding the term $\sqrt{Y_t}dB_t$ to the r.h.s. of \eqref{rand-env}. Some work has been developed in this sense in $\cite{BH}$ in the case without interaction $c=0$. Modeling branching processes in random environment will  be the aim of  Section \ref{BFDrc}.

\section{Continuous State Branching Processes}
\label{CSBPP}
%\marginpar{Premier jet pour transition entre nos deux parties}

\me In this part, we consider a new class of stochastic differential equations for monotype populations, taking into account exceptional events where an individual has a large number of offspring. We  generalize the Feller equation \eqref{Feller} obtained  in Subsection 3.2 by adding jumps  whose rates rates are proportional to the  population size.
The jumps are driven  by a Poisson point measure, as already done  in Subsection 2.4.
 This class of processes satisfies the branching property: the individuals of the underlying population evolve independently. Combining this property with the
 tools developed in the first part, we describe finely the processes, their  long time behavior and  the scaling limits they come from. \\

\subsection{Definition and examples}
\begin{Def}
\label{CSBP}
Let $r\in \R$, $\gamma\geq 0$ and $\mu$ be a $\sigma$-finite measure on $(0,\infty)$ such that $\int_{0}^{\infty}\big({h}\land {h}^2\big)\mu(\ud {h})$ is finite. 
Let $N_0(\ud s, \ud h, \ud u)$ be a  Poisson point  measure on $\R_{+}^3$ with intensity $\ud s\mu(\ud h)\ud u$  and  $\widetilde{N}_0$ its   compensated measure. Let 
$B$ be a standard Brownian motion  independent of $N_{0}$ and $Z_{0}$ an integrable non-negative random variable independent of $N_{0}$ and $B$. \\
The Continuous State Branching process (CSBP) $Z$ associated with $r$, $B$, $N_0$ and $Z_0$  is the unique  non-negative strong solution in $\mathbb D(\R_+,\R_+)$ of the following   stochastic differential equation  
\begin{equation}\label{defCSBP}
Z_t=Z_0+\int_0^t r Z_s \ud s +\int_0^t\sqrt{2\gamma Z_s} \ud B_s+\int_0^t\int_0^\infty\int_0^\infty \One_{\{u\leq Z_{s-}\}} h\,\widetilde{N}_0(\ud s, \ud h, \ud u).
\end{equation}
The triplet $(r, \gamma, \mu)$ is the characteristic triplet of the CSBP $Z$ and identifies its law.
\end{Def}
\noindent The difficulties in the proof of the existence and uniqueness come from the term $\sqrt{Z_t}$ (which is  non-Lipschitz for $Z$ close to $0$). 
We refer to    Fu and Li \cite{RePEc:eee:spapps:v:120:y:2010:i:3:p:306-330} for a  general framework on the existence and uniqueness of such equations. 
In particular, the definition can be extended to any initial non-negative random variable $X_0$. The fact that $Z_0$ has   a finite first moment and 
that $\int_{0}^{\infty}\big({h}\land {h}^2\big)\mu(\ud {h})$ is finite allows to focus here on conservative CSBP, i.e. for any $t\geq 0$,   $\E(Z_t)$ is finite.
We refer to the last part of this section for non-conservative CSBP : these latter  may blow up in finite time.  \\
In the forthcoming  Proposition \ref{lamperti} and Section  \ref{scaling}, the Lamperti representation  and scaling limits of discrete branching processes will be proved to actually provide   alternative ways to construct and identify CSBP. \\
We recall from the previous section that the term $\sqrt{2\gamma Z_s} \ud B_s$ corresponds to continuous fluctuations of the population size, with variance proportional to the number of individuals. The last term describes the
jumps of the population  size whose rate at time $s$ is proportional to the size $Z_{s-}$ and the distribution proportional to the measure $\mu$.
The jump term appears in the scaling 
limit when the individuals  reproduction law charges very large  numbers (the second moment has to be infinite, see Section  \ref{scaling}).  The case
 $\mu(\ud z)=c  z^{-(1+\alpha)}\ud z$ $(\alpha \in (1,2))$ plays a particular role (see below), since the corresponding CSBP is then a stable process.  We stress that in the definition given above the jumps appear only through the compensated Poisson measure, which uses the integrability assumption on $\mu$. Thus, the drift term $rZ_s$ can be seen as the sum of the drift term of a Feller diffusion and the drift term due to the mean effect of the jumps. 
%
%\marginpar{  $Z$ is a (non-continuous) semi-martingale, 
%the local martingale part is a martingale, thanks to integrability assumpt.} %In the equation above (and to fit with the cours prÃƒÂ©cÃƒÂ©dant), the big jumps have also been compensated, so that we read directly the semi-martingale decomposition. 

%\marginpar{On travaille avec les conditions habituelles, en particulier filtration continue ÃƒÂ  droite}

%The fact that these processes are stable and that they are the only ones among CSBP is not obvious. It can be proved using the tools developed in the next Section (one can first check that the branching mechanism is of the form $\lambda^{\alpha+1}$ and use the Lamperti representation). It is left as a (difficult) exercise.

%\begin{exo}
%One can try to check that such processes are stables. It can be achieved using the associated generator and martingale problem, as developped in the next Section.
%To check that they are only possible stable process, one can use the forthcoming Lamperti representation and express the .
%\end{exo}

\subsection{Characterization and properties}
\label{CP}

\me Let $Z$ be a CSBP with characteristic triplet $(r, \gamma, \mu)$.
It's a Markov process.
\begin{Prop} \label{MP}  The infinitesimal generator of $Z$
is given by: for every $f\in C^2_b(\R_+)$, 
\Bea
%\frac{\partial \E_x(f(Z_t))}{\partial t}\vert t=0
\mathcal{A}f(z) &=&r zf^\prime(z)+\gamma zf^{\prime\prime}(z)+\int_0^\infty \Big(f(z+h)-f(z)-hf^\prime(z)\Big)z\mu(\ud h).
%\\ &&\qquad :=.
\Eea
\end{Prop}

\noindent We refer to the Appendix for complements on the semimartingale decomposition. We note from  the expression of $\mathcal A$ that the function
$z\rightarrow z^2$ doesn't belong  (in general) to the domain of the generator. But one can prove that $C^2$ functions with two first derivative bounded
 belong to the domain. It can be achieved by monotone  convergence using   non-decreasing sequences of functions $g_n \in C^2_b$  such  that
$\parallel g_n' \parallel_{\infty} + \parallel g_n'' \parallel_{\infty}$ is bounded and   there exists $C>0$ such that
$$\vert \mathcal Af (z) \vert \leq C\left(\parallel f' \parallel_{\infty} + \parallel f''\parallel_{\infty}\right)z$$ 
for  every $z\geq 0$. 
%Indeed $\int_1^\infty \vert f(z+h)-f(z) \vert \mu(\ud h)\leq \parallel f'\parallel_{\infty} \int_1^\infty h  \mu(\ud h)$ and 
%$$\int_0^1\vert f(z+h)-f(z)-hf^\prime(z)\vert \mu(\ud h)\leq \sup_{y\in [z,z+1]} \vert f''(y) \vert \int_0^1 h^2\mu(\ud h).$$
%let us prove that  the identity does and compute the mean growth of the process.
% The integrability assumption in Definition \ref{CSBP} ensures that $\mathcal A f$ is well defined 
%for any $f \in  C^2(\R_+)$ whose two first derivatives are bounded and there exists $C>0$ such that 
%Moreover,we check in the proof that the local martingal part of $Z$ is martingale (bounded in $L^2$)??????.
% and $\int_1^\infty \vert zf^\prime(x) \vert \mu(\ud z)<\infty$, whereas
%Then by monotone approximations of the function $x\rightarrow x$ with functions $g_n\in C^2_b$ such that , the identity
% \be
%\label{margin}
%\E_{z}(g_n(Z_t))=\int_0^t \E_{z}\left(\mathcal Ag_n(Z_s)\right)\ud s
%\ee
%ensures that  $\E_{z}(Z_t) \leq z\exp(Gt)$ for some constant $G>0$. By bounded convergence theorem,
%we then obtain similarly that $(\ref{margin})$ hold for $f$
% with two first derivatives bounded and   $s\rightarrow \E_{z}\left(\mathcal Af(Z_s)\right)$ is continuous in $0$. This ensures that  $t\rightarrow  \E_{z}(f(Z_t))$ can be differentiated in $0$ and that the derivative equals $\mathcal Af(z)$.

\begin{exo} Prove that $\E_z(Z_t)=z\exp(rt)$ for any $t,z\geq 0$ and  that $(\exp(-rt)Z_t :  t\geq 0)$ is  a  martingale. What can you say about  the long time behavior of $Z_t$? What is the interpretation of $r$ ?\\
\emph{One can use the  two-dimensional It\^o formula.} \\%
\end{exo}

\me We give now the key property satisfied by our class of processes. If necessary, we denote by $Z^{(z)}$ a CSBP starting at $z$. 
\begin{Prop}
\label{branch}
 The process $Z$ satisfies the branching property, i.e. $$Z^{(z+\w{z})}\stackrel{d}{=} Z^{(z)}+\w{Z}^{(\w{z})}  \qquad (z,\w{z} \in \R_+),$$
where $Z$ et $\w{Z}$ are independent CSBP's with the same distribution.\\
Then the Laplace transform of $Z_t$ is of the form
\begin{equation*}
\mathbb{E}_z\Big[ \exp(-\lambda  Z_t)\Big]=\exp\{-z
u_t(\lambda)\},\qquad\textrm{with }\lambda\geq 0,
\label{CBut}
\end{equation*}
for some non-negative function $u_t$ and any $z\geq 0$. 
\end{Prop}
%\marginpar{TO  check and improve}
\begin{proof} To simplify the notation, we write $X_t=Z_t+Z'_t$, where $Z_0=z$, $Z_0'=\w{z}$ and  $Z$ and $Z'$ are two  independent CSBP's. The process $X$
 satisfies a.s. : 
\bea
\label{deccc}
X_t&=&X_0+\int_0^t rX_s \ud s +\int_0^t\sqrt{2\gamma Z_s} \ud B_s+\sqrt{2\gamma Z'_s} \ud B'_s \\
&& \qquad +\int_0^t\int_0^\infty\int_0^{Z_{s-}} h\widetilde{N}_0(\ud s, \ud h, \ud u)+\int_0^t\int_0^\infty \int_0^{Z'_{s-}} h\widetilde{N}'_0(\ud s, \ud h, \ud u) \nonumber
%&=& X_0+\int_0^t rX_s \ud s +\int_0^t\sqrt{2\gamma X_s}   \ud B''_s+\int_0^t\int_0^\infty\int_0^{X_{s-}} h\widetilde{N}''_0(\ud s, \ud h, \ud u)
\eea
where $B$ and $B'$ are two independent Brownian motions and $N_0$ and $N'_0$ are two independent Poisson point  measures on $\R_{+}^3$ with intensity $\ud s\mu(\ud h)\ud u$.
We introduce  the real valued process $B''$  defined by
$$B''_t=\int_0^t \ind_{Z_s >0}\frac{\sqrt{2\gamma Z_s} \ud B_s+\sqrt{2\gamma Z'_s} \ud B'_s}{\sqrt{2\gamma X_s}}+\int_0^t \ind_{Z_s =0}\ud B_s$$
and note that $B''$ is a Brownian motion by L\'evy Theorem since it is a continuous local martingale with quadratic variation equal to $t$.
We also define  the random point measure $N_0''$ on $\R_+^3$  by
$$N''_0(\ud s, \ud h, \ud u)= N_0(\ud s, \ud h, \ud u) \One_{\{u \leq Z_{s-}\}} + \hat{N'_0}(\ud s, \ud h, \ud u),$$
where $\hat{N'_0}$ is the random point measure on $\R_{+}^3$ given by
$\hat{N'_0}(\ud s,  \ud h , [u_1,u_2])=N'_0(\ud s, \ud h , [u_1-Z_{s-},u_2-Z_{s-}]).$
The random measure $N''_0$ is also a  Poisson point measure with intensity $\ud s\mu(\ud h)\ud u$ since $N_0$ and $N_0'$ are independent. 
Adding that (\ref{deccc}) can be rewritten as
$$X_t= X_0+\int_0^t rX_s \ud s +\int_0^t\sqrt{2\gamma X_s}   \ud B''_s+\int_0^t\int_0^\infty\int_0^{X_{s-}} h\widetilde{N}''_0(\ud s, \ud h, \ud u),$$
the process $X$ is a CSBP  with initial condition $z+\w{z}$. \\
Furthermore, the branching property ensures that for $\lambda>0$,
$$\mathbb{E}_{z+z'}\Big[ \exp(-\lambda  Z_t)\Big]=\mathbb{E}_z\Big[ \exp(-\lambda  Z_t)\Big]\mathbb{E}_{z'}\Big[ \exp(-\lambda  Z_t)\Big]$$
which yields the linearity of the Laplace exponent (taking the logarithm).
\end{proof}

\noindent Combining Propositions \ref{MP} and \ref{branch}, we characterize the finite dimensional distributions of a CSBP.
\begin{Cor}[Silverstein \cite{MR0226734}]
\label{expo}
Let $\lambda>0$. The Laplace exponent $u_t(\lambda)$ is the unique solution of
\be
\label{eq-branch}
\frac{\partial}{\partial t} u_t(\lambda)=-\psi(u_t(\lambda)), \qquad u_0(\lambda)=\lambda,\ee
where $\psi$ is called the  branching mechanism associated with $Z$ and is defined by
\begin{equation}\label{defpsi}\psi(\lambda) %:=\log \E(\exp(-\lambda Y_1))
= -r \lambda + \gamma \lambda^2 + \int_0^{\infty} \left( e^{-\lambda {h}}-1+\lambda {h}
\right) \mu(\ud {h}).\end{equation}
%the  L\'evy-Khintchine formula ensures that   
\end{Cor}
\begin{proof} %\marginpar{also possible via Lamperti ?}
Applying Propositions \ref{MP} and \ref{branch} and defining  $z\rightarrow f_{\lambda}(z):= \exp(-\lambda z) \in  C^2_b(\R_+)$, we  get  $P_t f_{\lambda}(z)=\exp(-zu_t(\lambda))$ and $$\frac{\partial}{\partial t} P_t f_{\lambda}(1)=\mathcal A P_tf_\lambda (1)=-\frac{\partial  u_t(\lambda)}{\partial t}\exp(-u_t(\lambda)).$$
Thus computing  the generator for the function $z\rightarrow \exp(-zu_t(\lambda))$ yields the result. 
\end{proof}

An alternative proof  of this result can be given by  using It\^o's formula to prove 
that $(\exp(-v_{T-t}(\lambda)Z_t) :  t\in[0,T])$ is a martingale if and only if  
$v$ is the solution of (\ref{eq-branch}). This idea will be extended to the random environment in the next section.
%. \emph{Indication : use  It\^o's formula}. \\
%2) Give an alternative proof of the previous corollary.
%\end{exo}
\begin{exo}
1) Check that  for any $\lambda>0$ and $t>0$,  $u_t(\lambda)$ is the unique solution of the
integral equation
\begin{equation*}
\int_{u_t(\lambda)}^\lambda \frac{1}{\psi(v)}\,{d} v=t.
\label{DEut}
\end{equation*}
2) Compute  $u_t(\lambda)$ for a  Feller diffusion   and deduce from it the extinction probability.
\end{exo}

\subsection{The Lamperti transform}

\me The following result is fundamental for the trajectorial and long time study of CSBP, since it allows to see these processes  (whose dynamics are  multiplicative) as the time change of some L\'evy processes (which are well known additive random processes). We recall that a L\'evy process is a c\`adl\`ag process with stationary independent increments.
\begin{Prop}[Lamperti \cite{MR0208685,MR0217893}] \label{lamperti} Let $Y$ be a L\'evy process defined by
$$Y_t:=y_0+r t+\gamma B_t+\int_0^t\int_0^{\infty} h \w{N}(\ud{s},\ud{h}),$$
where $r \in \R, \ \gamma \geq 0$,  $N$ is a Poisson point measure on $\R_+^2$ with intensity $\ud{s}\mu(\ud{h})$, $\w{N}$ its compensated measure and
% $\mu$ is $\sigma$-finite measure on such that
 $\int_{0}^{\infty}\big({h}\land {h}^2\big)\mu(\ud {h})<\infty$. 
Writing $Y^+$ for the process $Y$ killed when it reaches $0$,  the equation
\be
\label{lamprepre}
Z_t=Y^+_{\int_0^t Z_s\ud{s}}
\ee
has  a unique solution $Z$ for $t\geq 0$. This process $Z$ is c\`adl\`ag and  distributed as the CSBP
 with characteristic $(r,\gamma, \mu)$ started at $y_0$.
\end{Prop}
\noindent % For details on the integrability assumptions for the jump measure $\mu$ of the L\'evy process, we refer
%to the Appendix. 
In particular, when  $\mu(\ud z)=c  z^{-(1+\alpha)}\ud z$ ($r=0,\gamma=0$), we recover the stable L\'evy processes. 
A converse statement is given in the last part of this section. It relies
on the expression $Y_t=Z_{\gamma_t}$ with $\gamma_t=\inf\{ u : \int_0^u  \beta(Z_s)\ud{s} >t\}$.

\me
To prove  Proposition \ref{lamperti}, we use the following lemma, which we derive from Chapter 6 (random time changes) of Ethier \& Kurtz \cite{EK}.

\begin{Lem} \label{chgtimeEK1} Let $X$ be a c\`adl\`ag process from $\R_+$ to $\R_+$ and $\beta$ be a nonnegative 
continuous function on $\R_+$. We define $A_t:=\int_0^t 1/\beta(X_u) \ud{u}$ and assume that
$$\lim_{t\rightarrow \infty} A_t=+\infty, \qquad T:=\inf\{ t\geq 0 : A_t=+\infty \}=\inf\{ t\geq 0 :  \beta(X_t)=0 \} \quad \text{a.s}.$$

(i) There exists a unique function $\tau$ from $\R_+$ to $[0,T)$ which is solution  of the equation $A_{\tau_{t}}=t$.

(ii) The process $Z$ defined by $Z_t:=X_{\tau_t}$ for $t\geq 0$ is the unique solution of  $Z_t=X_{\int_0^t \beta(Z_s)\ud{s}}$ for $t\geq 0$.

(iii-Martingale problem)  If $(f(X_t)-\int_0^tg(X_s)\ud{s} : t\geq 0)$ is an $(\mathcal F^{X}_{t})_t$ martingale, then  $(f(Z_t)-\int_0^t\beta(Z_s)g(Z_s)\ud{s} :  t\geq 0)$ is an $(\mathcal F^{Z}_{t})_t$ martingale.
\end{Lem}
\begin{proof}
$(i)$  simply comes from the fact that $t \rightarrow A_t$ is  an increasing bijection from $[0,T)$ to $\R_+$. \\%, with $T:=\inf\{ t\geq 0 :  \beta(X_t)=0 \}$. \\
%\marginpar{derivee p.p. ? au point de continuite}
$(ii)$ is deduced from $(i)$ by noticing that 
\be
\label{eqpp}
A_{\tau_{t}}=t \  \ (t\geq 0) \ \  \Longleftrightarrow  \ \  \tau'_t=\beta(Z_t)\ \text{a.e.} \ \Longleftrightarrow \  \ \tau_t=\int_0^t\beta(Z_s)\ud{s}  \ \  (t\geq 0) \ .
\ee
 (Take care of the regularity of the processes). \\
%\marginpar{we recall $\mathcal F^{.}_{t+}=F^{.}_{t}$}
To prove $(iii)$, we first check that $\{\tau_s \leq t\}=\{A_t\geq s\}\in \mathcal F^{X}_{t}$. The optional sampling theorem ensures that if
$\ f(X_t)-\int_0^tg(X_s)\ud{s}$ is an $(\mathcal F^{X}_t)_t$ martingale, then $f(X_{\tau_t})-\int_0^{\tau_t}g(X_s)\ud{s}=f(Z_t)-\int_0^t
g(Z_s)\tau_s'\ud{s}\ $ is an $(\mathcal F^{Y}_{\tau(t)})_t$
martingale. Recalling from $(\ref{eqpp})$ that $\tau_s'=\beta(X_{\tau_s})=\beta(Z_s)$ a.e, we get the result. 
\end{proof}
\begin{proof}[Proof of Proposition \ref{lamperti}]
The existence and uniqueness of the problem (\ref{lamprepre}) come from Lemma \ref{chgtimeEK1}  $(i)$ and $(ii)$ with $X=Y^+$ and $\beta(x)=x$. Indeed,
 we first note that $\E(Y_1 ) \in (-\infty,\infty)$ and  the following law of large numbers holds: $Y_t/t\rightarrow \E(Y_1)$  a.s. 
 %\marginpar{Ref bertoin ? Cor 2 p190} 
 Then
  $\int_0^{\infty}1/Y_s^+\ud{s} =\infty$  a.s. Let us now check that the first time at which $A_t$ is   infinite is the first time at which $Y^+$   reaches $0$. The fact that
 $\inf\{ t\geq 0 :  Y_t^+=0 \} \leq  \inf\{ t\geq 0 : A_t=\infty \}$ is obvious. To get the converse inequality, we denote by $T$  the non-decreasing limit of the stopping times 
$T_{\epsilon}=\inf\{ t \geq 0 : Y_{t}\leq \epsilon\}$ for $\epsilon \rightarrow 0$ and prove  that
 $Y_{T}=0$ on the event $\{T <\infty\}$ (quasi-left continuity). For that purpose, we  use  
% the strong Markov property on $Y$ to get that $(Y_{s\wedge T_n\wedge t}-  Y_{T_n\wedge t} : s\geq 0)$ 
% is a also a LÃƒÂ©vy process (with the same distribution), which is independent of $\mathcal{F}_{T_n \wedge t}$. 
%\marginpar{Note def quasileft continuity, and be careful with local martingale/sotpp time/ $\mathcal C^2_b$}
$$\E\left(f(Y_{T\wedge t})-f(Y_{T_{\epsilon}\wedge t}) -\int_{T_{\epsilon}\wedge t}^{T \wedge t} Qf(Y_s) ds  \big\vert \mathcal F_{T_{\epsilon}\wedge t} \right)= 0.$$
for  $f \in C^2_b(\R_+)$, where  we have denoted by  $Q$ the generator of $Y$:
$$Qf(y):=rf'(y)+\gamma f''(y)+ \int_0^\infty \Big(f(y+h)-f(y)-hf^\prime(y)\Big)\mu(\ud h).$$
Then a.s. on the event $\{T \leq t\}$, we have
$$\lim_{\epsilon \rightarrow 0} \E(f(Y_{T\wedge t}) \vert \mathcal F_{T_{\epsilon}\wedge t} )= f(0),$$
and using some non-negative function  $f\in C^2_b(\R_+)$  which coincides with $x^2$ in  a neighborhood of $0$, we obtain $Y_{T}=0$ on $\{ T <\infty\}$.

\noindent To check that the process given by (\ref{lamprepre}) is indeed a CSBP, we use again the generator $Q$ of $Y$. Let us first note that the generator of $Y^+$ is
given by $x\rightarrow Qf(x)\ind_{x>0}$ for $f\in C^2_b(\R_+)$. Then
%compute the generator of $Qf=gf'+\gammaf''+ \int_0^\infty \Big(f(x+z)-f(x)-zf^\prime(x)\Big)x\mu(dz)$ of $Y$, so that
  Lemma \ref{chgtimeEK1}   $(iii)$ ensures that for $f \in C^2_b(\R_+)$,
$$f(Z_t)-\int_0^t Z_sQf(Z_s)\ud{s}$$ 
is a  martingale. 
It identifies  the distribution of the c\`adl\`ag Markov process $Z$ via its generator  $\mathcal A f(z)=zQf(z)$. More precisely, the uniqueness of the martingale problem is required here to ensure the uniqueness of the distribution of the process and we refer to Ethier \& Kurtz \cite{EK}, Theorem 4.1 in Section 4 for a general statement. In our particular case, the proof can be made directly using the  set
 of functions $f_{\lambda}(z)=\exp(-\lambda z)$. Indeed,   the independence and stationarity of the increments of $Y$ ensure the branching property of $Z$. One can  then follow the proof of 
Corollary \ref{expo} to derive   $\E_z(\exp(-\lambda Z_t))$ from $\mathcal A$ and identify the finite dimensional distributions of $Z$. 
\end{proof}

%\begin{exo} We say that a process is stable when
%there exists $\alpha>0$ such that  for all $z,y>0$, the law of
%\marginpar{Check !!}
%$(yZ_{y^{-\alpha}t} : t\geq 0)$ when $Z$ is started from $z$ is equal to the law of $(Z_t : t\geq 0)$ started from $yz$.  \\
%Use the Lamperti transform and the Laplace exponent to identify the whole class of stable CSBP.
% \end{exo}

\subsection{Long time behavior}

\me In this section, the characteristic triplet $(r,\gamma, \mu)$ is assumed to be non identical to $0$, to avoid the degenerate case where  $Z$ is  a.s. constant.
\begin{Prop} 
\label{fel-beh}
(i-unstability) With probability one,  $Z_t$ tends to  $0$ or to $\infty$ as $t\rightarrow \infty$. \\
(ii-extinction probability) Denoting by $\eta$  the largest  root of $\psi$, we have $$\P_z(\lim_{t\rightarrow \infty} Z_t=0)=\exp(-z\eta) \qquad (z\geq 0).$$ In particular,  
extinction occurs a.s. if and only if 
$r=-\psi'(0)\leq 0$. \\
(iii-absorption probability)  $\P( \exists t >0 : Z_t=0)>0$ if and only if
$\int^{\infty} 1/\psi(x) \ud{x} <\infty.$
\end{Prop}

\me As an illustration, which is left as an exercise, check that the CSBPs with  characteristics $(r,0,0)$ and $(0,0,x^{-2}1_{[0,1]}(x)dx)$  have positive extinction probability 
but null absorption probability.
For stable CSBPs (including the  Feller diffusion), extinction and absorption coincide.

\begin{proof} $(i)$ is a consequence of the Lamperti representation given in  Proposition \ref{lamperti}. Indeed
a non-degenerate L\'evy process $Y$ either goes to $+\infty$, $-\infty$ or oscillates and we stress that  $Y$ is killed at $0$ in the Lamperti transform.  So
$$\left\{\int_0^{\infty} Z_s \ud{s}=\infty\right\} \subset \left\{ Z_t \stackrel{t\rightarrow \infty}{\longrightarrow} 0\right\} \cup 
\left\{ Z_t \stackrel{t\rightarrow \infty}{\longrightarrow} \infty\right\}.$$
%The second inclusion comes from the Lamperti representation, 
Adding that 
$$\left\{\int_0^{\infty} Z_s \ud{s}<\infty\right\} \subset \left\{Z_t\stackrel{t\rightarrow \infty}{\longrightarrow} Y^+_{\int_0^{\infty} Z_s\ud{s}}  \ \emph{ and }  \ \int_0^{\infty} Z_s \ud{s}<\infty  \right\} \subset  \left\{Z_t \stackrel{t\rightarrow \infty}{\longrightarrow} 0\right\}$$
ends up the proof. 
%nt we define a sequence of  stopping times $T_k \in [0, \infty]$ by  
%$$T_0=0, \qquad T_{k}=\inf \{ t \geq T_{k-1}+1 : Z_{T_k} \geq \epsilon\} \quad (k\geq 1, \quad \inf \varnothing =\infty).$$ % when $T_k <\infty$; and $T_k=T_{k+1}= \ldots= \infty$ otherwise. 
%We use that  $\int_0^{\infty} Z_s \ud{s}\geq \sum_{i=1}^k \int_{T_{i}}^{T_{i}+1} Z_s \ud{s}$ and recall the monotonicity of $Z$ with respect to the initial value, which can be seen by the definition or the branching property. Then for each $A$, we obtain by successive conditionning
%$$\P\left(\int_0^{\infty} Z_s \ud{s} \leq A; T_k <\infty\right)\leq \P\left(\sum_{i=1}^{k} X_i \leq A\right),$$
%where $X_i$ are i.i.d. random variable distributed as   $\int_{0}^{1} Z_s^{(\epsilon)} \ud{s}$. Letting $k$ and then $A$ go to $\infty$ ensures that
%$$\P\left(\int_0^{\infty} Z_s \ud{s} <\infty,\liminf_{t\rightarrow \infty} Z_t>0\right) = \lim_{\epsilon \rightarrow 0}\P\left(\int_0^{\infty} Z_s \ud{s} <\infty;  T_k <\infty  \ \text{for} \ k=1, 2,... \right)=0.$$

\me Concerning  the extinction $(ii)$, we first use $(i)$ to write $\exp(- Z_t)\rightarrow \One_{\{\lim_{t\rightarrow \infty} Z_t=0\}}$ as $t\rightarrow \infty$. By the bounded convergence theorem, 
$\E_z(\exp(-Z_t))\rightarrow \P_z(\lim_{t\rightarrow \infty} Z_t=0)$. Moreover,  $\E_z(\exp(-Z_t))=\exp(-zu_t(1))$ by Proposition \ref{branch}.
Noting that the branching mechanism $\psi$ is convex (and non trivial), it is positive for $z>\eta$ and negative for $0<z<\eta$. Thus, $u_t(1)\rightarrow \eta$ as
$t\rightarrow \infty$ and  $\P_z(\lim_{t\rightarrow \infty} Z_t=0)=\exp(-z\eta)$.

\noindent Let us finally deal with the absorption $(iii)$. We note  that $\P_z(Z_t=0)=\lim_{\lambda \rightarrow \infty} \exp(-zu_t(\lambda))=\exp(-zu_{t}(\infty))$
and recall   from Proposition \ref{expo} that
$$\int^{\lambda}_{u_t(\lambda)} \frac{1}{\psi(u)} \ud{u}=t.$$
If $u_t(\lambda)$ is bounded for $\lambda >0$ (with some fixed $t$), then $\int^{\infty} 1/\psi<\infty$ (by letting $\lambda\rightarrow \infty$).
Conversely  the fact that $u_t(\infty)<\infty$ is bounded for $t\geq 0$ forces $ \int^{\infty} 1/\psi(x) \ud{x}=+\infty$ (by letting $\lambda$ and then $t$ go to $\infty$). 
\end{proof}

\subsection{Scaling limits}
\label{scaling}

\me In this section, we obtain the CSBP as a scaling limit of Galton-Watson processes.
We recall that a Galton-Watson process $X$ with reproduction law $\nu$ is defined by
$$X_{n+1}=\sum_{i=1}^{X_n} L_{i,n},$$
where $(L_{i,n} : i\geq 1, n\geq 0)$ are i.i.d random variables with common distribution $\nu$.\\
Let us associate a random walk to this process, denoted by  $S$. It is obtained by  summing  the number of offspring of each individual of the Galton-Watson tree as follows : 
$$S_0:=Z_0, \quad S_{k+1}:=S_k+  L_{k-A_n+1,n}-1$$
for each $k \in [A_n, A_{n+1})$ and $n\geq 0,$
where $A_n:=\sum_{j=0}^{n-1} X_j$. Thus the increments of the random walk $S$ are distributed as $\nu$ shifted by $-1$.
This random  walk $S$ satisfies $S_{A_{n+1}}-S_{A_n}=\sum_{i=1}^{X_n} (L_{i,n}-1)=X_{n+1}-X_n$, so that
\be
\label{dlt}
X_n=S_{A_n}=S_{\sum_{i=0}^{n-1} X_i}, 
\ee
which yields the discrete version of the  Lamperti time change. It both 
 enlightens the Lamperti transform in the continuous setting (Proposition \ref{lamprepre})
and allows us to prove the following scaling limit.

\begin{Thm} Let $X^K$ be a sequence of Galton-Watson processes with reproduction law $\nu_K$ and $[Kx]$ initial individuals. We consider the scaled 
process
$$Z^K_t=\frac{1}{K}X_{[v_Kt]}^K \qquad (t\geq 0),$$
where $(v_K)_{K}$ is a sequence tending to infinity. Denoting by $S^K$ the random walk associated with $Z^K$, we assume that
$$\frac{1}{K}S^K_{[Kv_K.]} \Rightarrow Y,$$
where $Y$ is a L\'evy process.
Then $Z^K \Rightarrow Z$, where $Z$ is the CSBP characterized by (\ref{lamprepre}).
\end{Thm}

\me The Feller diffusion case  $(\mu=0$) is the only possible limit of Galton-Watson processes with bounded variance (see \cite{MR0362529}). It comes from the convergence of $(S_{[K^2t]}/K : t\geq 0)$  to a Brownian motion under second moment assumption. More generally, the stable case with drift  $\psi(\lambda)=-r\lambda+c\lambda^{\alpha+1}$ ($\alpha$ in $(0,1]$)   corresponds to the class of CSBPs which can be   
obtained by scaling limits of Galton-Watson processes with a fixed   reproduction law (i.e. $\nu_K=\nu$). \\ %\marginpar{Ref cours sylvie}
%We also recall from Section \ref{LPA} that the Feller diffusion  is the limit of accelerated birth and deaths processes.  \\
%Namely, we consider a birth and death process $(Z^{N]}_t : t\geq 0)$ starting from $[xN]$ individuals, with birth rate $\lambda_n^N=n(N\gamma+\lambda)$ and death rate $\mu_n^N=n(N\gamma+\mu)$. Dividing this process by $N$ and letting $N\rightarrow \infty$, it converges weakly  to  branching Feller diffusion with growth rate $r=b-d$ and diffusion term $\sigma$. \\
Several  proofs of this theorem can  be  found. One can use a tightness argument and  identify  the limit thanks to the Laplace exponent. Such a proof is in the same vein as the previous section and we refer
to \cite{MR0362529} for details.
As mentioned above, the proof can also be achieved using 
discrete Lamperti transform (\ref{dlt}) with an argument of continuity.   This argument  can  be adapted
from Theorem 1.5 chapter 6 in Ethier-Kurtz \cite{EK} :
%\marginpar{Delete ?}
\begin{Lem} Let $Y$ be a L\'evy process killed at $0$ and $\beta$ a continuous function.
Assume that $Y^K \Rightarrow Y$, where $Y^K$  is a c\`adl\`ag process from $\R_+$ to $\R_+$ and define the process $Z^K$ as the solution of 
$Z^K_t=Y^K_{\int_0^t \beta(Z^K_s)\ud{s}}$. Then  $Z^K \Rightarrow Z$.
\end{Lem}

%\me Finally, we mention that the scaling limit can be stated in terms of characteristic triplet, see  \cite{bansim} with an extension to the varying environment case. 

%Actually, it is enough  to assume that $(Y^N_0,Y^N_1)\Rightarrow (Y_0,Y_1)$ (see e.g. Th 1.4 chap. 9 in \cite{EK} or the books of Jacod Shirayev or Kallenberg).
%\marginpar{A reprendre}

%The proof of the Lemma above is left to reader. One should
%begin with $\beta$ bounded (truncation argument) and use Skorokhod representation to get the a.s. convergence of $Y_K$ to $Y$.
%Noting that $\{Z^K_s : s\leq T\}\subset\{Y^K_u : u\leq T.\sup\beta\}$ ensures the tightness. 
%We recall from the proof of Proposition \ref{lamperti} that  $\lim_{\epsilon \rightarrow 0}\inf\{s : \beta(Y_s)<\epsilon\}=T$, so that we can apply Theorem 1.5 in Ethier Kurtz \cite{EK}.

\subsection{On the  general case.}

\me What is the whole class of branching processes in the continuous setting?
What is the whole class of scaling limits of Galton-Watson processes?
These two classes actually coincide and extend the class of  CSBPs with finite first moment (called conservative CSBPs) we have considered above.

\begin{Thm} \cite{MR0217893, CLU}
The c\`adl\`ag Markov processes $Z$ which take values in $[0,\infty]$ and satisfy the branching property are  in one to one correspondence with  L\'evy  processes $Y$ with no negative jumps,  through the equation 
$$Z_t=Y^+_{\int_0^t Z_s}.$$
\end{Thm}

\me Such a process $Z$  is still characterized by  a triplet $(r,\gamma, \mu)$, with $r \in \R, \gamma \in \R_+$ but the measure $\mu$ on $\R_+$ only satisfies  that $$\int_{0}^{\infty}\big({1}\land {z}^2\big)\mu(\ud {z})<\infty.$$ More specifically, the 
Laplace exponent $u_t$ of $Z$, which uniquely determines the finite dimensional distributions,  is the unique solution of
$$\frac{\partial u_t(\lambda)}{\partial t}=-\psi(u_t(\lambda)), \qquad u_0(\lambda)=\lambda,$$
where 
$$\psi(\lambda):=-r \lambda +\gamma \lambda^2+\int_0^{\infty} (e^{-\lambda h}-1+\lambda h1_{h\leq1}) \mu(\ud{h}).$$
More generally, the results given above can be extended.
Thus, the expression of the generator $\mathcal A$ remains valid and  $Z$ is given by the following SDE  (see Proposition 4 in \cite{CLU})
%\marginpar{et le generator ?}
\Bea
\label{defCSBP}
Z_t&=&Z_0+\int_0^t r Z_s \ud s +\int_0^t\sqrt{2\gamma Z_s} \ud B_s+\int_0^t\int_0^1\int_0^{Z_{s-}} h\widetilde{N}_0(\ud s, \ud h, \ud u) \\
&& \qquad \qquad \qquad \qquad \qquad \qquad \qquad+\int_0^t\int_1^\infty\int_0^{Z_{s-}} hN_0(\ud s, \ud h, \ud u),
\Eea
where $B$ is a standard Brownian motion, $N_0(\ud s, \ud z, \ud u)$ is a Poisson random measure on $\R_+^3$ with intensity $\ud s\mu(\ud z)\ud u$ independent of $B$,  and  $\widetilde{N}_0$ is the  compensated measure of $N_0$. \\

\me We stress that the class of CSBPs  obtained now yields the (only possible) scaling limits of Galton-Watson processes (or more generally
  discrete space continuous time branching processes) \cite{MR0208685,MR0217893, CLU}. \\
  %The stable CSBP are the only ones which can be obtained by a fixed reproduction law of the original Galton Watson process. The Feller branching diffusion are the only continuous one. 
%Similarly, one can  approach the CSBP by discrete space continuous time branching processes, see \cite{CLU}.
A new phenomenon appears in the non-conservative case : the process may explode in finite time. 

\begin{Prop} The CSBP $Z$ blows up with positive probability, which means that $\P_1(Z_t =\infty)>0$ for some $t\geq 0$, if and only if
$$\psi'(0+)=-\infty \quad \emph{and} \quad \int_0 \frac{\ud s}{\psi(s)}>-\infty.$$
\end{Prop}

\me In this section, we have  focused on the size of the population $Z_t$. The scaling limits actually provide a natural notion of genealogy for the limiting object, see \cite{DL}. An other point of view, using the flow of subordinators which comes by letting the initial size vary, has been exhibited recently by Bertoin and Le Gall  \cite{MR1771663}.
Finally,  several extensions of CBSPs have been considered. In particular, for
CSBP with immigration, we refer  to \cite{RePEc:eee:spapps:v:120:y:2010:i:3:p:306-330} for the SDE characterization and to \cite{CGU} for the Lamperti transform.

\section{Feller Diffusion with Random Catastrophes}
\label{BFDrc}

\me We deal now with a new family of branching processes taking into account the effects of the environment
on the population dynamics. It may cause random fluctuations of the growth rate \cite{BH, EvansSchreiber}
or   catastrophes which kill a random fraction of the population \cite{BPS}.  \\
Here, we are focusing on a  Feller diffusion with catastrophes, in order to simplify the study and fit with the motivations given in Section \ref{div}. We aim at 
enlightening new behaviors due to the random environment and we refer to Subsection \ref{complements} for extensions and comments on (more general) CSBP's in random
environment.

\subsection{Definition and scaling limit}
\label{scaling2}

\me We consider the  Feller diffusion  \eqref{Feller} and add some random catastrophes, whose rates are given by a function $\tau$ and whose effect on the population size are multiplicative and given by some random variable $F$ taking values in $[0,1]$.
%$$Z_{t}= F_tZ_{t-},$$
%where $\{(t,F_t) : F_t\ne 1\}$ is a Poisson Point process independent of the original process. 
This process $Y$ is defined as  the solution of the following SDE :
\begin{multline*}
Y_t =
y_0 +\int_0^t rY_s ds
 +  \int_0^t \sqrt{2\gamma Y_s}dB_s
 -  \int_0^t \int_{0}^{\infty}\int_0^1
\ind_{u\leq \tau(Y_{s_-})} \big(1-\theta\big) Y_{s_-}  N_1(\ud s,\ud u,\ud \theta)
\end{multline*}
where $B$ and $N_1$ are respectively a Brownian motion  and  a Poisson point measure on  $\R_+\times \R_+\times (0,1]$ with intensity  $ds\,du\,\P(F\in  \ud \theta)$, which are independent. \\
Thus, the random variable $F$  is the  intensity of the catastrophe.   We assume that $\P(F>0)=1$   and $\P(F\in (0,1))>0$ to avoid absorption by a single catastrophe. Similarly, we also assume that 
$$\E(\log(F))>-\infty.$$
The rate $\tau$ at which the catastrophes occur may depend on the population size. We refer
to Section \ref{div} for motivations for cell division. More generally, the fact that $\tau$ is increasing is relevant when the catastrophe is actually a regulation of the population dynamics. We may think about the effect of a treatment for an infected population or invasive species or protocols for web treatment of tasks  such as TCP... 
%One may also let the growth rate $r$ and the diffusion term $\beta$ depend on time. 

\bi
Following Section  \ref{ScalingDim1}, the process can be constructed as a scaling limit, which enlightens the two time scales involved in the dynamics, namely the time scale of
the demography of the  population and that of the catastrophes. It is achieved  by considering a linear birth and death process $Y_t^K$ starting from $[Ky]$ individuals. Its individual birth and death rates are given by $\lambda  + K\gamma$ and $\mu  + K\gamma$.
 Moreover the process $Y^K$ jumps with rate $\tau(n/K)$ from $n$ to $[\theta n]$ where $\theta$ is chosen randomly following the distribution $F$. More precisely
\Bea
Y_t^K&=&[Ky_0]+\int_0^t \int_0^{\infty} \left(\ind_{u\leq Y_{s-}^K(\lambda+K\gamma)} -\ind_{ Y_{s-}^K(\lambda+K\gamma)\leq u\leq Y_{s-}^K(\lambda+K\gamma+\mu+K\gamma)}\right)N_0(ds,du)\\
&& \qquad \qquad - \int_0^t \int_0^{\infty} \ind_{u\leq \tau(Y_{s-}^K/K)} (Y_{s-}-[\theta Y_{s-}])N_1(\ud s,\ud u,\ud \theta).
\Eea
Then $(Y_t^K/K : t \geq 0)$ converges weakly to $(Y_t : t \geq 0)$ as $K\rightarrow \infty$, see  \cite{MR2754402} for more details.
%\marginpar{Ref cours ?}
 Taking the integer part of the population size   after a catastrophe is  convenient when proving the convergence of the scaling limit via tightness and limiting
 martingale  problem. Other models in discrete space  would yield the same limiting object.
For example, each individual could be killed independently with probability $F$ when a catastrophe occurs.
We also recall from the previous sections that  scaling limits of other processes, such as Galton Watson processes, lead to the  Feller diffusion.

\subsection{Long time behavior when  catastrophes occur at a  constant rate}
\label{cata}

\me In this section, we assume that $\tau(.)=\tau$ is constant and
 the successive catastrophes can be represented by   a Poisson point process  $\{(T_k,F_k)  : k \geq 0\}$     on $\R_+\times [0,1]$ with intensity $\tau\ud t  \P(F \in \theta)$. 
The long time behavior of $Y$ can be guessed  thanks to the following heuristic:
$$Y_t \approx Z_t.\Pi_{T_k \leq t} F_k$$
where $Z$ is a  Feller diffusion with parameters $(r,\gamma)$ and   all the catastrophes during the time interval $[0,t]$  have been postponed at the final time.
We recall from Section \ref{CP}  that $Z_t$ is equal to $e^{rt}M_t$ where $M$ is a  martingale.  We prove in this section (see  the two forthcoming theorems) that
$Y_t$ behaves as $\ \exp(rt)\Pi_{T_k \leq t} F_k=\exp(K_t)\,$ 
as $t$ goes to infinity, where 
$$K_t:=rt+ \sum_{T_k\leq t} \log F_k=rt +\int_0^t\int_0^1 \log(\theta) N_1(\ud s,[0,\tau], \ud \theta),$$
with $N_1(\ud s,[0,\tau], \ud \theta)$  a Poisson point measure with intensity $\tau \ud s \P(F\in \ud \theta)$
and  $K$ a L\'evy process. 
It turns out that   the asymptotic behavior of  Feller diffusions with catastrophes will  be inherited from the classification of 
long time behavior of L\'evy processes. First, we  check that 
$$\bar{Y}_t:=\exp(-K_t)Y_t$$ is a continuous local martingale, which extends the result of Section  \ref{CP} to random environment.
\begin{Lem}\label{propcorrectionsauts} The process $(\bar{Y}_t : t\geq 0)$ satisfies the SDE 
\begin{equation}\label{sde}
\bar{Y}_t=y_0+ \int_0^t e^{-K_s/2} \sqrt{2\gamma \bar{Y_s}}dB_s.
\end{equation}
\end{Lem}
\begin{proof}Since for every $t\in \R_+,\, 0\leq Y_t\leq X_t$, where $X_t$ is a  Feller diffusion,
% and thanks to (\ref{equationmomentfeller}),  
all the stochastic integrals that we write are well defined.
Applying the two dimensional It\^o's formula with jumps to $F(K_t,Y_t)=\bar{Y}_t$, with $F(x,y)=\exp(-x)y$, we get 
\begin{align*}
\bar{Y}_t= & y_0+\int_0^t e^{-K_s}\left[rY_s ds+\sqrt{2\gamma Y_s} dB_s\right]-  \int_0^t rY_s e^{-K_s} \ud s
\nonumber\\
& \quad +  \int_0^t \int_{0}^{\infty}\int_{0}^1 \left(Y_s e^{-K_s}- Y_{s_-}e^{-K_{s_-}}\right)\ind_{u\leq \tau}N_1(\ud s,\ud u,\ud \theta)\nonumber\\
= & y_0+ \int_0^t e^{-K_s}\sqrt{2\gamma Y_s}\ud B_s+\int_0^t \int_{0}^{\infty}\int_0^1 \bar{Y}_{s_-}\left(\theta e^{-\log(\theta)}-1\right)
\ind_{u\leq \tau}N(\ud s,\ud u,\ud \theta)\nonumber\\
= & y_0+ \int_0^t e^{-K_s}\sqrt{2\gamma Y_s}\ud B_s.
\end{align*}
Then  $(\bar{Y}_t : t\geq 0)$  satisfies the SDE (\ref{sde}).
\end{proof}
%$\newline$

\bi We  now state the absorption  criterion.  
\begin{Thm}\label{extlignee} 
(i) If $r\leq \E(\log(1/F))\tau$, then  $\P(\exists t >0 : Y_t=0)=1$. \\
%Moreover if $r< \E(\log(1/\Theta))\tau$, %there exists $\alpha>0$ such that for every $y_0\geq 0$, there exists $c>0$
%\be
%\label{tempsderetour}
%\exists \alpha>0,\, \forall x_0\geq 0,\, \exists c>0,\, \forall t\geq 0,\quad \P_{x_0}(Y_t>0)\leq ce^{-\alpha t}.
%\ee
(ii) Otherwise, %If $r>\E(\log(1/F))\tau$, then  
$\mathbb{P}(\forall t\geq 0 : \  Y_t>0)>0$.
\end{Thm}
%extentable to more general CSBP and L\'evy process $K$, see .

\me For the proof, we first  consider the quenched process, conditioned by the environment $\mathcal F^{N_1}$. Indeed   the transformation $Y_{t-} \rightarrow xY_{t-}=Y_t$ preserves  the  branching property of the  Feller diffusion.  The  Feller diffusion then undergoes deterministic  catastrophes given by the conditional Poisson point measure $N_1$.
 \begin{Lem}[Quenched characterization] \label{quenched} (i) Conditionally on $\mathcal F^{N_1}$ and setting for $t_0, \lambda \geq0$ and $t\in [0,t_0]$,
$$u(t,y)=\exp\left(-\frac{\lambda y}{\gamma \lambda \int_{t}^{t_0} e^{-K_s}ds+1}\right),$$
The process $(u(t,\bar{Y}_t) :  0\leq t\leq t_0)$ is a bounded martingale. \\
(ii) For all  $t,\lambda,y_0\geq 0$,
 \be \label{laplace}
\mathbb{E}_{y_0}\left(\exp(-\lambda \bar{Y}_t)  \ \vert \ {\mathcal F^{N_1}}\right)=  \exp\left(-\frac{\lambda y_0}{\gamma \lambda \int_0^t e^{-K_s}\ud s+1}\right).
\ee
\end{Lem}

\begin{exo} Prove that conditionally on $\mathcal{F}^{N^1}$,   $Y$ satisfies the branching property. \\
\emph{One may write a direct proof following Proposition \ref{branch} or use Lemma \ref{quenched}  (ii).}
\end{exo}

\me We stress that the non conditional  (\emph{annealed}) branching property does not hold.
Formally, the quenched process can be defined  on the probability space $(\Omega, \mathcal F, \P):=(\Omega_e\times \Omega_d, \mathcal F_e \otimes \mathcal F_d,\P_{e}\otimes \P_d)$ by
 using a Poisson Point process $N_1(w)=N_1(w_e,w_d):=N_1(w_e)$ for catastrophes and a Brownian motion $B_t(w)=B_t(w_e,w_d):=B_t(w_d)$ for the demographic stochasticity. Thus, the process $Y$   
 conditioned  on the environment $\mathcal{F}^{N_1} =\sigma(K_s)=\sigma(\mathcal{F}_e)$ (quenched process) is  given by $Y(w_e,.)$ $\P_{e}$ a.s. 
\begin{proof}
Let us work conditionally on $\mathcal{F}^{N^1}$. %First, notice that:
%\begin{align}
%0\leq & \E\Big(\int_0^t 2\gamma \bar{Y}_s\, e^{K_s}\, ds\, |\, \mathcal{F}_\infty^\rho \Big) = 2\gamma \int_0^t   \E\big(Y_s\big) e^{2K_s}ds
%\leq  2\gamma \int_0^t  \E\big(X_s\big) e^{2K_s}ds,\label{etape7}
%\end{align}by using Fubini's theorem and the domination of $Y$ by $X$. Then (\ref{equationmomentfeller}) provides the finiteness of the expectation of the bracket and $\bar{Y}$ is a real square integrable martingale with respect to $(\mathcal{F}^\beta_t : t\geq 0)$ and conditionally on $\mathcal{F}_\infty^\rho$.
 %, which means that $(K_s : s\geq 0)$ is now a fixed function.
 Using It\^o's formula for a function $u(t,y)$ which is differentiable by parts with respect to $t$ and infinitely differentiable with respect to $y$, we get
\begin{align*}
 u(t,\bar{Y}_t)= u(0,y_0)+&  \int_0^t \left[\frac{\partial u}{\partial s}(s,\bar{Y}_s)+\frac{\partial^2 u}{\partial y^2}(s,\bar{Y}_s)\gamma \bar{Y}_s e^{-K_s}\right] \ud s \\
 +& \int_0^t \frac{\partial u}{\partial y}(s,\bar{Y}_s)e^{-K_s/2}\sqrt{2\gamma \bar{Y}_s}\ud B_s.
\end{align*}
The function $u$ has been chosen to cancel the finite variation part, i.e. it satisfies
$$\frac{\partial u}{\partial s}(s,y)+\frac{\partial^2 u}{\partial y^2}(s,y)\gamma y e^{-K_s} =0 \qquad (s,y\geq 0).$$
Then the process $(u(t,\bar{Y}_t) : 0\leq t\leq t_0)$ is a local martingale bounded by $1$ and thus a real martingale and $(i)$ holds. We  deduce that
$$\mathbb{E}_{y_0}\left(u(t_0,\bar{Y}_{t_0})\, |\, \mathcal{F}^{N_1}\right)=u(0,\bar{Y}_{0}),$$
which gives $(ii)$. 
\end{proof}

\me Before proving Theorem  \ref{extlignee}, we deal with the functional of the L\'evy process involved in the extinction.
\begin{Lem}\label{intexp}
If $r \leq \tau \E(\log(1/F))$, then  $\liminf_{t\rightarrow\infty} K_t=-\infty$ and $\int_0^{\infty} \exp(-K_s)\ud{s}=+\infty$ a.s. \\
Otherwise, %If $g>r\E( \log(1/F))$, alors  
$\lim_{t\rightarrow\infty} K_t=+\infty$ and $\int_0^{\infty} \exp(-K_s)\ud{s}<+\infty$ a.s. 
\end{Lem}
\begin{proof}
If  $r \leq \tau \E(\log(1/F))$, then $\E(K_1)\leq 0$ and the L\'evy process $K_t$ either goes to $-\infty$ or oscillates. In any case, the sequence of stopping times
$$T_0:=0, \qquad T_k:=\inf\{ t \geq T_{k-1}+1 :  K_t \leq 0\} <+\infty$$
is finite for $k \geq 0$.
Then,
\Bea
\int_0^{\infty} \exp(-K_s)\ud s &\geq & \sum_{k\geq 1} \int_{T_k}^{T_{k}+1} \exp(-K_s)\ud s \\
&\geq& \sum_{k\geq 1} \int_{0}^{1} \exp(-(K_{T_k+s}-K_{T_k}))\ud s 
=: \sum_{k\geq 1} X_k,
\Eea
where $X_k$ are non-negative (non identically zero) i.i.d.  random variables. Then $\int_0^{\infty} \exp(K_s) \ud s=+\infty$ a.s.
Conversely if  $r > \tau \E(\log(1/F))$, there exists $\epsilon>0$ such that $\E(K_1) -\epsilon>0$ and $K_t-\epsilon t$ goes to $+\infty$ a.s. Then
$$L:=\inf\{ K_s-\epsilon s : s\geq 0\}>-\infty \quad \text{a.s.}$$
and  $\int_0^{\infty} \exp(\epsilon s)\ud s=+\infty$ yields the result.
\end{proof}

\begin{proof}[Proof of Theorem \ref{extlignee}]
Integrating  (\ref{laplace}), we get by bounded convergence that %, the r.h.s. of (\ref{exposantdeLaplace}) converges to
$$ \lim_{t\rightarrow \infty}\E_{y_0}( \exp(-\lambda \bar{Y}_t)) =\E_{y_0}\left(\exp\left(-\frac{\lambda y_0}{\gamma \lambda\int_0^{\infty} \exp(-K_s)\ud{s}+1}\right)\right).$$ 
The process $(\bar{Y}_t : t\geq 0)$  converges  in distribution as $t\rightarrow + \infty$  to  $\bar{Y}_{\infty}$  whose distribution is specified by the right hand side of the above limit.
%where $(K_t : t\geq 0)$ defined in (\ref{defKt}) is a LÃƒÂ©vy process.
Letting  $\lambda\rightarrow +\infty$, we get by bounded convergence:
\begin{align}
\label{pext}
\P_{y_0}(\bar{Y}_{\infty}=0)& 
=  \mathbb{E}_{y_0}\left(\exp\left(-\frac{y_0}{\gamma \int_0^{\infty} \exp(-K_s)\ud{s}}\right)\right).
\end{align}
Recalling from Lemma \ref{propcorrectionsauts} that $(\bar{Y}_t : t\geq 0)$ is a non-negative local martingale, we obtain by Jensen's inequality that $(\exp(-\bar{Y}_t) : t\geq 0)$ is a positive sub-martingale bounded by $1$. We deduce that the convergence towards $\bar{Y}_\infty$, which is possibly infinite, also holds a.s.\\
Using Lemma \ref{intexp}, we obtain that the probability of the event 
$$\{\liminf_{t\rightarrow \infty} Y_t=0\}=\{\liminf_{t\rightarrow \infty} e^{K_t} \bar{Y}_t=0\}$$ is either one or less than one depending on $r\leq \tau\E(\log(1/F))$ or $r>\tau\E(\log(1/F))$. Moreover, the absorption probability of the  Feller diffusion  is positive (see Section \ref{CP}) and the Markov property ensures that $Y$ is a.s. absorbed on the event 
$\{\liminf_{t\rightarrow \infty} Y_t=0\}$. 
\end{proof}

\me Let us note that the a.s. convergence relies here on the fact that  $(\exp(-\bar{Y}_t) : t\geq 0)$ is a  bounded sub-martingale. This method can
 be adapted to  study the long time behavior of a conservative CSBP instead of using the Lamperti transform (see Section \ref{CSBPP}).
We describe now the speed at which the absorption occurs. The random environment makes three asymptotic  regimes appear  in the subcritical case.

\begin{Thm}[Growth and speed of extinction,  \cite{BPS}]
 \label{equiv}
We assume that $\E((\log F)^2)<\infty$.
%int_0^1 \log(1/\theta) K(\ud \theta)<\infty.$%$ $\nu$ satisfies (\ref{condh1}) and (\ref{CLT}), and that $\psi$ and $\phi$ satisfy (\ref{defstable}) and (\ref{condexp}) respectively.
\begin{enumerate}
\item[a/]
 If $r<\tau\E(\log (1/F))$ \emph{(subcritical case)}, then 
\begin{enumerate}
 \item[(i)] If $\tau\E(F\log F)+r<0$  \emph{(strongly subcritical regime)}, then there exists $c_1>0$ such that for every $y_0>0$,
$$ \P_{y_0}(Y_t>0) \sim c_1  y_0 e^{t(r+\tau[\E( F)-1])}, \qquad \textrm{ as }\quad t\rightarrow \infty.$$

 \item[(ii)] If $\tau\E(F\log F)+r=0$ \emph{(intermediate subcritical regime)}, then there exists $c_2>0$ such that for every $y_0>0$,
$$\P_{y_0}(Y_t>0) \sim c_2 y_0 t^{-1/2}e^{t(r+\tau[\E(F)-1])}, \qquad \textrm{ as }\quad t\rightarrow \infty.$$

 \item[(iii)] If $\tau\E(F\log F)+r>0$   \emph{(weakly subcritical regime)}, %and $\theta_{max}>\beta+1$, 
then for every $y_0>0$, there exists $c_3>0$ such that
$$\P_{y_0}(Y_t>0) \sim c_3t^{-3/2}e^{t(r+\tau[\E(F^\chi)-1])}, \qquad \textrm{ as }\quad t\rightarrow \infty, $$
where $\chi$ is the root of $\E(F^{\chi} \log F)+r$ on $(0,1)$. %: $\phi(\tau)+g\tau= \underset{0 < s < 1}{\min}\{\phi(s)+gs\}$. }
\end{enumerate}
\item[b/] If $r=\tau\E(\log (1/F))$
%\footnote{We exclude the degenerated case $\nu = 0$, $g=0$.} 
 \emph{(critical case)}, then for every $y_0>0$, there exists $c_4>0$ such that 
$$\P_{y_0}(Y_t>0) \sim c_4t^{-1/2}, \qquad \textrm{ as }\quad t\rightarrow \infty. $$
\item[c/] If $r>\tau\E(\log (1/F))$ \emph{(supercritical case)}, then there exists a finite r.v. $W$ such that
$$e^{-K_t}Y_t\xrightarrow[t\rightarrow \infty]{}W \quad a.s., \qquad \{W=0\}=\Big\{\lim_{t\to\infty}Y_t =0\Big\}.$$
\end{enumerate}
\end{Thm}

\me The  asymptotic results $a/-b/$ rely on the study of
$\ \P(Y_t>0)=\mathbb{E}\left(f\left(\int_0^t e^{-\beta K_s}\ud s\right)\right)$,
for $t\rightarrow \infty$, where $f$ has a polynomial decay when $x\rightarrow \infty$ and here $\beta=1$. It is linked to the asymptotic  distribution of
$I_t=\inf\{ K_s : s\leq t\}$ and the different asymptotics  appear for $\P(I_t \geq x)$ when $t\rightarrow \infty$. The proof in \cite{BPS} uses a discretization
of $\ \int_0^t \exp(-\beta K_s)\ud s\ $ of the form $\ \sum_{i=0}^n \Pi_{j=0}^i \,A_i\ $
and the different regimes are inherited from local limit theorems for semi-direct products \cite{LP}.

\subsection{Monotone rate of catastrophes}

\me We first deal with increasing rates, which are relevent for the applications on cell infection developped in the second part.
%For the motivations of cell infection and population regulation, the case where $\tau$ is increasing is particularly relevent.
\label{catam}
\begin{Prop}\label{linerincre} We assume that $\tau$ is a non-decreasing function. \\
%\textbf{Extinction criterion of $Y$ for increasing $r$} \\
(i) If there exists  $y_1\geq 0$ such that  $r\leq \E(\log(1/F))\tau(y_1)$, then
$$\P\Big(\exists  t>0, \ Y_t=0\Big)=1.$$
(ii) If $r>\E(\log(1/F))\sup_{x\geq 0} \tau(x)$,   then $\ \P(\forall t\geq 0: \  Y_t>0)>0$.
%\\ Furthermore, for every $0\leq \kappa<g-\E(\log(1/\Theta))r^*,$ we have a.s.
%$$\{\lim_{t\rightarrow +\infty}e^{-\kappa t}Y_t=\infty\}=\{\forall t\geq 0: \  Y_t>0\}.$$
\end{Prop}
%The expected result is a.s. extinction but this may depend on the speed of convergence of $r(x)$ to $r^*$ as $x\rightarrow +\infty$.

 \me Heuristically, if $r\leq \E(\log(1/F))\tau(y_1)$, as soon as $Y\geq y_1$, the division rate is larger than $r(y_1)$ and
Theorem \ref{extlignee} $(i)$ ensures that the process is pushed back to $x_1$. Eventually, it reaches zero by the strong Markov property, since each time it goes below $y_1$, it has a 
 positive probability to reach $0$ before time $1$. \\
The proof can be made rigorous by using a coupling with a  Feller diffusion with catastrophes  occurring at constant rate $\tau(y_1)$. Finally, we   note that  the 
case $$r>\E(\log(1/F))\tau(x) \ \ \emph{ for every } \ x\geq 0; \qquad r=\E(\log(1/F))\sup_{x\geq 0} \tau(x)$$ remains open. 

\me Let us now consider the case when $\tau$ is non-increasing.
The asymptotic behavior now depends on
\begin{equation}\tau_*=\inf_{x\geq 0}\tau(x).\label{defr_star(inf)}\end{equation}
\begin{Prop}\label{extrdecr} We assume that $r$ is a non-increasing function. \\
(i) If $r\leq \E(\log(1/F))\tau_*$, then
$\P\Big(\exists  t>0, \ Y_t=0\Big)=1.$\\
(ii) Else, $\P(\forall t>0, \ Y_t>0)>0$.% and for all $0<\kappa<g-\E(\log(1/\Theta))r_*$,
%ithere exists $x_1\geq 0$ such that $g> \E(\log(1/\Theta))\,r(x_1)$, then
%$$ \{\lim_{t\rightarrow +\infty} e^{-\kappa t} Y_t= \infty\}=\{ \forall t>0: \ Y_t>0\}\quad \mbox{a.s.}$$
\end{Prop}

\me The proof is easier in that case and we   refer to  \cite{MR2754402} for details.

\subsection{Further comments : CSBPs in random environment}

\me We have studied  a particular case of continuous state branching processes in random environment.  It can be extended in several ways.
A general construction of  Feller diffusions in random environment  has been provided by Kurtz \cite{Kurtz78}. It relies on the scaling and  time change 
$$Y_t=\exp(M_t)Z_{\tau_t}, \qquad \emph{where} \ \tau_t=\int_0^t \exp(-M_s)A_s \ud s,$$
$M$ is a c\`adl\`ag process and   $Z$ is a  Feller diffusion and $A$ is non-decreasing and absolutely continuous w.r.t. Lebesgue measure. \\
In the case where $M_t=K_t$ and $A_t=\gamma t$, this construction    leads to  the  Feller diffusion with catastrophes. 
   When  $M_t$ is a Brownian motion and  $A_t=\gamma t$, we obtain a  Feller diffusion in a Brownian environment. Its asymptotic behavior is close to the one of  Feller diffusion  with catastrophes and can 
 be found in \cite{BH}. It uses the explicit expression  of the density of some functional of Brownian motion involved in the Laplace exponent of the process.

\me The construction of Kurtz  has been extended by Borovkov \cite{Bor} to the case where $A$ is no more absolutely continuous.  The time change can also be generalized
 to the stable case but does not hold for general CSBP in random environment.
We refer to \cite{bansim} for the quenched construction without second moment assumptions, as a scaling limit, when $M$ has finite variations.
Let us mention \cite{BPS} for the asymptotic study of CSBP with catastrophes, which extends the asymptotic results given in Section \ref{cata}.
New asymptotic behaviors appear for the process, which may oscillate in $(0,\infty)$ in  the critical regime. 

\me A possible generalization of our results is the combination of  Brownian fluctuations  of the drift and  catastrophes
\label{complements}
$$
Y_t =
y_0 +\int_0^t Y_{s-} \ud M_s
 +  \int_0^t \sqrt{2\gamma_s Y_s}\ud B_s
 -  \int_0^t \int_{\R_+\times [0,1]}
\ind_{u\leq \tau(Y_{s_-})} \big(1-\theta\big) Y_{s_-}  N_1(\ud s,\ud u,\ud \theta),$$
where $\ud M_s=r_s\ud s+\sigma_e\ud B^e_s$ and  $N_1$ is a Poisson Point Process with intensity $ds\,du\,\P(F\in  \ud \theta)$ and both are independent of $B$.
 Such stochastic differential equations  both integrate small fluctuations and dramatic random events due to the environment. Finding a 
 general approach to deal with the long time behavior of CSBP in random environment is an open and delicate question.

\part{Structured Populations and Measure-valued Stochastic Differential Equations}

In this chapter the individuals are characterized by some quantitative traits. Therefore the population  is described by  a random point measure with support on the trait space. We first study the measure-valued process modeling the population dynamics including competition and mutation events.  We give its martingale properties and determine some scaling limits. Then we consider the particular case of cell division with parasite infection and the Markov processes indexed by continuous time Galton-Watson trees. 

\section{Population Point Measure Processes}
\label{sec:ppp}

%\marginpar{A mettre ÃƒÂ  la fin de la partie et lier ÃƒÂ  la dim infinie ?}
\subsection{Multitype  models}
\me In the previous sections, the models that we considered described a homogeneous population and could be considered as toy models. A first generalization consists in considering multitype population dynamics. The demographic rates of a  subpopulation   depend on its own type. The ecological parameters  are functions of the different types of the individuals competiting  with each other. Indeed, we assume that the type has an influence on the reproduction  or survival abilities, but also  on the access to resources. Some subpopulations can be more adapted than others to the environment.

\bi For simplicity, the examples that we consider now deal with only two types of individuals.
Let us consider   two sub-populations  characterized by two different  types $1$ and $2$. For $i=1, 2$, the growth rates of these
populations are $r_{1}$ and $r_{2}$. Individuals compete for resources either inside the same species (intra-specific competition) or 
with individuals of the other species (inter-specific competition).
As before, let $K$ be the scaling parameter describing  the  capacity of the environment. The  competition pressure exerted by an individual of type $1$ on an individual of type 1 
(resp. type 2)  is given by $c_{11}/K$ (resp.   $c_{21}/K$). The competition pressure exerted by an individual of type $2$ is respectively given 
by $c_{12}/ K$ and $c_{22}/ K$. The parameters $c_{ij}$ are assumed to be positive.

\me  By similar arguments as in Subsection 3.1, the large $K$-approximation of the population dynamics is described by the well known competitive Lotka-Volterra dynamical system. Let $x_{1}(t)$ (resp. $x_{2}(t)$) be the  limiting renormalized type  $1$ population size (resp. type $2$ population size). We get
 \be
 \label{LV}
\begin{cases}
& x'_{1}(t) = x_{1}(t)\,(r_{1}- c_{11}\,x_{1}(t) - c_{12}\,x_{2}(t));\\
 & x'_{2}(t) = x_{2}(t)\,(r_{2}- c_{21}\,x_{1}(t) - c_{22}\,x_{2}(t)).
 \end{cases}
 \ee
 This system has been extensively studied and its  long time behavior is well known. There are 4 possible equilibria: the unstable equilibrium $(0,0)$ 
 and three stable ones:  $({r_{1}\over c_{11}},0)$, $(0, {r_{2}\over c_{22}})$ and a non-trivial equilibrium $(x_{1}^*, x_{2}^*)$ given by
 $$x_{1}^* = {r_{1}  c_{22} - r_{2}c_{12}\over c_{11}c_{22}-c_{12}c_{21} }\ ;\ x_{2}^*= {r_{2}  c_{11} - r_{1}c_{21}\over c_{11}c_{22}-c_{12}c_{21} }.$$
 Of course, the latter is possible if the two coordinates are positive.
  The (unique) solution of \eqref{LV} converges to one of the stable equilibria, describing either the fixation of one species or the co-existence of both species. 
  The choice of the limit depends on the signs of the quantities  $r_{2}  c_{11} - r_{1}c_{21}$ and $r_{1}  c_{22} - r_{2}c_{12}$  which respectively quantify 
  the invasion ability of the subpopulation $2$ (resp. $1$) in  a  type $1$ (resp. type $2$) monomorphic resident population.
    
  \me One could extend \eqref{LV}  to negative coefficients $c_{ij}$, describing a cooperation effect of species $j$ on the growth of species $i$. The long time behavior can be totally different.
 For example, the prey-predator models have  been extensively studied  in ecology (see \cite{Hofbauer}, Part 1). The simplest prey-predator system
 \be
 \label{PP}
\begin{cases}
& x'_{1}(t) = x_{1}(t)\,(r_{1}-  c_{12}\,x_{2}(t));\\
 & x'_{2}(t) = x_{2}(t)\,(c_{21}\,x_{1}(t) - r_{2}),
 \end{cases}
 \ee
 with $r_{1}, r_{2}, c_{12}, c_{21}>0$,
  has   periodic solutions.    
  
  \bi Stochastic models have also been developed, based on this two type-population model. Following the previous sections,
  a first  point of view consists in generalizing the logistic Feller stochastic differential equation to this two-dimensional framework. The stochastic logistic Lotka-Volterra process is then defined by
 \be
 \label{SLVP}
 \begin{cases}
& d X_{1}(t) = X_{1}(t)\,(r_{1}- c_{11}\,X_{1}(t) - c_{12}\,X_{2}(t))\, dt + \sqrt{\gamma_{1} X_{1}(t) } d B_{t}^1;\nonumber \\
 & d X_{2}(t) = X_{2}(t)\,(r_{2}- c_{21}\,X_{1}(t) - c_{22}\,X_{2}(t)) \, dt + \sqrt{\gamma_{2} X_{2}(t) } d B_{t}^2,
 \end{cases}
 \ee
 where the Brownian motions $B^{1}$ and $B^{2}$ are independent  and give rise to  the demographic stochasticity (see Cattiaux-M\'el\'eard \cite{cattiaux}).
 Another point of view consists in taking account of environmental stochasticity (see Evans, Hening, Schreiber \cite{evans}).

 \me Of course, we could also study multi-dimensional systems corresponding to multi-type population models. In what follows we are more interested in modeling  the case where  the types of the individuals belong to a continuum. That will allow us to add mutation events where the offspring of an individual may randomly mutate and  create a new type. 

 \subsection{Continuum of types and measure-valued processes}

\noindent Even if the evolution appears as a global change in the
state of a population, its basic mechanisms, mutation and
selection, operate at the level of individuals. Consequently, we
model the evolving population as a stochastic 
 system of interacting individuals, where each individual is characterized by a
vector of phenotypic trait values. The trait space ${\cal X}$ is
assumed to be a closed subset of $\rit^d$, for some $d\geq 1$.

\me We will denote by $M_F({\cal X})$ the set of all finite non-negative
measures on ${\cal X}$. Let ${\cal M}$ be the subset of $M_F({\cal
X})$ consisting of all finite point measures:
\begin{equation*}
  {\cal M} = \left\{ \sum_{i=1}^n \delta_{x_i} , \; n \geq 0,
    x_1,...,x_n \in {\cal X} \right\}.
\end{equation*}
Here and below, $\delta_x$ denotes the Dirac mass at $x$. For any
$\mu\in M_F({\cal X})$ and any measurable function $f$ on ${\cal X}$,
we set $\left< \mu, f \right> = \int_{{\cal X}} f d\mu$.

\me We wish to study the stochastic process $(Y_t, t\geq 0)$, taking its values
in ${\cal M}$, and describing the distribution of individuals and
traits at time $t$. We define
\begin{equation}
  \label{pop}
  Y_t = \sum_{i=1}^{N_{t}} \delta_{X^i_t},
\end{equation}
$N_{t}  = \langle Y_{t}, 1 \rangle \in {\mathbb{N}}$ standing for the number of individuals
alive at time $t$, and $X^1_t,...,X^{N_{t}}_t$ describing the
individuals' traits (in ${\cal X}$).

  \me We assume that the birth rate of an individual with trait $x$ is $b(x)$ and that for a population $\nu=\sum_{i=1}^{N}\delta_{x^i}$, its death rate is given by 
      $d(x,C*\nu(x))=d(x,\sum_{i=1}^{N}C(x-x^i))$. This death rate takes into account the intrinsic death rate of the individual, depending on its  phenotypic trait
      $x$ but also on the competition pressure exerted by the other individuals alive, modeled by the competition kernel $C$. Let $p(x)$ and $m(x,z)dz$ be respectively the 
      probability that an offspring produced by
    an individual with trait $x$ carries a mutated trait and the law of this mutant
    trait.

\noindent Thus, the population dynamics can be roughly summarized as
follows. The initial population is characterized by a (possibly
  random) counting measure $\nu_0\in {\cal M}$ at time $0$, and
  any
 individual with trait $x$ at time $t$ has two independent random exponentially
  distributed ``clocks'': a birth clock with parameter
  $b(x)$, and a death clock with parameter
  $d(x,C*Y_t(x))$.
If the death clock of an individual rings, this individual dies
  and disappears.
If the birth clock of an individual with trait $x$ rings, this
  individual produces an offspring. With probability $1-p(x)$ the
  offspring carries the same trait $x$; with probability $p(x)$ the
  trait is mutated.
 If a mutation occurs, the mutated offspring instantly acquires a
  new trait $z$, picked randomly according to the mutation step
  measure $m(x,z)dz$. When one of these events occurs, all
  individual's clocks are reset to 0.

\me  We are looking for a ${\cal M}$-valued Markov process
$(Y_t)_{t\geq 0}$ with infinitesimal generator $L$, defined for all
real bounded functions $\phi$ and $\nu \in {\cal M}$ by
\begin{align}
\label{generator}
 L\phi(\nu)  & = \sum_{i=1}^{N}b(x^i)(1-p(x^i))
  (\phi(\nu+\delta_{x^i})-\phi(\nu)) \notag \\
  & +\sum_{i=1}^{N} b(x^i)p(x^i)
 \int_{\cal X}(\phi(\nu+\delta_{z})-\phi(\nu))m(x^i,z)dz\notag \\
  & +\sum_{i=1}^{N}d(x^i,C*\nu(x^i))
  (\phi(\nu-\delta_{x^i})-\phi(\nu)).
\end{align}
The first term in~(\ref{generator}) captures the effect  of births without mutation, the second term the effect of
births with mutation and the last term the effect  of deaths. The
density-dependence makes the third term nonlinear.

\subsection{Pathwise construction of the process}
\label{sec:constr}

\me Let us justify the existence of a Markov process admitting $L$ as
infinitesimal generator. The explicit construction of
$(Y_t)_{t\geq
  0}$ also yields two side benefits: providing a rigorous and
efficient algorithm for numerical simulations (given hereafter) and establishing a general
method that will be used to derive some large population limits
(Section~\ref{sec:large-popu}).

\me 
We make the biologically natural assumption that the trait
dependency of birth parameters is ``bounded'', and at most linear
for the death rate. Specifically, we assume

\begin{ann}
\label{hyp1}

There exist constants $\bar{b}$, $\bar{d}$, $\bar{C}$, 
and $\alpha$ and a probability density function $\bar{m}$ on $\rit^d$
such that for each %$\nu=\sum_{i=1}^{N}\delta_{x^i}$ and for
$x,z\in {\cal X}$, $\zeta\in \mathbb{R}^+$,

\begin{gather*}
  b(x)\leq \bar{b},\quad d(x,\zeta)\leq
  \bar{d}(1+\zeta), \\
  C(x)\leq \bar{C}, \\
  m(x,z)\leq \alpha\, \bar{m}(z-x).
\end{gather*}
\end{ann}

\me These assumptions ensure that there exists a constant $\widehat{C}$,
such that  for a population  measure
$\nu=\sum_{i=1}^{N}\delta_{x^i}$, the total event rate,  obtained as the sum of all event
rates, is bounded by $\ \widehat{C}N(1+N)$.

\me Let us now give a pathwise description of the population process
$(Y_t)_{t\geq 0}$. We introduce the following notation.

\begin{Notation}
  \label{defh}
  Let $\mathbb{N}^*=\mathbb{N}\backslash \{0\}$. Let
  $H=(H^1,...,H^k,...): {\cal M} \mapsto
  (\mathbb{R}^d)^{\mathbb{N}^*}$ be defined by
$ H\left(\textstyle\sum_{i=1}^n \delta_{x_i}\right) =
    (x_{\sigma(1)},...,x_{\sigma(n)},0,...,0,...)$,
  where $\sigma$ is a permutation such that $x_{\sigma(1)}\curlyeqprec ... \curlyeqprec x_{\sigma(n)}$,
  for some arbitrary order $\curlyeqprec$ on $\mathbb{R}^d$ (for example the lexicographic order).
\end{Notation}
\noindent This function $H$ allows us to overcome the following (purely
notational) problem. Choosing a trait uniformly among all traits
in a population $\nu \in {\cal M}$ consists in choosing $i$
uniformly in $\{1,...,\left<
  \nu,1\right>\}$, and then in choosing the individual {\it number}
$i$ (from the arbitrary order point of view). The trait value of
such an individual is thus $H^i(\nu)$.

\me We now introduce the probabilistic objects we will need.

\begin{Def}
  \label{poisson}
  Let $(\Omega, {\cal F}, P)$ be a (sufficiently large) probability
  space. On this space, we consider
  the following four independent random elements:
  \begin{description}
  \item[\textmd{(i)}] an ${\cal M}$-valued random variable $Y_0$ (the
    initial distribution),
  \item[\textmd{(ii)}] independent Poisson point measures
    $N_1(ds,di,d\theta)$, and $N_3(ds,di,d\theta)$ on $\R_+
    \times \mathbb{N}^*\times\rit^+$, with the same intensity measure
    $\:  ds \left(\sum_{k\geq 1} \delta_k (di) \right) d\theta\:$ (the
    "clonal" birth and the death Poisson measures),
  \item[\textmd{(iii)}] a Poisson point measure
    $N_2(ds,di,dz,d\theta)$ on $\R_+ \times \mathbb{N}^* \times
    {\cal X} \times \rit^+$, with intensity measure $\: ds
    \left(\sum_{k\geq 1} \delta_k (di) \right) dz d\theta\:$ (the
    mutation Poisson point measure).
  \end{description}
  Let us denote by $( {\cal F}_t)_{t\geq 0}$ the canonical filtration
  generated by these processes.
\end{Def}

\me We finally define the population process in terms of these
stochastic objects.

\begin{Def}
  \label{dbpe}
  %Assume $(H)$. 
  A $( {\cal F}_t)_{t\geq 0}$-adapted stochastic process
  $Y=(Y_t)_{t\geq 0}$  is called
  a population process if a.s., for all $t\geq 0$,
  \begin{align}
    Y_t &= Y_0 + \int_{[0,t] \times \mathbb{N}^*
      \times\rit^+}\delta_{H^i(Y_\sm)} \One_{\{i \leq \left<
        Y_\sm,1 \right>\}} \  \indiq_{\left\{\theta \leq
        b(H^i(Y_\sm))(1-p(H^i(Y_\sm)))\right\}}
    N_1(ds,di,d\theta) \notag \\ &+ \int_{[0,t] \times \mathbb{N}^* \times
      {\cal X}\times\rit^+} \delta_{z } \indiq_{\{i \leq \left<
        Y_\sm,1 \right>\}} \  \One_{\left\{\theta \leq
        b(H^i(Y_\sm))
        p(H^i(Y_\sm))\,m(H^i(Y_\sm),z)
      \right\}} N_2(ds,di,dz,d\theta) \notag \\ &- \int_{[0,t] \times
      \mathbb{N}^* \times\rit^+} \delta_{H^i(Y_\sm)} \indiq_{\{i \leq
      \left< Y_\sm,1 \right>\}} \indiq_{\left\{\theta \leq
        d(H^i(Y_\sm), C*Y_\sm(H^i(Y_\sm))) \right\}}
    N_3(ds,di,d\theta) \label{bpe}
  \end{align}
\end{Def}

\me Let us now show that if $Y$ solves~(\ref{bpe}), then $Y$
follows the Markovian dynamics we are interested in.

\begin{Prop}
  \label{gi}
  Assume Assumption \ref{hyp1} holds and consider a solution $\, (Y_t)_{t\geq 0}\, $ of
  (\ref{bpe}) such that $\,\E(\sup_{t\leq
    T}\langle Y_t,\mathbf{1}\rangle^2)<+\infty,\ \forall T>0$. Then
  $(Y_t)_{t\geq 0}$ is a Markov process. Its infinitesimal generator
  $L$ is defined 
  by~(\ref{generator}). In particular, the law of $(Y_t)_{t\geq 0}$
  does not depend on the chosen order $\curlyeqprec$.
\end{Prop}

\begin{proof}
  The fact that $(Y_t)_{t\geq 0}$ is a Markov process is classical.
  Let us now consider a measurable  bounded function $\phi$.
   With our notation, $Y_0 = \sum_{i=1}^{\left< Y_0, 1\right>}
  \delta_{H^i(Y_0)}$.  A simple computation, using the
  fact that a.s., $\phi(Y_t) = \phi(Y_0) + \sum_{s\leq t}
  (\phi(Y_\sm + (Y_s-Y_\sm))- \phi(Y_\sm))$, shows that
  \begin{align*}
    \phi(Y_t) & = \phi(Y_0)+\int_{[0,t] \times
      \mathbb{N}^* \times\rit^+}\left(\phi(Y_\sm +
      \delta_{H^i(Y_\sm)}) - \phi(Y_\sm) \right)\indiq_{\{i \leq
      \left< Y_\sm,1 \right>\}} \\ &\hskip 2cm\indiq_{\left\{\theta
        \leq
        b(H^i(Y_\sm))(1-p(H^i(Y_\sm)))\right\}}
    N_1(ds,di,d\theta) \\ &+ \int_{[0,t] \times \mathbb{N}^* \times
      {\cal X}\times\rit^+} \left(\phi(Y_\sm + \delta_{z}) -
      \phi(Y_\sm) \right) \indiq_{\{i \leq \left< Y_\sm,1
      \right>\}} \\ &\hskip 2cm \indiq_{\left\{\theta \leq
        b(H^i(Y_\sm))p(H^i(Y_\sm))
        \,m(H^i(Y_\sm),z)
      \right\}} N_2(ds,di,dz,d\theta) \\ &+ \int_{[0,t] \times
      \mathbb{N}^* \times\rit^+} \left(\phi(Y_\sm -
      \delta_{H^i(Y_\sm)}) - \phi(Y_\sm) \right) \indiq_{\{i \leq
      \left< Y_\sm,1 \right>\}} \\ &\hskip 2cm \indiq_{\left\{\theta
        \leq d(H^i(Y_\sm), C*Y_\sm(H^i(Y_\sm))) \right\}}
    N_3(ds,di,d\theta) .
  \end{align*}
  Taking expectations, we obtain
  \begin{align*}
    \E(\phi & (Y_t))=\E(\phi(Y_0)) \\ &+\intot
    \E\Big(\sum_{i=1}^{\left< Y_s, 1\right>}\bigg\{\left(\phi(Y_s +
      \delta_{H^i(Y_s)}) - \phi(Y_s)
    \right)b(H^i(Y_s))(1-p(H^i(Y_s))) \\
    &+\int_{{\cal X}} \left(\phi(Y_s + \delta_{z}) - \phi(Y_s)
    \right)b(H^i(Y_s))p(H^i(Y_s))\,m(H^i(Y_s),z)dz
    \\ &+ \left(\phi(Y_s - \delta_{H^i(Y_s)}) - \phi(Y_s)
    \right)d(H^i(Y_s), C*Y_s(H^i(Y_s)))\bigg\}\Big)ds % \\
  \end{align*}
 % \begin{align*}
  %  & =E[\phi(\nu_0)] + \intot E\Big(\int_{{\cal X}}\bigg\{
   % \left(\phi(Y_s + \delta_{x}) - \phi(Y_s) \right)
    %b(x,V*Y_s(x))(1-p(x)) \\ &+
    %\int_{{\cal X}}  \left(\phi(Y_s +
     % \delta_{z}) - \phi(Y_s) \right)
   % b(x,V*Y_s(x))p(x)M(x,z)dz \\ &+ \left(\phi(Y_s - \delta_x)
    %  - \phi(Y_s) \right) d(x,U*Y_s(x)) \bigg\}
    %Y_s(dx)\Big)ds.
  %\end{align*}
  Differentiating this expression at $t=0$ leads to~(\ref{generator}).
\end{proof}

\me Let us show the existence and some moment properties for the population
process.

\begin{Thm}\label{existence}
  \begin{description}
  \item[\textmd{(i)}] Assume Assumption \ref{hyp1} holds and that $\E\left( \left<Y_0,1
      \right> \right) <\infty$. Then the process $(Y_t)_{t\geq 0}$
    defined in Definition \ref{dbpe} is well defined on $\rit_+$.
  \item[\textmd{(ii)}] If furthermore for some $p \geq 1$, $\E\left(
      \left<Y_0,1 \right>^p \right) <\infty$, then for any
    $T<\infty$,
    \begin{equation}
      \label{lp}
      \E(\sup_{t\in[0,T]} \left< Y_t,1\right>^p ) < +\infty.
    \end{equation}
  \end{description}
\end{Thm}

\begin{proof}
  We first prove~(ii). Consider the process $(Y_t)_{t\geq 0}$. We
  introduce for each $n$ the stopping time $\tau_n = \inf \left\{ t
    \geq 0 , \; \left<Y_t,1 \right> \geq n \right\}$. Then a simple
  computation using Assumption \ref{hyp1} shows that, dropping the
  non-positive death terms,
  \begin{align*}
    \sup_{s\in[0,t\land \tau_n]} & \left< Y_s,1\right>^p\leq \left<
      Y_0,1\right>^p + \int_{[0,t\land \tau_n] \times
      \mathbb{N}^* \times\rit^+} \left((\left< Y_\sm,1\right> +1)^p-
      \left< Y_\sm,1\right>^p\right) \indiq_{\{i \leq \left< Y_\sm,1
      \right>\}} \\ &\hskip 2cm\indiq_{\left\{\theta \leq
        b(H^i(Y_\sm))(1-p(H^i(Y_\sm)))\right\}}
    N_1(ds,di,d\theta) \\ &+ \int_{[0,t\land \tau_n] \times \mathbb{N}^* \times
      {\cal X}\times\rit^+} \left((\left< Y_\sm,1\right> +1)^p- \left<
        Y_\sm,1\right>^p\right)\indiq_{\{i \leq \left< Y_\sm,1
      \right>\}} \\ &\hskip 2cm\indiq_{\left\{\theta \leq
        b(H^i(Y_\sm))p(H^i(Y_\sm))
        \,m(H^i(Y_\sm),z)\right\}} N_2(ds,di,dz,d\theta).
  \end{align*}
  Using the inequality $(1+x)^p-x^p\leq C_p(1+x^{p-1})$ and taking
  expectations, we thus obtain, the value of $C_p$ changing from one line
  to the other,
  \begin{align*}
    \E(\sup_{s \in [0,t\land \tau_n]}\left< Y_s,1\right>^p)
     &\leq C_p\E \left( \int_0^{t\land \tau_n} 
        \left(\left< Y_\sm,1\right>+ \left< Y_\sm,1\right>^p
        \right) ds \right)\\ &\leq 
        C_p\left(1+ \E\left( \intot
        \left< Y_{s\land \tau_n},1\right>^p 
        ds\right)\right).
  \end{align*}
  The Gronwall Lemma allows us to conclude that for any $T<\infty$,
  there exists a constant $C_{p,T}$, not depending on $n$, such that
  \begin{equation}
    \label{xxxx}
    \E( \sup_{t \in [0,T\land \tau_n]} \left< Y_t,1\right>^p)\leq C_{p,T}.
  \end{equation}
  First, we deduce that $\tau_n$ tends a.s. to infinity. Indeed, if
  not, one may find  $T_0<\infty$ such that $\e_{T_0}=P\left( \sup_n
    \tau_n <T_0 \right) >0$. This would imply that $\,\E\left( \sup_{t
      \in [0,T_0\land \tau_n]}\left< Y_t,1\right>^p \right) \geq
  \e_{T_0} n^p\ $ for all $n$, which contradicts~(\ref{xxxx}).
We may let $n$ go to infinity in~(\ref{xxxx}) thanks to the Fatou
  Lemma. This leads to~(\ref{lp}).

 \me  Point~(i) is a consequence of point~(ii). Indeed, one builds the
  solution $(Y_t)_{t\geq 0}$ step by step. One only has to check
  that the sequence of jump instants $T_n$ goes a.s.\ to infinity as
  $n$ tends to infinity. But this follows from~(\ref{lp}) with $p=1$.
\end{proof}

\subsection{Examples and simulations}
\label{sec:simul}

\me Let us remark that Assumption \ref{hyp1} is satisfied in the case where
\begin{equation}
  d(x,C*\nu(x))=d(x)+\alpha(x)\int_{{\cal
  X}}C(x-y)\nu(dy),
\end{equation}
and $b$, $d$ and $\alpha$ are bounded functions.

\me In the case where moreover, $p\equiv 1$, this individual-based
model can also be interpreted as a model of ``spatially structured
population'', where the trait is viewed as a spatial location and
the mutation at each birth event is viewed as dispersal. This kind
of models has been introduced by Bolker and
Pacala~(\cite{BP97,BP99}) and Law et al.~(\cite{LMD03}), and
mathematically studied by Fournier and M\'el\'eard~\cite{FM04}.
The  case $\ C\equiv 1\ $ corresponds to a density-dependence in
the
total population size.

\me Later, we will consider the particular set of parameters leading to the
logistic interaction model, taken from Kisdi~\cite{Ki99} and
corresponding to a model of asymmetric competition:
\begin{gather}
{\cal X}=[0,4],\quad d(x)=0,\quad\alpha(x)=1,\quad p(x)=p,
  \notag \\ b(x)=4-x,\quad C(x-y)=\frac{2}{K}\bigg(1-{1\over 1+1.2
  \exp(-4(x-y))}\bigg)\label{eq:ex}
\end{gather}
and $m(x,z)dz$ is a Gaussian law with mean $x$ and variance
$\sigma^2$ conditioned to stay in
$[0,4]$. As we will see in Section~\ref{sec:large-popu}, the
constant $K$ scaling the strength of competition also scales the
population size (when the initial population size is proportional
to $K$). In this model, the trait $x$ can be interpreted as body
size. Equation~(\ref{eq:ex}) means that body size influences the
birth rate negatively, and creates asymmetrical competition
reflected in the sigmoid shape of $C$ (being
larger is competitively advantageous).

\bi Let us give now  an algorithmic construction of the population process
(in the general case), giving the size $N_{t}$ of the
population and the trait vector $\mathbf{X}_t$ of all individuals
alive at time $t$.

\me At time $t=0$, the initial population $Y_0$ contains $N_{0}$
individuals and the corresponding trait vector is
$\mathbf{X}_0=(X_0^i)_{1\leq i\leq N_{0}}$. We introduce the
following sequences of independent random variables, which will
drive the algorithm.
\begin{itemize}
\item The type of birth or death events will be selected according
to
  the values of a sequence of random variables $(W_k)_{k\in\nit^*}$
  with uniform law on $[0,1]$.
\item The times at which events may be realized will be described
  using a sequence of random variables $(\tau_k)_{k\in\nit}$ with
  exponential law with parameter $\widehat{C}$.
\item The mutation steps will be driven by a sequence of random
variables $(Z_k)_{k\in\nit}$ with law $\bar{m}(z)dz$.
\end{itemize}
We set $T_0=0$ and construct the process inductively for $k\geq 1$
as follows.

\noindent At step $k-1$, the number of individuals is $N_{k-1}$, and the
trait vector of these individuals is $\mathbf{X}_{T_{k-1}}$.

\noindent Let
$\displaystyle{T_{k}=T_{k-1}+\frac{\tau_k}{N_{k-1}(N_{k-1}+1)}}$.
Notice that $\displaystyle{\frac{\tau_k}{N_{k-1}(N_{k-1}+1)}}$
represents the time between jumps for $N_{k-1}$ individuals, and
$\widehat{C}(N_{k-1}+1)$ gives an upper bound of the total rate of events affecting 
each individual.

\me At time $T_k$, one chooses an individual $i_k=i$ uniformly at
random among the $N_{k-1}$ alive in the time interval
$[T_{k-1},T_k)$; its trait is $X^i_{T_{k-1}}$.  (If $N_{k-1}=0$
then $Y_t=0$ for all $t\geq T_{k-1}$.)
\begin{itemize}
\item If $\displaystyle{0\leq W_k\leq
    \frac{d(X^{i}_{T_{k-1}},\sum_{j=1}^{I_{k-1}}
      C(X^{i}_{T_{k-1}}-X^{j}_{T_{k-1}}))}{\widehat{C}(N_{k-1}+1)}=
    W_1^i(\mathbf{X}_{T_{k-1}})}$, then the chosen individual dies,
  and $N_k=N_{k-1}-1$.
\item If $\displaystyle{W_1^i(\mathbf{X}_{T_{k-1}})< W_k \leq
    W_2^i(\mathbf{X}_{T_{k-1}})}$, where
  \begin{equation*}
    W_2^i(\mathbf{X}_{T_{k-1}})=
    W_1^i(\mathbf{X}_{T_{k-1}})+{[1-p(X^{i}_{T_{k-1}})]b(X^{i}_{T_{k-1}}) \over
      \hat{C}(N_{k-1}+1)},
  \end{equation*}
  then the chosen individual gives birth to an offspring with
  trait $X^i_{T_{k-1}}$, and $N_k=N_{k-1}+1$.
\item If $\displaystyle{W_2^i(\mathbf{X}_{T_{k-1}})< W_k \leq
    W_3^i(\mathbf{X}_{T_{k-1}}, Z_k)}$, where
  \begin{multline*}
    W_3^i(\mathbf{X}_{T_{k-1}}, Z_k)=W_2^i(\mathbf{X}_{T_{k-1}})+ 
    {p(X^{i}_{T_{k-1}})b(X^{i}_{T_{k-1}})\, m(X^{i}_{T_{k-1}},
      X^{i}_{T_{k-1}}+Z_k)\over \widehat{C}\bar{m}(Z_k)(N_{k-1}+1)},
  \end{multline*}
  then the chosen individual gives birth to a mutant offspring with
  trait $X^i_{T_{k-1}}+Z_k$, and $N_k=N_{k-1}+1$.
\item If $W_k>W_3^i(\mathbf{X}_{T_{k-1}}, Z_k)$, nothing happens,
and
  $N_k=N_{k-1}$.
\end{itemize}

%\bigskip
%\end{document}

\me Then, at any time $t\geq 0$, the number of individuals  and the population process are defined by
$$N_{t}=\sum_{k\geq
  0}1_{\{T_k\leq t< T_{k+1}\}}N_k, \qquad 
  %and the population process is obtained as $
  Y_t=\sum_{k\geq
  0}1_{\{T_k\leq t<
  T_{k+1}\}}\sum_{i=1}^{N_k}\delta_{X^i_{T_k}}.$$

\begin{figure}
  \centering
  \mbox{\subfigure[$p=0.03$, $K=100$, $\sigma=0.1$.]%
{\epsfig{bbllx=0pt,bblly=0pt,bburx=17.99cm,bbury=22.23cm,%
figure=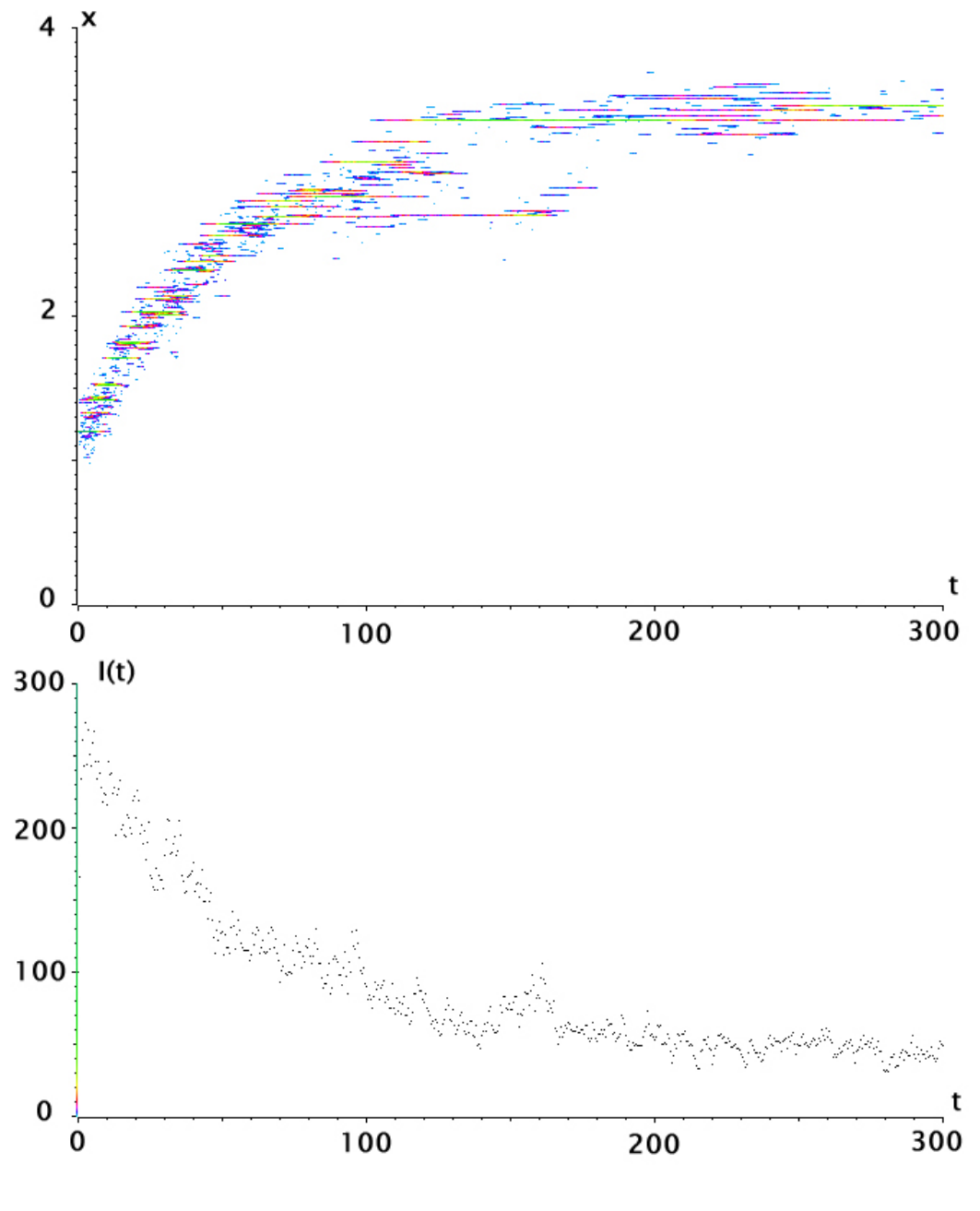, width=.47\textwidth}}\quad
   \subfigure[$p=0.03$, $K=3000$, $\sigma=0.1$.]%
{\epsfig{bbllx=0pt,bblly=0pt,bburx=17.99cm,bbury=22.23cm,%
figure=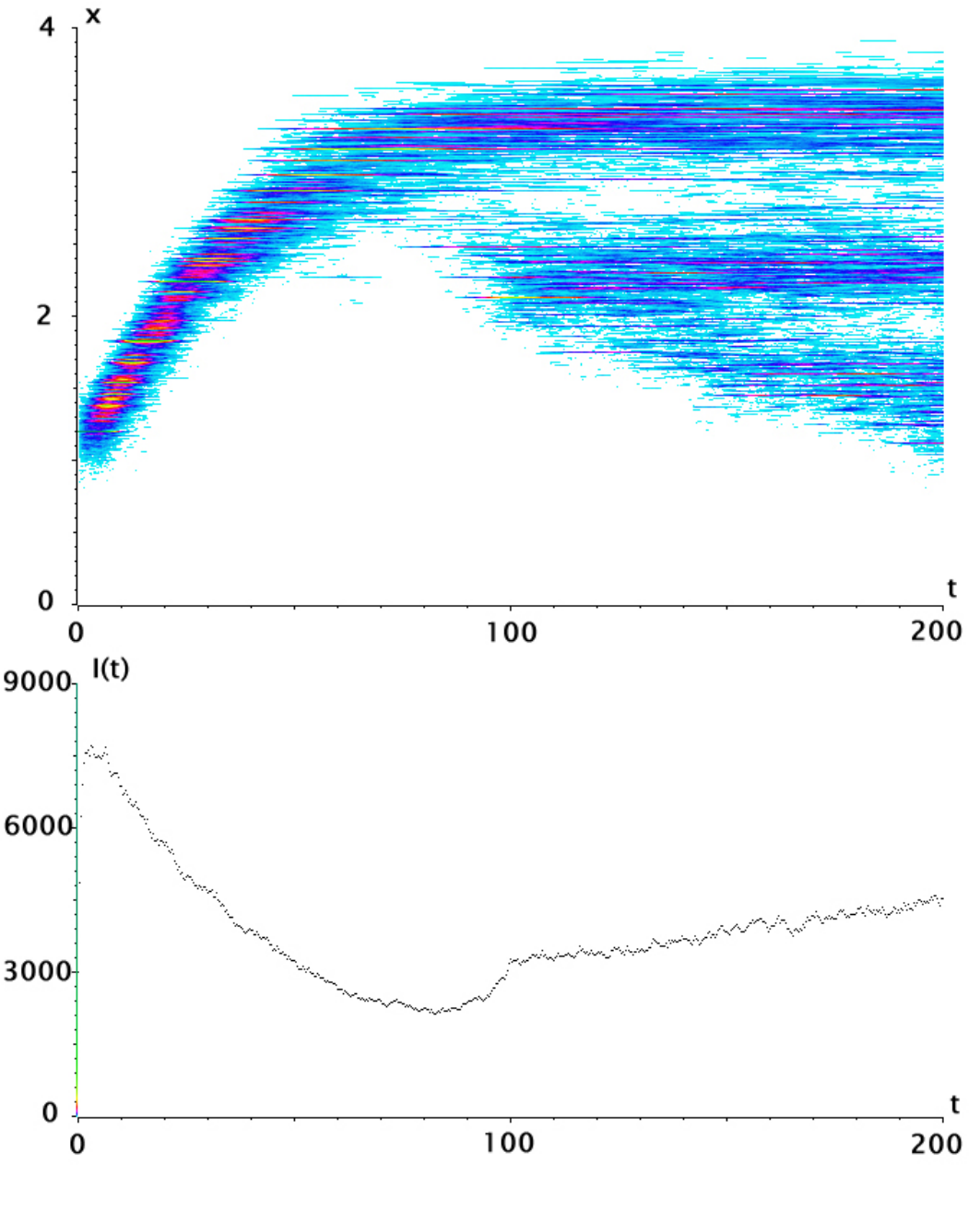, width=.47\textwidth}}} \\
  \mbox{\subfigure[$p=0.03$, $K=100000$, $\sigma=0.1$.]%
{\epsfig{bbllx=0pt,bblly=0pt,bburx=17.99cm,bbury=22.23cm,%
figure=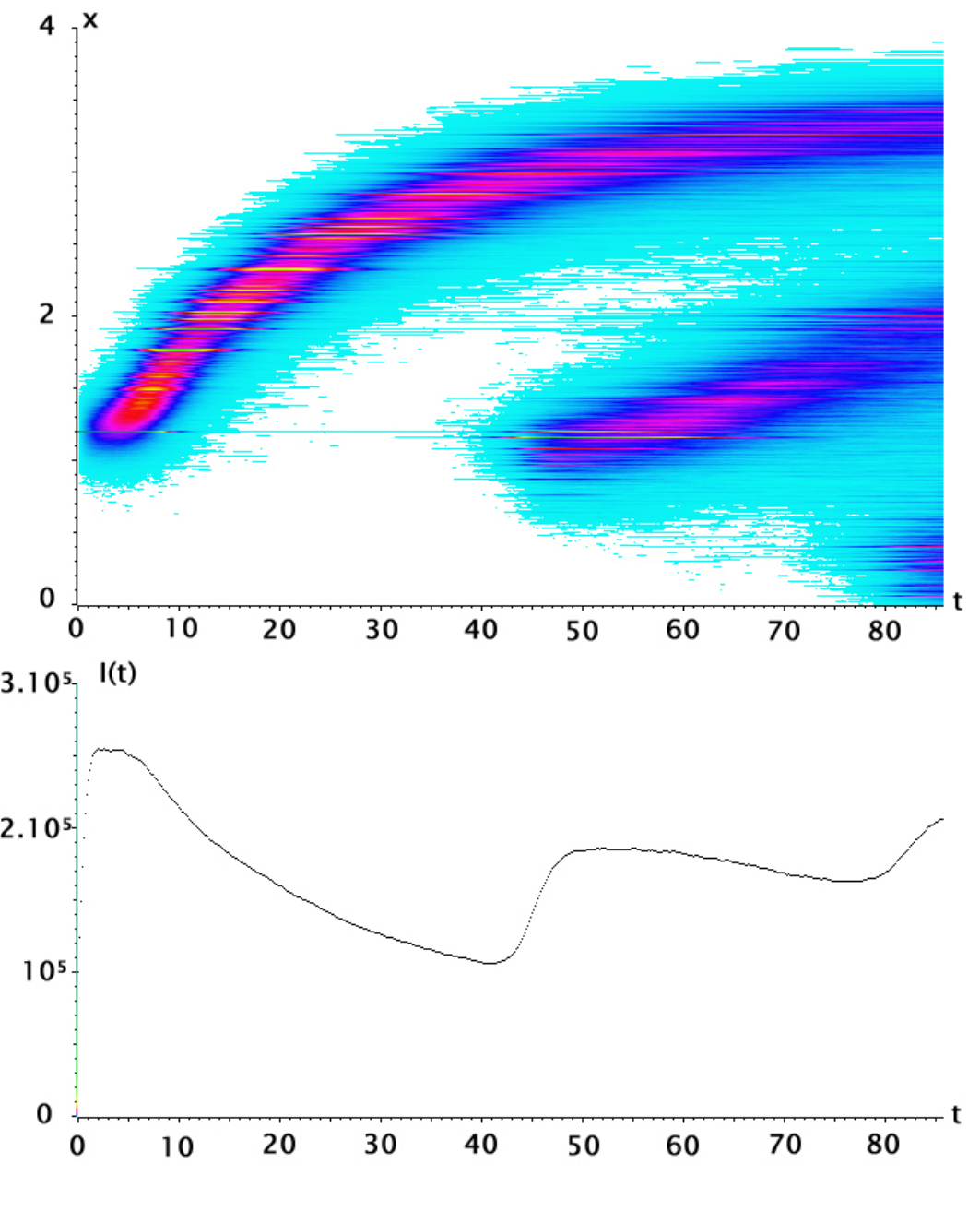, width=.47\textwidth}}}
  \caption{Numerical simulations of trait distributions (upper panels,
    darker means higher frequency) and population size (lower panels).
    The initial population is monomorphic with trait value $1.2$ and
    contains $K$ individuals. (a--c) Effect of increasing
    the system size (measured by the parameter $K$).}
  \label{fig:IBM-1}
\end{figure}

\me The simulation of Kisdi's example~(\ref{eq:ex}) can be carried out
following this algorithm. We can show a very wide variety of
qualitative behaviors depending on the value of the parameters
$\sigma$, $p$ and $K$.

\noindent In the following figures (cf. Champagnat-Ferri\`ere-M\'el\'eard \cite{CFM06}), the upper part gives the distribution of
the traits in the population at any time, using a grey scale code
for the number of individuals holding a given trait. The lower
part of the simulation represents the dynamics of the total population size
$N_{t}$.

\noindent These simulations will serve to illustrate the different
mathematical scalings described in Section~\ref{sec:large-popu}. In Fig.~\ref{fig:IBM-1}~(a)--( c),  we see the qualitative and quantitative effects of increasing scalings $K$, from a finite trait support  process for small $K$ to  a wide population density
for large $K$. The simulations of Fig.~\ref{fig:IBM-2} involve  birth and death processes with large rates (see Section~\ref{sec:accel}) given by
\begin{equation*}
  b(x)=K^{\eta}+b(x)\quad\mbox{and}\quad
  d(x,\zeta)=K^{\eta}+d(x)+\alpha(x)\zeta
\end{equation*}
and small mutation step $\sigma_{K}$.
There is a noticeable qualitative difference between
Fig.\ref{fig:IBM-2}~(a)  where $\eta=1/2$, and
Fig.\ref{fig:IBM-2}~(b)  where $\eta=1$. In the latter,
we observe strong fluctuations in the population size and a finely
branched structure of the evolutionary pattern, revealing a new
form of stochasticity in the large population approximation.

\begin{figure}
\label{fig:IBM-2}
  \centering
\mbox{\subfigure[$p=0.3$, $K=10000$, $\sigma=0.3/K^{\eta/2}$, $\eta=0.5$.]%
{\epsfig{bbllx=0pt,bblly=0pt,bburx=17.99cm,bbury=22.23cm,%
figure=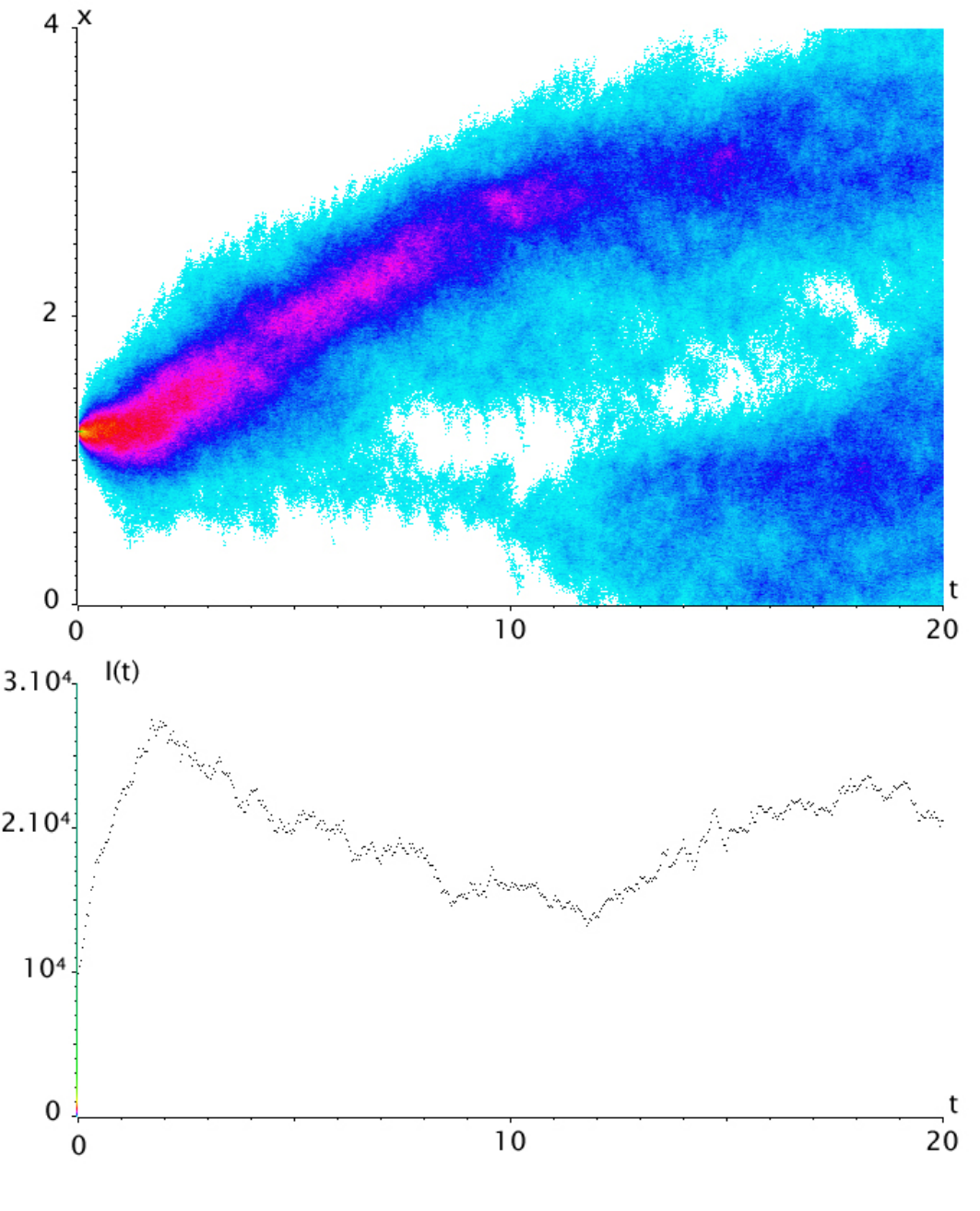, width=.47\textwidth}}\quad
   \subfigure[$p=0.3$, $K=10000$, $\sigma=0.3/K^{\eta/2}$, $\eta=1$.]%
{\epsfig{bbllx=0pt,bblly=0pt,bburx=17.99cm,bbury=22.23cm,%
figure=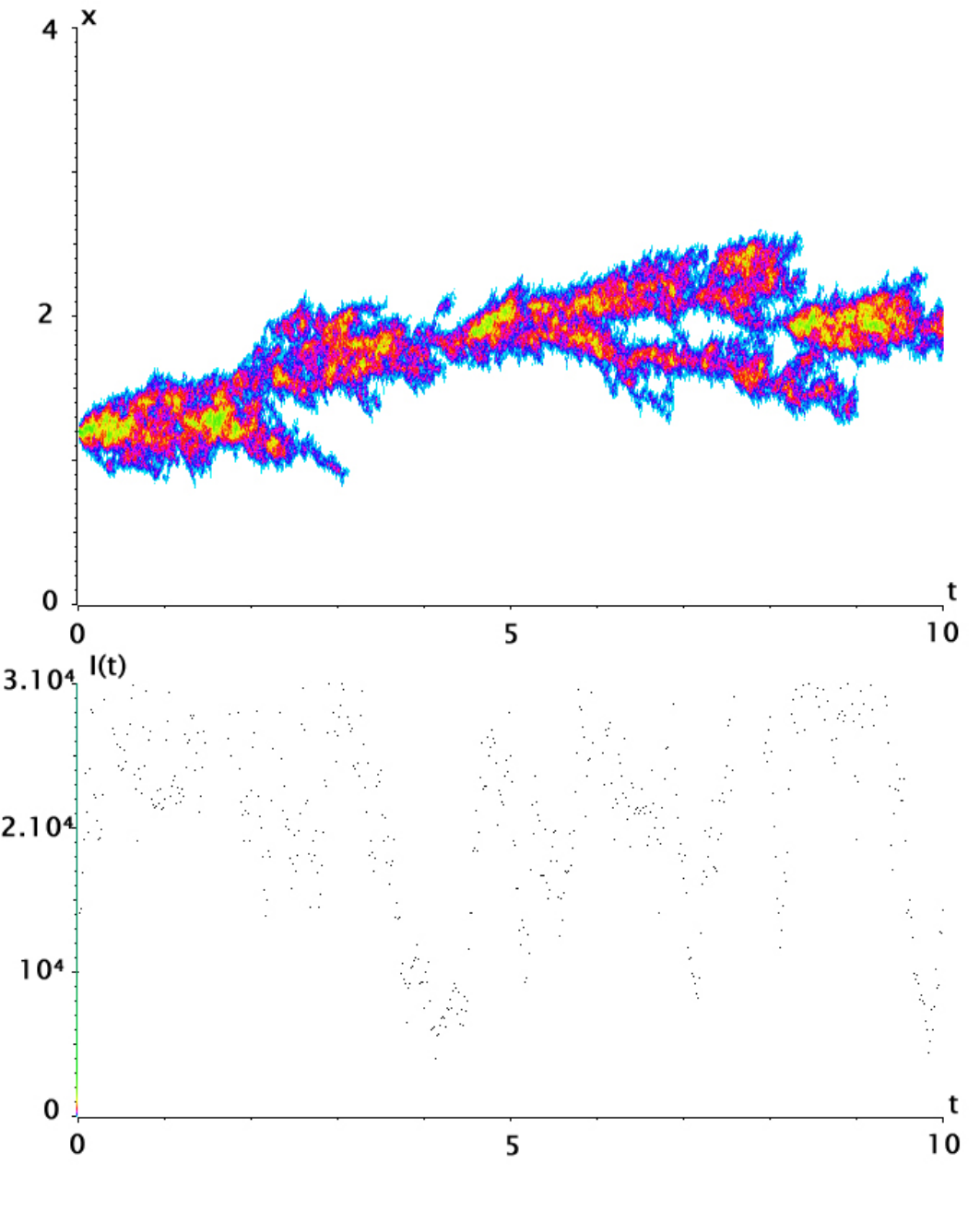, width=.47\textwidth}}}
\caption{Numerical simulations of trait distributions (upper panels,
  darker means higher frequency) and population size (lower panels) for
  accelerated birth and death and concurrently increased parameter
 $K$. The parameter $\eta$ (between $0$ and $1$) relates the acceleration of
 demographic turnover to the increase of system size $K$. (a) Case $\eta=0.5$. (b)  Case $\eta=1$. The initial population is monomorphic with trait value
  $1.2$ and contains $K$ individuals.} \label{fig:IBM-2}
\end{figure}

\subsection{Martingale Properties}\label{secexist}

\me We  give some martingale properties of the process
$(Y_t)_{t\geq 0}$, which  are the key point of our approach.

\begin{Thm}
  \label{martingales}
  Suppose Assumption \ref{hyp1} holds and that for some $p\geq 2$, $E \left( \left<
      Y_0,1 \right>^p \right) <\infty$.
  \begin{description}
  \item[\textmd{(i)}] For all measurable functions $\phi$ from ${\cal
      M}$ into $\mathbb{R}$ such that for some constant $C$, for all
    $\nu \in {\cal M}$, $\vert \phi (\nu) \vert + \vert L \phi (\nu)
    \vert \leq C (1+\left< \nu,1 \right>^p)$, the process
    \begin{equation}
      \label{pbm2}
      \phi(Y_t) - \phi(Y_0) - \intot  L\phi (Y_s) ds
    \end{equation}
    is a c\`adl\`ag $({\cal F}_t)_{t\geq 0}$-martingale starting from
    $0$.
  \item[\textmd{(ii)}] Point~(i) applies to any function $\phi(\nu)=
    \left< \nu, f \right>^q$, with $0\leq q\leq p-1$ and with $f$
    bounded and measurable on ${\cal X}$.
  \item[\textmd{(iii)}] For such a function $f$, the process
    \begin{align}
      M^f_t &= \langle Y_t, f\rangle - \langle Y_0,f\rangle -
      \intot \int_{\cal X}\bigg\{\bigg((1-p(x))b(x)
      -d(x,C*Y_s(x))\bigg)f(x) \notag \\
      &+p(x)b(x)\int_{\cal X}f(z)\,m(x,z)dz \bigg\}
      Y_s(dx) ds \label{eq:mart}
    \end{align}
    is a c\`adl\`ag square integrable martingale starting from $0$
    with quadratic variation
    \begin{align}
      \langle M^f\rangle_t &= \intot \int_{\cal
        X}\bigg\{\bigg((1-p(x))b(x)+d(x,C*Y_s(x))\bigg)f^2(x)
      \notag \\ &+p(x)b(x)\int_{\cal X}f^2(z)\,m(x,z)dz
      \bigg\} Y_s(dx) ds. \label{qv1}
    \end{align}
  \end{description}
\end{Thm}

\begin{proof} The proof follows  the proof of Theorem \ref{BD-mart}. 
  First of all, note that point~(i) is immediate thanks to
  Proposition~\ref{gi} and~(\ref{lp}). Point~(ii) follows from a
  straightforward computation using~(\ref{generator}). To prove~(iii),
  we first assume that $E \left( \left< Y_0,1 \right>^3 \right)
  <\infty$. We apply~(i) with $\phi(\nu) = \langle \nu, f\rangle$.
  This gives us that $M^f$ is a martingale.  To compute its bracket, we
  first apply~(i) with $\phi(\nu) = \left< \nu, f\right>^2$ and obtain
  that
  \begin{align}
    \langle Y_t, f\rangle^2 &- \langle Y_0,f\rangle^2 - \intot
    \int_{\cal X}\bigg\{\bigg((1-p(x))b(x)(f^2(x) + 2
    f(x)\left< Y_s, f \right>) \notag \\
    &\hskip 3cm+d(x,C*Y_s(x))(f^2(x) - 2 f(x)\left< Y_s, f
    \right>)\bigg) \notag \\ &+p(x)b(x)\int_{\cal X}(f^2(z)
    + 2 f(z) \left< Y_s, f \right>)\,m(x,z)dz \bigg\} Y_s(dx) ds \label{cc1}
  \end{align}
  is a martingale. On the other hand, we apply the It\^o formula to
  compute $\left< Y_t, f\right>^2$ from~(\ref{eq:mart}). We deduce
  that
  \begin{align}
    \langle Y_t, f\rangle^2 &- \langle Y_0,f\rangle^2 -\intot
    2\left< Y_s, f \right>\int_{\cal
      X}\bigg\{\bigg((1-p(x))b(x)-d(x,C*Y_s(x))\bigg)f(x)
    \notag \\ &+p(x)b(x)\int_{\cal X} f(z) m(x,z)dz
    \bigg\} Y_s(dx) ds - \langle M^f \rangle_t \label{cc2}
  \end{align}
  is a martingale. Comparing~(\ref{cc1}) and~(\ref{cc2}) leads
  to~(\ref{qv1}). The extension to the case where only $E \left(
    \left< Y_0,1 \right>^2\right)<\infty$ is straightforward by a localization argument, since
  also in this case, $E(\langle M^f \rangle_t)<\infty$ thanks
  to~(\ref{lp}) with $p=2$.
\end{proof}

\section{Scaling limits for the individual-based
  process}
\label{sec:large-popu}

\me As in Section 2, we consider the case where the  system size becomes very large. We scale  this size by the integer  $K$   and   look for approximations of the conveniently renormalized measure-valued population process, when $K$ tends to infinity. 

\me For any $K$, let the set of parameters $C_K$, $b_K$, $d_K$,
$m_K$, $p_K$ satisfy Assumption \ref{hyp1}. Let $Y^K_t$ be the
counting measure of the population at time $t$. We define the
measure-valued Markov process $(X^K_t)_{t\geq 0}$ by
\begin{equation*}
  X^K_t=\frac{1}{K}\,Y^K_t.
\end{equation*}

\me As the system size $K$ goes to infinity, we need to assume the

\begin{ann}
\label{hyp2} The parameters $C_K$,  $b_K$, $d_K$,
$m_K$ and $p_K$ are continuous,  $\zeta\mapsto d_{K}(x,\zeta)$ is uniformly  Lipschitz  for  $x\in{\cal
X}$ and
\begin{equation*}
  C_K(x)={C(x)\over K}.
\end{equation*}
\end{ann}

\me A biological interpretation of this renormalization is that larger
systems are made up of smaller individuals, which may be a
consequence of a fixed amount of available resources to be
partitioned among individuals.  Indeed, the biomass of each
interacting individual scales like $1/K$, which may imply that the
interaction effect of the global population on a focal individual
is of order $1$. The parameter $K$ may also be interpreted as scaling
the amount of resources available, so that the renormalization of $C_K$ reflects the decrease of competition for resources.

\me The generator $\tilde{L}^K$ of $(Y^K_t)_{t\geq 0}$ is given
by~(\ref{generator}), with parameters $C_K$,  $b_K$, $d_K$,
$m_K$, $p_K$. The generator $L^K$ of $(X^K_t)_{t\geq 0}$ is
obtained by writing, for any measurable function $\phi$ from
$M_F({\cal X})$ into $\rit$ and any $\nu\in M_F({\cal X})$,
\begin{equation*}
  L^K\phi(\nu)=\partial_t\E_{\nu}(\phi(X^K_t))_{t=0}= \partial_t
  \E_{K\nu}(\phi(Y^K_t/K))_{t=0} = \tilde{L}^K\phi^K(K\nu)
\end{equation*}
where $\phi^K(\mu)=\phi(\mu/K)$. Then we get
\begin{align}
  \label{eq:def-LK}
  L^K\phi(\nu)&=K\int_{\cal X} b_K(x)(1-p_K(x))(\phi(\nu+{1\over
    K}\delta_x)-\phi(\nu)) \nu(dx) \notag \\ &+K\int_{\cal X}\int_{\cal X}
    b_K(x)p_K(x)
    (\phi(\nu+{1\over
    K}\delta_{z})-\phi(\nu))\,m_K(x,z)dz \nu(dx)\notag \\ &+K\int_{\cal X}
  d_K(x,C*\nu(x))(\phi(\nu-{1\over
    K}\delta_x)-\phi(\nu)) \nu(dx).
\end{align}

\noindent By a similar proof as that carried out in Section~\ref{secexist}, we may
summarize the moment and martingale properties of $X^K$.
\begin{Prop}
  \label{XK}
  Assume that for some $p\geq 2$,  $\E(\langle X^K_0,1\rangle^p)<+\infty$.
  \begin{description}
  \item[\textmd{(1)}] For any $T>0$, $\E\big(\sup_{t\in[0,T]}\langle X^K_t,1\rangle^p\big)<+\infty$.
  \item[\textmd{(2)}] For any bounded and measurable function $\phi$
    on $M_F$ such that $|\phi(\nu)|+|L^K\phi(\nu)|\leq
    C(1+\langle \nu,1\rangle^p)$, the process
    $\ \phi(X^K_t)-\phi(X^K_0)-\int_0^tL^K\phi(X^K_s)ds\ $
    is a c\`adl\`ag martingale.
  \item[\textmd{(3)}] For each measurable bounded function $f$, the
    process
    \begin{align*}
      M^{K,f}_t & = \langle X^K_t,f\rangle
      -\langle X^K_0, f\rangle \\ &
      -\int_0^t\int_{\cal X}(b_K(x)-d_K(x,C*X^K_s(x)))f(x)
      X^K_s(dx)ds \\ &-\int_0^t\int_{\cal X}
      p_K(x)b_K(x)
      \bigg(\int_{\cal X}f(z)m_K(x,z)dz-f(x)\bigg)X^K_s(dx)ds
    \end{align*}
    is a square integrable martingale with quadratic variation
  \end{description}
% \vspace{-0.8cm}
  \begin{multline}
    \label{varquad}
    \langle M^{K,f}\rangle_t={1\over
      K}\bigg\{\int_0^t\int_{\cal X}p_K(x)b_K(x)
    \bigg(\int_{\cal X}f^2(z)\,m_K(x,z)dz- f^2(x)\bigg)X^K_s(dx)ds \\
    +\int_0^t\int_{\cal X}( b_K(x)+d_K(x,C*
    X^K_s(x)))f^2(x)X^K_s(dx)ds\bigg\}.
  \end{multline}
\end{Prop}

\noindent The search of tractable limits for the semimartingales $\langle
X^K,f\rangle$ yields  different choices of scalings for the
parameters.   In particular, we obtain a
 deterministic or stochastic  approximation, depending on the quadratic variation of the martingale term given
in~(\ref{varquad}).

\subsection{Large-population limit}
\label{sec:limit1}

\noindent We assume here that $b_K=b$, $d_K=d$, $p_K=p$, $m_K=m$. We also assume that ${\cal X}$ is a compact subset of $\mathbb{R}^d$ and we endow the measure space $M_F({\cal X})$ with the weak topology.

\begin{Thm}
  \label{largepop}
  Assume Assumptions \ref{hyp1} and \ref{hyp2} hold. Assume moreover  that $\ \sup_KE(\langle X^K_0,1\rangle^3)<+\infty$ and that the initial
  conditions $X^K_0$ converge in law and for the weak topology on
  $M_F({\cal X})$ as $K$ increases, to a finite deterministic measure
  $\xi_0$.

  \noindent Then for any $T>0$, the process $(X^K_t)_{t\geq 0}$ converges in
  law, in the Skorohod space $\dit([0,T],M_F({\cal X}))$, as $K$ goes
  to infinity, to the unique deterministic continuous function $\xi\in
  C([0,T],M_F({\cal X}))$ satisfying  for any bounded $f:{\cal
  X}\rightarrow\rit$
  \begin{align}
    \label{eq:limit1}
    \langle\xi_t,f\rangle & =\langle\xi_0,f\rangle
    +\int_0^t\int_{\cal
      X}f(x)[(1-p(x))b(x)-d(x,C*\xi_s(x))]\xi_s(dx)ds
    \notag \\ & +\int_0^t\int_{\cal X}p(x)b(x)
    \left(\int_{\cal X}f(z)\,m(x,z)dz\right)\xi_s(dx)ds
  \end{align}
\end{Thm}
\noindent The proof of Theorem~\ref{largepop} is left to the reader. It can
be adapted from the proofs of Theorem~\ref{readif}
and~\ref{readifstoch} below, or obtained as a generalization of
Theorem \ref{det-1}. This result is illustrated by the
simulations of Figs.~\ref{fig:IBM-1}~(a)--(c).\bigskip

\noindent {\bf Main Examples:}
\begin{description}
\item[(1) A density case.]   \begin{Prop}
  \label{densBP}
  Suppose  that
  $\xi_0$ is absolutely continuous with respect to  Lebesgue
  measure.  Then for all $t\geq 0$, $\,\xi_t$ is absolutely continuous
  where respect to  Lebesgue
  measure and $\,\xi_t(dx)=\xi_t(x)dx$, where $\,\xi_t(x)$
    is the solution of the functional
  equation
  \begin{align}
    \label{eq:EID}
    \partial_t\xi_t(x) & =\left[(1-p(x))b(x)
      -d(x,C*\xi_t(x))\right]\xi_t(x)
     +\int_{\rit^d}\,m(y,x) p(y) b(y)\xi_t(y)dy
  \end{align}
  for all $x\in{\cal X}$ and $t\geq 0$. 
   \end{Prop}
   
   \noindent Some people refer to $\xi_t(.)$ as the population
  number density.
  
 \begin{proof}
  Consider a Borel set $A$ of $\mathbb{R}^d$ with Lebesgue measure
  zero.  A
  simple computation allows us to obtain, for all $t\geq 0$,   \begin{align*}
    \langle\xi_{t}, \One_A \rangle &\leq
    \langle \xi_0,\One_A\rangle+ \bar{b} \ \int_0^{t} \int_{\cal X} \One_A(x)\xi_s(dx) ds 
    +\bar{b} \ \int_0^{t} \int_{\cal X}
   \int_{\cal X} \One_A(z) m(x,z) dz \xi_s(dx)ds.
  \end{align*}
  By assumption, the first term on the RHS is zero. The third term is
  also zero, since for any $x \in {\cal X}$, $\int_{\cal X}
  \One_A(z)m(x,z)dz=0$. Using Gronwall's Lemma, we conclude that $\langle \xi_{t},\One_A\rangle = 0$ and then, the measure $\xi_{t}(dx)$ is absolutely continuous w.r.t.  Lebesgue measure.
  One can easily prove from \eqref{eq:limit1} that  the density  function $\xi_t(.)$
   is solution of the functional equation \eqref{eq:EID}.
\end{proof}

\item[(2) The mean field case.]  The case of  structured
  populations with $d(x,C*\xi(x)) = d+ \alpha\, C*\xi(x)$ with constant rates $b$, $d$, $\alpha$ is meaningful, and has been developed in a spatial context where the kernel $C$ describes the resources available (see for example \cite{Ki99}).
  In this context,~(\ref{eq:limit1}) leads to the following equation on the total mass
  $\: n_t=\langle \xi_{t},1\rangle$:
  \begin{equation}
    \label{eq:phi=1}
    \partial_t n_t= (b-d)n_t -
    \alpha\int_{{\cal X}\times {\cal X}} C(x-y)\xi_t(dx)\xi_t(dy).
  \end{equation}
  This equation is not closed in $(n_{t})_{t}$ and presents an unresolved hierarchy of nonlinearities. 
  In the very particular case of uniform competition where $C\equiv 1$ (usually called "mean-field case"), there is a "decorrelative" effect and we recover the classical mean-field
  logistic equation of population growth:
  \begin{equation*}
    \partial_t n_t= (b-d)n_t - \alpha n_t^2.
  \end{equation*}
  
 \item[(3) Monomorphic and dimorphic cases with no mutation.]  We
  assume here that the  mutation probability is
  $p=0$. Without mutation,  the  trait space is fixed at time $0$.

  {\bf (a) Monomorphic case:} All the individuals have the same trait $x$.  Thus, we can write $X^K_0=n^K_0(x)\delta_x$, and then
  $X^K_t=n^K_t(x)\delta_x$ for any time $t$. In this case, Theorem~\ref{largepop}
  recasts  into $n^K_t(x)\rightarrow n_t(x)$ with
  $\xi_t=n_t(x)\delta_x$, and~(\ref{eq:limit1}) reads
  \begin{equation}
    \label{eq:monomorph}
    \frac{d}{dt}n_t(x)=n_t(x)\big(b(x)-d(x,C(0)n_t(x))\big),
  \end{equation}
When $d(x,C*\xi(x)) = d+ \alpha\, C*\xi(x)$, we recognize the logistic equation \eqref{log}. 

  {\bf (b) Dimorphic case:} The population consists in two subpopulations of individuals with traits $x$ and
  $y$, i.e.\  $X^K_0=n^K_0(x)\delta_x+n^K_0(y)\delta_y$ and when $K$ tends to infinity, the limit of $X^K_{t}$  is given by  $\xi_t=n_t(x)\delta_x+n_t(y)\delta_y$
  satisfying~(\ref{eq:limit1}), which recasts into the following system
  of coupled ordinary differential equations:
\end{description}
\begin{equation}
  \label{eq:dimorph}
  \begin{aligned}
    \frac{d}{dt}n_t(x) & \!=\!n_t(x)\big(b(x)
    \!-\!d(x,C(0)n_t(x)\!+\!C(x\!-\!y)n_t(y))\big) \\
    \frac{d}{dt}n_t(y) & \!=\!n_t(y)\big(b(y)
    \!-\!d(y,C(0)n_t(y)\!+\!C(y\!-\!x)n_t(x))\big).
  \end{aligned}
\end{equation}
When $d(x,C*\xi(x)) = d+ \alpha\, C*\xi(x)$, we obtain the competitive Lotka-Volterra system \eqref{LV}.

\subsection{Large-population limit with accelerated births and
deaths} \label{sec:accel}

\me We consider here an alternative limit of  large population,
combined with accelerated birth and death. This may  be useful to
investigate the qualitative differences of evolutionary dynamics
across  populations with allometric demographies (larger
populations made up of smaller individuals who reproduce and die
faster, see \cite{Calder:84}, \cite{Charnov:93}).

\noindent Here, we assume that ${\cal X}= \rit^d$.  Let us denote by $M_F$
the space $M_F(\rit^d)$.  It will be endowed with the weak topology (the test functions are continuous and bounded). We will 
also consider the vague topology (the test functions are compactly supported). We assume that the dominant terms of the birth
and death  rates are proportional to $K^{\eta}$ while
preserving the demographic balance. That is,
the birth and death rates satisfy

\begin{ann}\label{hyp3}
There exists  a bounded continuous function  $\gamma$ such that
\begin{equation*}
  b_K(x)=K^{\eta}\gamma(x)+b(x),\quad
  d_K(x,\zeta)=K^{\eta}\gamma(x)+d(x,\zeta).
\end{equation*}
The functions $b$ and $d$ are continuous and satisfy the properties stated in Assumptions \ref{hyp1} and \ref{hyp2}.
\end{ann}

\me
The allometric effect (smaller individuals reproduce and die faster)
is parametrized by a  positive and bounded function $\gamma$ and by  a real number $\eta\in (0,1]$. A detailed discussion of the
biological meaning of these parameters in terms of allometry and
life-history  can be found in~\cite{CFM06}. Observe that $\eta$ is a parameter scaling the speed of acceleration
of the birth and death rates when $K\rightarrow\infty$ (births and
deaths occur faster for larger $\eta$) and that  the individual growth rate $b_K-d_K$ stays bounded from
above. In other words, the timescale of population growth is assumed to be slower
than the timescale of individuals' births and deaths. As in
Section~\ref{sec:limit1}, the interaction kernel $C$ is
renormalized by $K$. Using similar arguments as in
Section~\ref{sec:limit1}, the process $X^K={1\over K}\,Y^K$ is now a
Markov process with generator
\begin{align*}
%   \label{eq:def-LK}
  L^K\phi(\nu)&=K\int_{\rit^d} (K^{\eta}
  \gamma(x)+b(x))(1-p_K(x))(\phi(\nu+{1\over
    K}\delta_x)-\phi(\nu)) \nu(dx) % \notag
  \\ & +K\int_{\rit^d} (K^{\eta}
  \gamma(x)+b(x))p_K(x)\int_{\rit^d}(\phi(\nu+{1\over
    K}\delta_{z})-\phi(\nu))\,m_K(x,z)dz \nu(dx) % \notag
  \\ & +K\int_{\rit^d}
  (K^{\eta} \gamma(x)+d(x,C*\nu(x)))(\phi(\nu-{1\over
    K}\delta_x)-\phi(\nu)) \nu(dx).
\end{align*}
As before, for any measurable functions $\phi$ on $M_F$ such that
$|\phi(\nu)|+|L^K\phi(\nu)|\leq C(1+\langle\nu,1\rangle^3)$, the
process
\begin{equation}
  \label{eq:mart-gal}
  \phi(X^K_t)-\phi(X^K_0)-\int_0^tL^K\phi(X^K_s)ds
\end{equation}
is a martingale. In particular, for each measurable bounded
function $f$, we obtain
\begin{align}
  & M^{K,f}_t= \langle X^K_t,f\rangle -
  \langle X^K_0,f\rangle \notag \\ &
  -\int_0^t\int_{\rit^d}(b(x)-d(x,C*X^K_s(x)))f(x)
  X^K_s(dx)ds \label{mart1} \\ & -\int_0^t\int_{\rit^d}
 p_K(x)(K^{\eta}\gamma(x)+b(x))
  \bigg(\int_{\rit^d}f(z)m_K(x,z)dz-f(x)\bigg)X^K_s(dx)ds \notag
\end{align}
is a square integrable martingale with quadratic variation
\begin{multline}
  \label{quadra1}
  \langle M^{K,f}\rangle_t={1\over K}\bigg\{
  \int_0^t\int_{\rit^d}(2K^{\eta}\gamma(x)+ b(x)+d(x,C*
  X^K_s(x)))f^2(x)X^K_s(dx)ds \\
  +\int_0^t\int_{\rit^d}p_K(x)(K^{\eta}\gamma(x)+b(x))
  \bigg(\int_{\rit^d}f^2(z)m_K(x,z)dz- f^2(x)\bigg)X^K_s(dx)ds
  \bigg\}.
\end{multline}

\me In what follows, the
variance of the mutation distribution  $\: m_K\:$ is of order $1/K^{\eta}$. That
will ensure that the deterministic part in~(\ref{mart1}) converges.
In the large-population renormalization
(Section~\ref{sec:limit1}), the quadratic variation of the
martingale part was of  order $1/K$. Here, it is of order
$K^{\eta}\times 1/K$. This quadratic variation will thus stay
finite provided that $\eta\in(0,1]$, in which case tractable
limits will result.  Moreover, this limit will be zero if $\eta<1$
and nonzero if $\eta=1$, which will lead to deterministic or
random limit models, as stated in the two following theorems.

%\subsubsection{Fast mutation and small mutation steps}
%\label{sec:limit2}

\me We assume here that the mutation rate is fixed, so that mutation events
appear more often as a consequence of accelerated births.
We assume\\

\begin{ann}
\label{hyp4}
\begin{description}

\item[\textmd{(1)}] $p_K=p$. \item[\textmd{(2)}] The mutation
step density $m_K(x,z)$ is the
  density of a random variable with mean $x$, variance-covariance
  matrix $\Sigma(x)/K^{\eta}$ (where
  $\Sigma(x)=(\Sigma_{ij}(x))_{1\leq i,j\leq d}$) and with  third centered
  moment of order $1/K^{\eta+\varepsilon}$ uniformly in $x$
  ($\varepsilon>0$). (Thus, as $K$ goes to infinity, mutant traits
  become more concentrated around their progenitors).
\item[\textmd{(3)}] $\sqrt{\Sigma}$ denotes the symmetric
square
  root matrix of $\Sigma$ and the function $\sqrt{\Sigma \,\gamma\,p}$ is
  Lipschitz  continuous and bounded.
\end{description}
\end{ann}

\me The main example is when the mutation density is  the
vector density of independent Gaussian variables with mean
$x$ and variance $\sigma^2(x)/K^{\eta}$:
\begin{equation}
  \label{eq:gauss}
  m_K(x,z)=\left(\frac{K^{\eta}}{2\pi\sigma^2(x)}\right)^{d/2}
  \exp\left(-K^{\eta}|z-x|^2/2\sigma^2(x)\right)
\end{equation}
where $\sigma^2(x)$ is positive and bounded over $\rit^d$.

\bi Then the convergence results of this section can be stated as
follows.
\begin{Thm}
  \label{readif}
  \begin{description}
  \item[\textmd{(1)}] Assume~ that Assumptions \ref{hyp1}--\ref{hyp4} hold and that $0<\eta<1$.
    Assume also that the initial conditions $X^K_0$ converge in law
    and for the weak topology on $M_F$ as $K$ increases, to a finite
    deterministic measure $\xi_0$ and that
    \begin{equation}
      \label{eq:X3}
      \sup_K \E(\langle X^K_0,1\rangle^3)<+\infty.
    \end{equation}
    Then, for each $T>0$, the sequence of processes $(X^K)$ belonging
    to $\dit([0,T],M_F)$ converges (in law) to the unique
    deterministic function $(\xi_t)_{t\geq 0} \in C([0,T],M_F)$
    satisfying: for each function $f\in C^2_b(\rit^d)$,
    \begin{align}
      \label{readif1}
      \langle\xi_t,f\rangle &
      =\langle \xi_0,f\rangle+ \int_0^t\int_{\rit^d}
      (b(x)-d(x,C*\xi_s(x)))f(x)\xi_s(dx)ds \notag \\ &
      +\int_0^t\int_{\rit^d}{1\over 2}p(x)\gamma(x)
      \sum_{1\leq i,j\leq d}\Sigma_{ij}(x)\partial^2_{ij}
      f(x)\xi_s(dx)ds,
    \end{align}
    where $\partial^2_{ij}f$ denotes the second-order partial
    derivative of $\ f$ with respect to $x_i$ and $x_j$
    ($x=(x_1,\ldots,x_d)$).
  \item[\textmd{(2)}] Assume moreover that there exists $c>0$ such
    that $\gamma(x)p(x)s^*\Sigma(x)s\geq c||s||^2$ for any $x$ and $s$ in
    $\rit^d$. Then for each $t>0$, the measure $\xi_t$ has a density
    with respect to Lebesgue measure.
  \end{description}
\end{Thm}

\noindent Observe that the limit~(\ref{readif1}) is independent of
$\eta\in(0,1)$. As will appear in the proof, this
comes from the fact that the growth rate $b_K-d_K$ is
independent of $\eta$ and that the mutation kernel $m_K(x,z)$
compensates exactly the dispersion in the trait space induced by the
acceleration of the births with mutations.

\begin{rema}
  In case (2), Eq.~(\ref {readif1}) may be written as
  \begin{equation}
    \label{readifde}
    \partial_t
    \xi_t(x)=\big(b(x)-d(x,C*\xi_t(x))\big)\xi_t(x) +
    {1\over 2}\sum_{1\leq i,j\leq d}
    \partial^2_{ij}(\gamma\, p \Sigma_{ij}\xi_t)(x).
  \end{equation}
  Observe that for the example~(\ref{eq:gauss}), this equation writes
  \begin{equation}
    \label{eq:reacdiff}
    \partial_t\xi_t(x)=\big(b(x)-d(x,C*\xi_t(x))\big)\xi_t(x)
    +{1\over 2}\Delta(\sigma^2\gamma\, p \,\xi_t)(x).
  \end{equation}
  Therefore, Eq.~(\ref{eq:reacdiff}) generalizes the Fisher
  reaction-diffusion equation known from classical population genetics
  (see e.g.~\cite{Bu00}).
\end{rema}

\begin{Thm}
  \label{readifstoch}
  Assume~that Assumptions \ref{hyp1}--\ref{hyp4} hold and that $\eta=1$. Assume also that the
  initial conditions $X^K_0$ converge in law and for the weak topology
  on $M_F$ as $K$ increases, to a finite (possibly random)
  measure $X_0$, and that
  $\: \sup_K \E(\langle X^K_0,1\rangle^3)<+\infty.$

\noindent Then, for each $T>0$, the sequence of processes $(X^K)$ converges in
  law in $\dit([0,T],M_F)$ to the unique (in law) continuous
  superprocess $X \in C([0,T],M_F)$, defined by the following
  equations :
  \begin{equation}
    \label{eca}
    \sup_{t \in [0,T]} \E\left( \langle X_t,1\rangle^3 \right) <\infty,
  \end{equation}
  and for any $f \in C^2_b(\mathbb{R}^d)$,
  \begin{align}
    \label{dfmf}
    \bar M^f_t&=\langle X_t,f\rangle - \langle X_0,f\rangle - \frac{1}{2}
    \int_0^t  \int_{\rit^d}  p(x)\gamma(x)
    \sum_{1\leq i,j\leq d}\Sigma_{ij}(x)\partial^2_{ij}
    f(x) X_s(dx)ds \notag
    \\ &- \int_0^t \int_{\rit^d} f(x)\left(
      b(x)-d(x,C*X_s(x))\right)X_s(dx)ds
  \end{align}
  is a continuous martingale with quadratic variation
  \begin{equation}
    \label{qvmf}
    \langle\bar M^f\rangle_t= 2\int_0^t \int_{\rit^d}\gamma(x) f^2(x) X_s(dx)ds.
  \end{equation}
\end{Thm}

\begin{rema}
  \begin{description}
  \item[\textmd{(1)}] The limiting measure-valued process $X$ appears
    as a generalization of the one proposed by Etheridge~\cite{Et04}
    to model spatially structured populations.
  \item[\textmd{(2)}] The conditions characterizing the  above process $X$
    can be formally rewritten as
    \begin{equation*}
%       \label{eq:EDPS}
      \partial_t X_t(x) =\bigg(b(x)-d(x,C*X_t(x))\bigg)X_t(x)
      +{1\over 2}\sum_{1\leq i,j\leq d}
      \partial^2_{ij}(\gamma\,p\,\Sigma_{ij}X_t)(x)+ \dot{M}_t
    \end{equation*}
    where $\dot{M}_t$ is a random fluctuation term, which reflects the
    demographic stochasticity of this fast birth and death process,
    that is, faster than the accelerated birth and death process which
    led to the deterministic reaction-diffusion
    approximation~(\ref{eq:reacdiff}).
  \item[\textmd{(3)}] As developed in Step~1 of the proof of
    Theorem~\ref{readifstoch} below, a Girsanov's theorem relates the law of
    $X_t$ and the one of a standard super-Brownian motion, which leads us to
    conjecture that a density for $X_t$ exists only when $d=1$, as
    for 
    super-Brownian motion.
  \end{description}
\end{rema}

These two theorems are illustrated by the simulations of
Figs.~\ref{fig:IBM-2}~(a) and~(b).

\paragraph{Proof of Theorem~\ref{readif}} {\bf (1)} We divide the proof in six
steps. Let us fix $T>0$.\bigskip

\noindent {\bf Step 1} \phantom{9} Let us first show the uniqueness of a
solution of the equation~(\ref{readif1}).

\noindent To this aim, we define the evolution equation associated
with~(\ref{readif1}). It is easy to prove that if $\xi$ is a
solution of~(\ref{readif1}) satisfying $\sup_{t\in [0,T]}
\langle\xi_t,1 \rangle <\infty$, then for each test function
$\psi_t(x)=\psi(t,x) \in C^{1,2}_b(\rit_+\times \rit^d)$, one has
\begin{align}
  \label{readif2}
  \langle\xi_t,\psi_t\rangle&=\langle \xi_0,\psi_0\rangle+
  \int_0^t\int_{\rit^d}
  (b(x)-d(x,C*\xi_s(x)))\psi(s,x)\xi_s(dx)ds \notag \\ &
  +\int_0^t\int_{\rit^d}(\partial_s \psi(s,x)+{1\over
    2}\gamma(x)p(x)\sum_{i,j}\Sigma_{ij}(x)\partial^2_{ij}
    \psi_s(x))\xi_s(dx)ds.
\end{align}
Now, since the function $\sqrt{\Sigma \,\gamma\,p}$ is Lipschitz
continuous, we may define the transition semigroup $(P_t)$ with
infinitesimal generator $\displaystyle{f\mapsto{\gamma\,p\over 2} \sum_{i,j}\Sigma_{ij}\partial^2_{ij}f}$.  Then, for each
function $f\in C^{2}_b(\rit^d)$ and fixed $t>0$, choosing
$\psi_s(x)=P_{t-s}f(x)$ yields
\begin{equation}
  \label{readif3} \langle
  \xi_t,f\rangle=\langle \xi_0,P_tf\rangle+ \int_0^t\int_{\rit^d}
  (b(x)-d(x,C*\xi_s(x)))P_{t-s}f(x)\xi_s(dx)ds,
\end{equation}
since $\partial_s \psi_s(x)+{1\over
  2}\gamma(x)p(x)\sum_{i,j}\Sigma_{ij}(x)\partial^2_{ij} \psi_s(x)=0$ for
this choice.

\me We now prove the uniqueness of a solution of~(\ref{readif3}).

\noindent Let us consider two solutions $(\xi_t)_{t\geq 0}$ and $(\bar
\xi_t)_{t\geq 0}$ of~(\ref{readif3}) satisfying $\sup_{t\in [0,T]}
\left< \xi_t+\bar \xi_t,1 \right> =A_T <+\infty$. We consider the
variation norm defined for $\mu_1$ and $\mu_2$ in $M_F$ by
\begin{equation}
  || \mu_1 - \mu_2 || = \sup_{f \in
    L^\infty({\rit^d}), \; | | f ||_\infty \leq 1}
  |\left<\mu_1-\mu_2,f\right>|.
\end{equation}
Then, we consider some bounded and measurable function $f$ defined
on ${\cal X}$ such that $|| f ||_\infty \leq 1$ and obtain
\begin{align}
  \label{lip}
  | \left<\xi_t-\bar \xi_t,f\right> |&\leq
  \intot \left| \int_{\rit^d} [\xi_s(dx) - \bar \xi_s(dx)]
    \left(b(x)-d(x,C*\xi_s(x))\right)P_{t-s}f(x)\right| ds
    \notag \\ &  + \intot \left|\int_{\rit^d} \bar
    \xi_s(dx)(d(x,C*\xi_s(x))-d(x,C*\bar{\xi}_s(x)))P_{t-s}f(x)\right|
    ds.
\end{align}

Since $|| f ||_\infty \leq 1$, then $|| P_{t-s}f ||_\infty \leq 1$
and for all $x\in \rit^d$,
\begin{equation*}
  \left|(b(x)-d(x,C*\xi_s(x)))P_{t-s}f(x)\right|\leq
  \bar{b}+\bar{d}(1+\bar{C}A_T).
\end{equation*}
Moreover,  $d$ is Lipschitz continuous in its second
variable with Lipschitz constant  $K_d$. Thus we obtain
from~(\ref{lip}) that
\begin{equation}
  | \left<\xi_t-\bar \xi_t,f\right> | \leq
  \left[\bar{b}+\bar{d}(1+\bar{C}A_T) +K_d A_T\bar{C}
  \right] \intot  || \xi_s - \bar\xi_s ||ds.
\end{equation}
Taking the supremum over all functions $f$ such that $| | f
||_\infty \leq 1$, and using  Gronwall's Lemma, we finally deduce
that for all $t\leq T$, $|| \xi_t - \bar\xi_t| | =0$. Uniqueness
holds.\bigskip

\me {\bf Step 2} \phantom{9} Next, we would like to obtain some moment
estimates.  First, we check that for any $T>0$,
\begin{equation}
  \label{me1}
  \sup_K \sup_{t \in [0,T]}  \E \big(\langle X^K_t,1\rangle^3
  \big)
  <\infty.
\end{equation}
To this end, we use~(\ref{eq:mart-gal}) with $\phi(\nu)=\langle
\nu,1\rangle^3$. (To be completely rigorous, one should first use
$\phi(\nu)=\langle \nu,1\rangle^3 \land A$ and  let $A$ tend to
infinity). Taking expectation, we obtain that for all $t\geq 0$,
all $K$,
\begin{align*}
  &\E \left(\langle X^K_t,1\rangle^3 \right) = \E \left(\langle
    X^K_0,1\rangle^3 \right) \\ &+ \intot  \E \bigg(\int_{\rit^d}
   \bigg([K^{\eta+1}\gamma(x) +K b(x)] \left\{[\langle
      X^K_s,1\rangle+ \frac{1}{K}]^3-\langle X^K_s,1\rangle^3 \right\}
  \\ &+\left\{K^{\eta+1}\gamma(x) + K d(x,C* X^K_s(x)) \right\}
    \left\{[\langle X^K_s,1\rangle- \frac{1}{K}]^3-\langle
      X^K_s,1\rangle^3 \right\}\bigg) X^K_s(dx) \bigg)ds.
\end{align*}
Dropping the non-positive death term involving $d$, we get
\begin{align*}
 \E \left(\langle X^K_t,1\rangle^3 \right) & \leq  \E \left(\langle
    X^K_0,1\rangle^3 \right)  \\
 & \quad  + \int_0^t
    \E \bigg(\int_{\rit^d}\bigg(
  K^{\eta+1}\gamma(x) \left\{[\langle X^K_s,1\rangle+ \frac{1}{K}]^3 +
    [\langle X^K_s,1\rangle- \frac{1}{K}]^3 -2\langle X^K_s,1\rangle^3
  \right\}  \\
  & \qquad \qquad  + Kb(x) \left\{[\langle X^K_s,1\rangle+
    \frac{1}{K}]^3-\langle X^K_s,1\rangle^3 \right\}\bigg)X^K_s(dx)
  \bigg)ds.
\end{align*}
But for all $x\geq 0$, all $\e \in (0,1]$, $(x+\e)^3 -x^3 \leq 6
\e (1+x^2)$ and $| (x+\e)^3+(x-\e)^3-2x^3 | = 6 \e^2 x$. We
finally obtain
\begin{equation*}
  \E \left(\langle X^K_t,1\rangle^3 \right) \leq \E \left(\langle
      X^K_0,1\rangle^3 \right)+ C \intot  \E\left(\langle
      X^K_s,1\rangle +\langle X^K_s,1\rangle^2  +\langle X^K_s,1\rangle^3
  \right)ds.
\end{equation*}
Assumption~(\ref{eq:X3}) and
 Gronwall's Lemma allow us to conclude that~(\ref{me1}) holds.\\
Next, we wish to check that
\begin{equation}
  \label{me2}
  \sup_K \E \big(\sup_{t \in [0,T]} \langle X^K_t,1\rangle^2\big)
  <\infty.
\end{equation}
Applying~(\ref{mart1}) with $f \equiv 1$, we obtain
\begin{equation*}
  \langle X^K_t,1\rangle = \langle X^K_0,1\rangle+ \intot  \int_{\cal X}
  \left( b(x)-d(x,C* X^K_s(x)) \right) X^K_s(dx)ds +
  M^{K,1}_t.
\end{equation*}
Hence
\begin{equation*}
  \sup_{s \in [0,t] }\langle X^K_s,1\rangle^2 \leq C\bigg(\langle
    X^K_0,1\rangle^2 + \bar{b} \intot  \langle X^K_s,1\rangle^2 ds +
  \sup_{s \in [0,t] } | M^{K,1}_s |^2 \bigg).
\end{equation*}
Thanks to~(\ref{eq:X3}), the Doob inequality and  Gronwall's
Lemma, there exists a constant $C_t$ not depending on $K$ such
that
\begin{equation*}
  \E \big(\sup_{s \in [0,t]} \langle X^K_s,1\rangle^2 \big) \leq
  C_t \left(1+ \E \left( \langle M^{K,1} \rangle_t \right) \right).
\end{equation*}
Using now~(\ref{quadra1}), we obtain, for some other constant
$C_t$ not depending on $K$,
\begin{equation*}
  \E \left( \langle M^{K,1} \rangle_t \right) \leq C \intot
  \big(\E\left( \langle X^K_s,1\rangle +
    \langle X^K_s,1\rangle^2 \right)\big)ds \leq C_t
\end{equation*}
thanks to~(\ref{me1}). This concludes the proof
of~(\ref{me2}).\bigskip

\me {\bf Step 3} \phantom{9} We first endow $M_F$ with the vague
topology, the extension to the weak topology being handled in Step
6 below. To show the tightness of the sequence of laws
$Q^K=\loi(X^K)$ in ${\cal
  P}(\dit([0,T],(M_F,v)))$, it suffices, following Roelly~\cite{Ro86}, to
show that for any continuous bounded function $f$ on $\rit^d$, the
sequence of laws of the processes $\langle X^K,f\rangle$ is tight
in $\dit([0,T], \rit)$. To this end, we use the Aldous
criterion~\cite{Al78} and the Rebolledo criterion
(see~\cite{JM86}). We have to show that
\begin{equation}
  \label{tight} \sup_K
  \E\big(\sup_{t\in [0,T]} | \langle X^K_t,f\rangle | \big)
  <\infty,
\end{equation}
and the Aldous condition  respectively for the  predictable
quadratic variation of the martingale part and
for the drift part of the semimartingales $\langle X^K,f\rangle$. \\
Since $f$ is bounded,~(\ref{tight}) is a consequence
of~(\ref{me2}): let us thus consider a couple $(S,S')$ of stopping
times satisfying a.s.~$0 \leq S \leq S' \leq S+\delta\leq T$.
Using~(\ref{quadra1}) and (\ref{me2}), we get for constants $C,
C'$
\begin{equation*}
  \E\left(\langle M^{K,f}\rangle_{S'}- \langle
    M^{K,f}\rangle_S\right) \leq C \E\left( \int_S^{S+\delta} \left(
      \langle X^K_s,1\rangle+ \langle X^K_s,1\rangle^2
    \right)ds\right)\leq C' \delta.
\end{equation*}
In a similar way,  the expectation of the finite variation part of
$\langle X^K_{S'},f\rangle - \langle X^K_S,f\rangle$ is bounded by
$C' \delta$.

\noindent Hence,  the sequence $(Q^K=\loi(X^K))$ is tight  in ${\cal
P}(\dit([0,T],(M_F,v)))$.

\me {\bf Step 4}  We want to get a convergence result with $M_{F}$ endowed with the weak topology. To this aim, as in   \cite{jourdain12} Lemmas 4.2 and 4.3,
 we use a sequence of functions  which will help to control the measures outside the compact sets.  

\begin{Lem}
\label{trick}
There exists a sequence   $(f_{n})_{n}$   of elements of $C^2_{b}(\mathbb{R}^d, [0,1])$ such that
$$f_{n}= 0 \ \hbox{ on }B(0,n-1) \ ;\  f_{n}= 1 \ \hbox{ on } B(0,n)^c,$$
where $B(0,r)$ is the euclidian ball centered in $0$ with radius $r$,
which satisfies
\be
\label{trick1}
\lim_{n\to \infty } \limsup_{K\to \infty} \E\left(\sup_{t\leq T}\langle X^K_{t}, f_{n}\rangle\right)= 0.
\ee
\end{Lem}

\noindent 
We refer to \cite{jourdain12} for the proof of Lemma \ref{trick}. 

\bi {\bf Step 5} \phantom{9} Let us now denote by $Q$ the weak limit  in ${\cal
P}(\dit([0,T],(M_F,v)))$
of a subsequence of $(Q^K)$ which we  also denote by
$(Q^K)$. Let $X=(X_t)_{t\geq 0}$ a process with law $Q$. We remark
that by construction, almost surely,
\begin{equation*}
  \sup_{t\in [0,T]}\ \sup_{f \in L^\infty(\rit^d), || f
    ||_\infty\leq 1} | \langle X^K_t,f \rangle -
  \langle X^K_{t^-},f \rangle | \leq 1/K.
\end{equation*}
Since, for each $f$ in a countable measure-determining set of continuous and compactly supported functions on $\rit$, the mapping $\nu\mapsto\sup_{t\leq T} |\langle \nu_t,f\rangle - \langle \nu_{t-},f\rangle |$ is continuous on $\mathbb{D}([0,T], (M_F,v))$, one deduces that $Q$ only charges the continuous processes  from $[0,T]$ into $(M_F,v)$. 
Let us now endow  $M_F$ with the weak convergence topology and check that $Q$ only charges the continuous processes from $[0,T]$ into $(M_F,w)$, and that the sequence $(Q^K)$ in ${\cal
P}(\dit([0,T],(M_F,w)))$ converges weakly to $Q$.
 For this purpose, we need to control the behavior of the total mass of the measures. We  employ  the sequence $(f_n)$ of smooth functions introduced in  Lemma  \ref{trick} which approximate the functions ${\bf 1}_{\{|x|\geq n\}}$.
For each  $n\in{\mathbb N}$, the continuous and compactly supported functions $(f_{n,l}:=f_n(1-f_l))_{l\in{\mathbb N}}$ increase to $f_n$, as $l\to\infty$. Continuity of the mapping $\nu\mapsto \sup_{t\leq T} \langle \nu_t,f_{n,l}\rangle$ on $\mathbb{D}([0,T], (M_F,v))$, and its    uniform integrability deduced from 
\eqref{me2}, imply the bound
$$
\E\left(\sup_{t\leq T}\langle X_t,f_{n,l}\rangle\right)=\lim_{K\to\infty}\E\left(\sup_{t\leq T}\langle X^K_t,f_{n,l}\rangle\right)\leq\liminf_{K\to\infty}\E\left(\sup_{t\leq T}\langle  X^K_t,f_{n}\rangle\right).
$$
Taking the limit, $l\to\infty$,  in the left-hand-side, in view of the monotone  convergence theorem and respectively \eqref{me2} and Lemma \ref{trick}, one concludes that   for $n=0$, 
\begin{equation}
 \E\left(\sup_{t\leq T}\langle X_t,1\rangle\right)=\E\left(\sup_{t\leq T}\langle X_t,f_0\rangle\right)<+\infty\label{contmass}
\end{equation} 
 and for general $n$,    \begin{equation}
 \lim_{n\to\infty}\E\left(\sup_{t\leq T}\langle X_t,f_n\rangle\right)=0.\label{tightlim}
\end{equation}%%%%%%%%%
As a consequence one may extract a subsequence of the sequence $(\sup_{t\leq T}\langle X_t,f_n\rangle)_n$ that converges a.s. to $0$ under $Q$, and the set of measures $(X_t)_{t\leq T}$ is tight  $Q$-a.s. Thus $(X_{s_n}
 )$ is relatively compact for any $s_n\rightarrow t$. Moreover the unique limiting measure is 
 $X_t$ since  $Q$ also only charges  the continuous processes  from $[0,T]$ into $(M_F,v)$.
One deduces that $Q$ also only charges  the continuous processes  from $[0,T]$ into $(M_F,w)$.

\noindent According to M\'el\'eard and Roelly~\cite{MR93}, to prove that the sequence $(Q^K)$ converges weakly to $Q$ in ${\cal
P}(\dit([0,T],(M_F,w)))$, it is sufficient to check that the processes $(\langle X^K,1\rangle=(\langle X^K_t,1\rangle)_{t\leq T})_K$ converge in law to $\langle X,1\rangle\stackrel{\rm def}{=}(\langle X_t,1\rangle)_{t\leq T}$ in $\dit([0,T],\rit)$.
For  a Lipschitz continuous and bounded function $F$ from  $\mathbb{D}([0,T],  \mathbb{R})$ 
to $\mathbb{R}$, we have
\begin{align*}
\limsup_{K\to\infty}|\E(F(\langle \nu^K,1\rangle)&- F(\langle X,1\rangle)| \leq \limsup_{n\to\infty}\, \limsup_{K\to\infty}|\E(F(\langle X^K,1\rangle)- F(\langle X^K,1-f_n\rangle))| \\
&+
\limsup_{n\to\infty}\limsup_{K\to\infty}|\E(F(\langle X^K,1-f_n\rangle)- F(\langle X,1-f_n\rangle))| \\
&+\limsup_{n\to\infty}|\E(F(\langle X,1-f_n\rangle)- F(\langle X,1\rangle))|. 
\end{align*}
Since $|F(\langle X,1-f_n\rangle)- F(\langle X ,1\rangle)|\leq C\sup_{t\leq T}\langle X_t,f_n\rangle$, Lemma \ref{trick} and  \eqref{tightlim} respectively imply that the first and the third terms in the r.h.s. are equal to $0$. The second term is $0$ in view of the  continuity of the mapping $\nu\mapsto \langle\nu,1-f_n\rangle$ in $\dit([0,T],(M_F,v))$.
%This implies that the process $X$ is a.s.~strongly continuous.
\bigskip

\me {\bf Step 6} \phantom{9} The time $T>0$ is fixed. Let us now check
that almost surely, the process $X$ is the unique solution
of~(\ref{readif1}).  Thanks to~(\ref{contmass}), it satisfies
$\sup_{t\in
  [0,T]} \langle X_t,1\rangle <+\infty$ a.s., for each $T$. We fix now
a function $f\in C^3_b(\rit^d)$ (the extension of~(\ref{readif1})
to
any function $f$ in $C^2_b$ is not hard) and some $t\leq T$.\\
For $\nu \in C([0,T],M_F)$, denote 
\begin{align}
  \Psi^1_t(\nu)&= \langle \nu_t,f\rangle - \langle\nu_0,f\rangle -
  \intot \int_{\rit^d}
  (b(x)-d(x,C*\nu_s(x)))f(x)\nu_s(dx)ds,\notag \\
  \Psi^2_t(\nu)&=-\intot \int_{\rit^d}  {1\over
    2}p(x)\gamma(x)\sum_{i,j}\Sigma_{ij}(x)\partial^2_{ij}f(x)\nu_s(dx) ds.
\end{align}

\me Our aim is to show that
\begin{equation}
  \label{wwhtp}
  \E \left( |\Psi^1_t(X)+\Psi^2_t(X) | \right)=0,
\end{equation}
implying that $X$ is solution of \eqref{readif2}.

\noindent By~(\ref{mart1}), we know that for each $K$,
\begin{equation*}
  M^{K,f}_t=\Psi^1_t(X^K)+\Psi^{2,K}_t(X^K),
\end{equation*}
where
\begin{multline}
  \Psi^{2,K}_t(X^K)=-\int_0^t\int_{\rit^d}
  p(x)(K^{\eta}\gamma(x)+b(x))\\ \times \bigg(\int_{\rit^d}
  f(z)m_K(x,z)dz-f(x)\bigg)X^K_s(dx)ds.
\end{multline}
Moreover,~(\ref{me2}) implies that for each $K$,
\begin{equation}
  \label{cqmke}
 \E \left( | M^{K,f}_t |^2 \right) = \E \left( \langle
      M^{K,f}\rangle_t \right)\leq \frac{C_{f}K^{\eta}}{K} \E\left(
    \intot  \left\{\langle X^K_s,1\rangle
      +\langle X^K_s,1\rangle^2\right\}ds \right) \leq
  \frac{C_{f,T}K^{\eta}}{K},
\end{equation}
which goes to $0$ as $K$ tends to infinity, since $0<\eta<1$.
Therefore,
\begin{equation*}
  \lim_K \E(|\Psi^1_t(X^K)+\Psi^{2,K}_t(X^K)|)=0.
\end{equation*}

\noindent Since $X$ is a.s.~strongly continuous (for the weak topology) and since $f\in C^3_b(\rit^d)$
and thanks to the continuity of the parameters, the functions
$\Psi^1_t$ and $\Psi^2_t$ are a.s.~continuous at $X$. Furthermore,
for any $ \nu \in \dit([0,T],M_F)$,
\begin{equation}
  |\Psi^1_t(\nu)+\Psi^2_t(\nu)| \leq C_{f} \int_0^T
  \left(1+\langle \nu_s,1\rangle^2 \right)ds.
\end{equation}
Hence using~(\ref{me1}), we see that the sequence
$(\Psi^1_t(X^K)+\Psi^2_t(X^K))_K$ is uniformly integrable, and
thus
\begin{eqnarray}
  \label{cqv1}
  \lim_K \E\left(|\Psi^1_t(X^K)+\Psi^2_t(X^K)|\right)
  = \E\left(|\Psi^1_t(X)+\Psi^2_t(X)|\right).
\end{eqnarray}

\noindent We have now to deal with $\Psi^{2,K}_t(X^K)-\Psi^2_t(X^K)$. The
convergence of this term is due to the fact that the measure
$m_K(x,z)dz$ has mean $x$, variance $\Sigma(x)/K^{\eta}$, and
third moment bounded by $C/K^{\eta+\varepsilon}$ ($\varepsilon>0$)
uniformly in $x$.  Indeed, if $Hf(x)$ denotes the Hessian matrix
of $f$ at $x$,
\begin{align}
  & \int_{\rit^d} f(z) m_K(x,z)dz \notag \\ &=\int_{\rit^d}
  \left(f(x)+(z-x)\cdot\nabla f(x) +\frac{1}{2} (z-x)^*
  Hf(x)(z-x)+O((z-x)^3)\right) m_K(x,z)dz \notag \\ &= f(x)+{1\over
    2}\sum_{i,j}\frac{\Sigma_{ij}(x)}{K^{\eta}}\partial^2_{ij}
  f(x)+{\it o}({1\over K^{\eta}}). \label{eq:M_K}
\end{align}
where $K^{\eta}{\it o}({1\over K^{\eta}})$ tends to $0$ uniformly
in $x$ (since $f$ is in $C^3_b$), as $K$ tends to infinity.  Then,
\begin{multline*}
  \Psi^{2,K}_t(X^K)=-\int_0^t\int_{\rit^d}
  p(x)(K^{\eta}\gamma(x)+b(x))\times \\ \times\bigg({1\over 2}
  \sum_{i,j}\frac{\Sigma_{ij}(x)}{K^{\eta}}\partial^2_{ij}
  f(x)+{\it o}({1\over K^{\eta}})\bigg)X^K_s(dx)ds,
\end{multline*}
and
\begin{equation*}
  |\Psi^{2,K}_t(X^K)-\Psi^2_t(X^K)| \leq C_f
  \big(\sup_{s\leq T}<X^K_s,1>\big)\bigg({1\over K^{\eta}}+K^{\eta}{\it o}({1\over
    K^{\eta}})\bigg).
\end{equation*}
Using~(\ref{me2}), we conclude the proof of~(\ref{wwhtp}) and Theorem~\ref{readif}  {\bf (1)}.

\bi Let us now prove the point  {\bf (2)}. The non-degeneracy property
$\gamma(x)p(x)s^*\Sigma(x)s\geq c\|s\|^2$ for each $x,s\in \mathbb{R}^d$ implies that for each time $t>0$, the transition
semigroup $P_t(x,dy)$ introduced in Step 1 of this proof has for
each $x$ a density function $p_t(x,y)$ with respect to 
Lebesgue measure. Then if we come back to the evolution
equation~(\ref{readif3}), we can write
\begin{multline*}
  \int_{\rit^d} f(x) \xi_t(dx) =\int_{\rit^d}
  \left(\int_{\rit^d}f(y)p_t(x,y)dy\right) \xi_0(dx)\\ +
  \int_0^t\int_{\rit^d}
  (b(x)-d(x,C*\xi_s(x)))\bigg(\int_{\rit^d}
  f(y)p_{t-s}(x,y)dy\bigg)\xi_s(dx)ds.
\end{multline*}

\noindent Using the fact that the parameters are bounded, that $\
\sup_{t\leq
  T}\langle\xi_t,1\rangle<+\infty\ $ and that $\ f\ $ is bounded, we can
apply Fubini's theorem and deduce that
\begin{equation*}
  \int_{\rit^d} f(x) \xi_t(dx)=\int_{\rit^d}H_t(y)f(y)dy
\end{equation*}
with $H\in L^{\infty}([0,T],L^1(\rit^d))$, which implies that
$\xi_t$ has a density with respect to  Lebesgue measure for
each time $t\leq T$.

\noindent Equation~(\ref{readifde}) is then the dual form of
(\ref{readif1}).\hfill$\Box$
\bigskip

\paragraph{Proof of Theorem~\ref{readifstoch}}
We  use a similar method as the one of the previous theorem.
Steps~2, 3, 4 and~6 of this proof can be achieved exactly in the
same way. Therefore, we only have to prove the uniqueness (in law)
of the solution to the martingale
problem~(\ref{eca})--(\ref{qvmf}) (Step~1), and that any
accumulation point of the sequence of laws of $X^K$ is solution
to~(\ref{eca})--(\ref{qvmf}) (Step~6).\bigskip

\me {\bf Step 1} \phantom{9} This uniqueness result is well-known for
the super-Brownian process (defined by a similar martingale
problem, but with $b=d=0$, $\gamma=p=1$ and $\Sigma=\mbox{Id}$, cf.
\cite{Ro86}). Following ~\cite{Et04}, we may use the version of
Dawson's Girsanov transform obtained in Evans and
Perkins~\cite{EP94} (Theorem~2.3), to deduce the uniqueness in our
situation, provided the condition
\begin{equation*}
  \E\left(\intot \intrd
    [b(x)-d(x,C*X_s(x))]^2X_s(dx) ds\right)<+\infty
\end{equation*}
is satisfied. This is easily obtained from the assumption that
$\sup_{t\in [0,T]} \E[\langle X_t,1\rangle^3]<\infty$ since the
coefficients are bounded.\bigskip

\me {\bf Step 6} \phantom{9} Let us now identify the limit. Let us call
$Q^K=\loi(X^K)$ and denote by $Q$ a limiting value of the tight
sequence $Q^K$, and by $X=(X_t)_{t\geq 0}$ a process with law $Q$.
Because of Step~5, $X$ belongs a.s.~to $C([0,T],(M_F,w))$. We want to
show that $X$ satisfies the conditions~(\ref{eca}), (\ref{dfmf})
and~(\ref{qvmf}). First note that~(\ref{eca}) is straightforward
from~(\ref{me2}). Then, we show that for any function $f$ in
$C^3_b(\mathbb{R}^d)$, the process $\bar M^f_t$ defined
by~(\ref{dfmf}) is a martingale (the extension to every function
in $C^2_b$ is not hard). We consider $0\leq s_1\leq...\leq
s_n<s<t$, some continuous bounded maps $\phi_1,...\phi_n$ on
$M_F$, and our aim is to prove that, if the function $\Psi$ from
$\mathbb{D}([0,T],M_F)$ into $\mathbb{R}$ is defined by
\begin{align}
  &\Psi(\nu) = \phi_1(\nu_{s_1})...\phi_n(\nu_{s_n}) \Big\{
  \langle \nu_t,f\rangle -\langle \nu_s,f\rangle
  \notag \\ & -\! \int_s^t\! \intrd\! \bigg({1\over 2}p(x)
  \gamma(x)\sum_{i,j}\Sigma_{ij}\partial^2_{ij} f(x)+ f(x)
  \left[b(x)-d(x,C*\nu_u(x)) \right]\bigg)\nu_u(dx)du
  \Big\},
\end{align}
then
\begin{equation}
  \label{cqfd44}
  \E\left( |\Psi(X) |\right)=0.
\end{equation}
It follows from~(\ref{mart1}) that
\begin{eqnarray}
  \label{262626}
  0=\E \left(\phi_1(X^K_{s_1})...\phi_n(X^K_{s_n})\left\{ M^{K,f}_t -
    M^{K,f}_s \right\}\right)=\E \left( \Psi(X^K) \right) - A_K,
\end{eqnarray}
where $A_K$ is defined by
\begin{multline*}
  A_K=\E \Big( \phi_1(X^K_{s_1})...\phi_n(X^K_{s_n})\int_s^t \intrd
    p(x) \Big\{ b(x)\Big[\intrd  (f(z)- f(x))m_K(x,z)dz\Big]\\
    +\gamma(x) K\Big[ \intrd  (f(z) - f(x)
    - \sum_{i,j}{\Sigma_{ij}(x)\over
    2K}\partial^2_{ij}f(x))m_K(x,z)dz\Big]
  \Big\}X^K_u(dx)du\Big).
\end{multline*}
Using~(\ref{eq:M_K}), we see  that $A_K$ tends to zero  as $K$
grows to infinity and using~(\ref{me2}), that the sequence
$(|\Psi(X^K)|)_K$ is uniformly integrable, so
\begin{equation}
  \label{333}
  \lim_K \E\left(|\Psi (X^K)|\right) = \E\left(|\Psi(X)| \right).
\end{equation}
Collecting the previous results allows us to conclude
that~(\ref{cqfd44}) holds, and thus $\bar M^f$ is a martingale.\\
We finally have to show that the bracket of $\bar M^f$ is given
by~(\ref{qvmf}). To this end, we first check that
\begin{align}
  \label{lfalqoc}
  \bar N^{f}_t &= \langle X_t,f \rangle^2 - \langle X_0,f \rangle^2
  -\intot  \intrd 2 \gamma(x)f^2(x) X_s(dx)ds \notag \\
  &- 2\intot \langle X_s,f \rangle  \intrd  f(x)
  \left[b(x)-d(x,C*X_s(x)) \right]X_s(dx) ds \notag \\
  &- \intot   \langle X_s,f \rangle  \intrd
  p(x)\gamma(x)\sum_{i,j}\Sigma_{ij}(x)\partial^2_{ij}f(x) X_s(dx) ds
\end{align}
is a martingale. This can be done exactly as for $\bar M^f_t$,
 using the semimartingale
decomposition of $\langle X^K_t,f \rangle^2$, given by
(\ref{eq:mart-gal}) with $\phi(\nu)=\langle\nu,f\rangle^2$.
  On the other hand, It\^o's formula implies that
\begin{multline*}
  \langle X_t,f \rangle^2 - \langle X_0,f \rangle^2 - \langle \bar
  M^f\rangle_t
  - \intot  \langle X_s,f \rangle \intrd
  \gamma(x)p(x)\sum_{i,j}\Sigma_{ij}(x)\partial^2_{ij}f(x)X_s(dx) ds \\
  - 2\intot  \langle X_s,f \rangle  \intrd  f(x)
  \big[b(x)-d(x,C*X_s(x)) \big]X_s(dx) ds
\end{multline*}
is a martingale. Comparing this formula with~(\ref{lfalqoc}), we
obtain (\ref{qvmf}).\hfill$\Box$
\bigskip

%\marginpar{Deux situations : branchement \`a l'int\'erieur de branchement (ou il y a instabilit\'e) et un processus ergodique sur un arbre  (general, plusieurs enfants possibles).}
\section{Splitting  Feller Diffusion for Cell Division with Parasite Infection}
\label{div}
 \me We now deal  with  a continuous time model for dividing cells which are infected by parasites. We assume that parasites proliferate in the cells and that their life times are much shorter than the cell life times. The  quantity of parasites $(X_t : t\geq 0)$ in a cell is modeled by a  Feller diffusion (See  Sections \ref{LPA} and \ref{CSBP}). The cells divide in continuous time at  rate $\tau(x)$ which may depend on the quantity of parasites $x$ that they contain. When a cell divides, a random fraction $F$ of the parasites goes in 
 the first daughter cell and a fraction $(1-F)$ in the second one. 
 More generally,  splitting  Feller diffusion may model  the quantity of some biological content which grows (without ressource limitation) in the cells and is shared randomly when the cells divide (for example  proteins, nutriments, energy or extrachromosomal rDNA circles in yeast). 
 
 \me   Let us give some details about the biological motivations.
 The modeling of  parasites sharing is inspired by experiments conducted in Tamara's Laboratory where bacteria E-Coli have been infected with bacteriophages. These experiments show that 
 a heavily infected cell often shares in a heavily infected cell and a lightly infected  cell. Thus we are interested in taking into account unequal parasite splitting  and we do not make restrictive 
 (symmetry) assumptions about the sharing of parasites. 
 We aim at quantifying the role of asymmetry in the infection. Without loss of generality, we assume that $F$ is distributed as $1-F$ and we say that the sharing is asymmetric when its distribution is not closely concentrated  around $1/2$. \\
This splitting diffusion is a "branching within branching" process, in the same vein as 
the multilevel model for plasmids considered by Kimmel \cite{kimmel}. In the latter model,
the cells divide in continuous time at a constant rate and the number of parasites is a discrete quantity which is fixed at the birth of the cell: the parasites reproduce 'only when the cells divide'. Moreover the parasites sharing is symmetric. 

\me Let us first describe briefly our process.
%In  \cite{kim}, a discrete time model where the sharing of the parasites may be asymmetric is considered. 
%\marginpar{ref cours sylvie, existence et unicite}
We denote by  $\mathfrak{I}:=\cup_{n\geq 0} \{1,2\}^n$   the usual labeling of a binary tree and by   $n(\ud i)$  the counting measure on $\mathfrak{I}$. We  define $V_t\subset \mathfrak{I}$ as the set of cells alive at time $t$ and $N_{t}$ the number of cells alive at time $t$: $N_{t}= \# V_{t}$.
  For $i\in V_t$, we denote by  $X^i_t\in \R_+$ the quantity of parasites in the cell $i$ at time $t$. \\
The population of cells at time $t$ including their parasite loads is modeled  by  the random point measure on $\R_+$ :
\begin{equation}Z_t(\ud x)=\sum_{i\in V_t}\delta_{X_t^i}(\ud x),\label{defz}
\end{equation}
and the dynamics of $Z$ is described as follows.
\begin{enumerate}
\item A cell with load $x$ of parasites  divides into two daughters at rate $\tau(x)$, where for some $p\geq 1$ and any
$x \geq0$,  $\tau(x)\leq \bar{\tau}(1+x^p)$.
\item During the division, the parasites are shared between the two daughters: $Fx$ parasites  in one cell (chosen at random) and $(1-F)x$.
\item Between two divisions, the quantity of parasites in a cell follows a Feller diffusion process (see \eqref{Feller}), with  diffusion coefficient $\,\sqrt{2\gamma x}$ and drift coefficient $\, r x$, $r$ and $\gamma$ being two real numbers, $\gamma>0$.
\end{enumerate}
Let us give a pathwise representation of the Markov process $(Z_{t},t\geq 0)$. 
 Let $(B^i,i\in \mathfrak{I})$ be a family of independent Brownian motions (BMs) and let $N(\ud s,\ud u,\ud i,\ud \theta)$ be a Poisson point measure (PPM) on $\R_+\times \R_+\times \mathfrak{I}\times [0,1]$ with intensity $q(\ud s,\ud v,\ud i,\ud \theta)=\ud s\, \ud v\,n(\ud i)\,\P(F \in \ud \theta)$ independent of the BMs.  We denote by $(\mathcal{F}_t : t\geq 0)$ the canonical filtration associated with the BMs and the PPM. Then, 
%,  \textit{e.g.} \cite{joffemetivier}). \\
% We follow in this the inspiration of \cite{fourniermeleard, bansayedelmasmarsalletran}.\\
for every $(t,x)\mapsto f(t,x)\in C^{1,2}_b(\R_+\times \R_+,\R)$ (the space of bounded functions of class $C^1$ in $t$ and $C^2$ in $x$ with bounded derivatives),
 \begin{align}
\langle Z_t,f\rangle = & f(0,x_0) +\int_0^t\int_{\R_+}\left(\partial_s f(s,x)+rx\partial_x f(s,x)+\gamma x \partial^2_{xx} f(s,x)\right)
Z_s(\ud x)\,\ud s \nonumber\\
 +&  M^f_t + \int_0^t \int_{\R_+\times \mathfrak{I}\times [0,1]} \ind_{i\in V_{s_-},\, u\leq \tau (X^i_{s_-})}\Big(f(s,\theta X^i_{s_-})\nonumber\\
 & \quad \qquad \qquad \qquad +f(s,(1-\theta)X^i_{s_-})-f(s,X^i_{s_-})\Big)N(\ud s,\ud u,\ud i,\ud \theta), \label{martingalegdepopp}
\end{align}
where $x_0$ is the load of parasites in the ancestor cell $\emptyset$ at $t=0$ and
 \begin{align}
M^f_t=\int_0^t \sum_{i\in V_s}\sqrt{2 \gamma X^i_s}\partial_x f(s,X^i_s)dB^i_s\end{align}is a continuous square integrable martingale with quadratic variation:
\begin{align}
\langle M^{f}\rangle_t= & \int_0^t\int_{\R_+}2\gamma x (\partial_x f(s,x))^{2}\,ds\,Z_s(\ud x).\label{crochetmartingalegdepopp}
\end{align}

\begin{rema} The existence and uniqueness of a solution of  \eqref{martingalegdepopp} are obtained from an adaptation of Subsection 5.3.  See also the next subsection for an approximation proof 
of the existence and \cite{MR2754402} for details. \\
Notice that between two jumps, $\langle Z_t,f\rangle = \sum_{i\in V_t}f(X_t^i)$ and It\^o's formula explains the second and third terms of \eqref{martingalegdepopp}. The fourth term (driven by the PPM) models the division  events with the random sharing of parasites. 
\end{rema}

\begin{Prop}
The total quantity of parasites $X_t=\int_{\R_+}x \,Z_t(\ud x)$ is a  Feller diffusion  (defined in \eqref{Feller}) with drift $r x$ and diffusion coefficient $\sqrt{2 \gamma x}$ starting from $x_{0}$. As a consequence,
\be
\forall t\in \R_+,\, \E_{x_0}(X_t)= x_0 e^{rt}<+\infty\quad, \quad \P_{x_0}(\exists t\geq 0 : \ X_t =0)=\exp(-rx_0/\gamma).\label{equationmomentfeller}
\ee
 %Moreover, for $f\in C^{1,2}_b(\R_+^2,\R)$, $\E\Big(\int_0^t \int_{\R_+}x\,Z_s(\ud x) ds\Big)<+\infty$  and thus  $M^f_{t}$ is square integrable for any $t\in \R_+$.
\end{Prop}

\begin{proof}
We remark that $X_{t}$ can be written
$\, X_{t}= x_{0}+ \int_{0}^t
 r X_{s}  ds + M_{t}$ where $M$ is a continuous square integrable martingale with quadratic variation $\int_{0}^t 2 \gamma X_{s} ds$. The representation theorem explained in the proof of Theorem \ref{Thlimstoc}   allows us to conclude. The properties \eqref{equationmomentfeller} follow by classical arguments.
  \end{proof}

 \subsection{Approximation by scaling limit}

\me
Inspired by the previous sections, we are looking for discrete approximations of the continuous model defined in \eqref{martingalegdepopp}, where each cell hosts a discrete parasite population. Let us introduce, as previously, the scaling parameter $K$. Let us describe the approximating model. The initial cell contains $[K x_0]$ parasites. %, where $N\in \N^*=\{1,2,\dots\}$ is a parameter that will tend to infinity and where $[x]$ denotes the integer part of $x$. 
%The parasites are reweighted by $1/N$ so that the biomass in a cell remains of constant order. 
%\marginpar{Parler du cas stable obtenu avec des queues lourdes ?}
The parasites  reproduce asexually with the individual birth and death rates
$K\gamma+\lambda, \ 
K\gamma+\mu$, where $\lambda,\mu>0$ satisfy  $\lambda-\mu=r>0$. The cell population is fully described by the point measure
$$\bar{Z}^K_t(\ud u, \ud x)=\sum_{i\in V_t}\delta_{(i,X_t^{K,i})}(\ud u, \ud x)$$
%\marginpar{representation differente, soulignÃƒÂ©. Noter $u$ plutot que $i$ ? }
where $X_t^{K,i}$ is the number of parasites renormalized by $K$ in the cell $i$ at time $t$. This representation allows to keep a record of  the underlying genealogy, which is useful in the forthcoming proofs. It also  provides a closed equation to characterize the process $\bar{Z}$ (and derive $Z$). Notice that an alternative representation has been given in  Section \ref{sec:large-popu}
 by ordering the atoms  $X_t^{K,i}$ for $i\in V_t$. \\

\me Let $N^0$ and $N^1$ be two independent PPMs on $\R_+\times \mathcal X_0:=\R_+\times \mathfrak{I}\times \R_+$ and $\R_+\times \mathcal X_1:=\R_+\times \mathfrak{I}\times \R_+\times [0,1]$ with intensity measures $\ud s n(\ud i)\ud u$ and  $\ud s \,n(\ud i)\ud u \,\P(F\in \ud \theta)$. We associate $N^1$ to 
 to the births and deaths of parasites, while $N^2$ corresponds to  the cell divisions. The discrete space process is the unique strong solution of
\begin{align}
& \bar{Z}^K_t=  \delta_{(\emptyset,[Kx_0]/K)} \label{eds} \\
&+  \int_0^t \int_{\mathcal X_0} N^0(\ud s,\ud i, \ud u)\ind_{i\in V_{s_-}}\left[\left(\delta_{(i,X_{s_-}^{K,i}+1/K)}-\delta_{(i,X_{s_-}^{K,i})}\right)\ind_{u\leq  \lambda_K
 X^{K,i}_{s_-}}\right.\nonumber\\
  &\qquad \qquad \qquad \qquad \qquad +  \left. \left(\delta_{(i,X_{s_-}^{K,i}-1/K)}-\delta_{(i,X_{s_-}^{K,i})}\right)\ind_{\lambda_K X^{K,i}_{s_-}< u
   \leq (\lambda_K+\mu_K)X^{K,i}_{s_-}}\right].\nonumber\\
&+ \int_0^t \int_{\mathcal X_1} N^1(\ud s,\ud i,\ud u,\ud \theta)\ind_{i\in V_{s_-}}\ind_{u\leq \tau(X_{s_-}^{K,i})}\Big(\delta_{(i1,[\theta K X_{s_-}^{K,i}]/K)}\nonumber\\
&  \qquad \qquad \qquad \qquad \qquad \qquad \qquad \qquad \qquad +\delta_{(i2, X_{s_-}^{K,i}-[\theta K X_{s_-}^{K,i}]/K)}-\delta_{(i,X_{s_-}^{K,i})}\Big),\nonumber
\end{align}
where we set $\lambda_K=K(\lambda+K\gamma), \mu_K=K(\mu+K\gamma)$. 
We recall from Sections \ref{scaling} and \ref{scaling2} that other  discrete models would lead to the same continuous limiting object. For example, 
the parasites could be shared following a binomial distribution whose parameter is picked according to $\P(F \in \ud \theta)$.
%parasites may have several offspring and could be shared according to a random binomial distribution: we could draw $\theta$ in the distribution $K(\ud \theta)$ and then send each parasite in the first daughter cell with the probability $\theta$, or else, keep it for the second daughter cell (see \cite{kim}).\\
%We assume that the division rate $r(x)$ is upper bounded by $\bar{r}>0$ \emph{(OBLIGE ???)}, the size of the cell population is upper bounded by a Yule process of rate $\bar{r}$.
%The parasite population is a continuous time birth and death process of constant rates $\lambda_N$ and $\mu_N$. Hence for every $n\in \N^*$, there is existence and strong uniqueness of the solution of (\ref{eds}) for a given initial condition $Z^N_0$ and PPMs $Q^1$ and $Q^2$.
\begin{Prop}\label{propconvergence}
Assume that  there exists an integer $p\geq 1$ and a positive $\bar{\tau}>0$ such that for all $x\in \R_+$, $0\leq \tau(x)\leq \bar{\tau}(1+x^p)$.
Then, the sequence $(Z^K : n\in \N^*)$ defined in (\ref{eds}) converges in distribution in
$\mathbb{D}(\R_+,\mathcal{M}_F(\R_+))$  as  $K\rightarrow +\infty$ to the process
$Z$ defined in (\ref{martingalegdepopp})-(\ref{crochetmartingalegdepopp}).
\end{Prop}
\noindent The proof can be found in the Appendix  of \cite{MR2754402}.  It follows the scheme of proof developed in the previous section, namely  control of the moments and  tightness and 
 identification of the limit via the martingale problem for $<\bar{Z}_t^K,f>$. The additional regularities on the division rate $\tau$ are required to control the difference between the microscopic process (\ref{eds}) and its approximation (\ref{martingalegdepopp})-(\ref{crochetmartingalegdepopp}). \\
 
 \begin{exo}
Write the measure-valued equation which characterizes the limiting process 
 $\bar{Z}_t(du,dx)=\sum_{i\in V_t}\delta_{(i,X_t^i)}(du,dx)$. 
 %The population of cells at time $t$ may be represented by the random point measure on $\mathfrak{I}\times \R_+$,
 %We define by\begin{equation}Z_t(\ud x)=\sum_{i\in V_t}\delta_{X_t^i}(\ud x),\label{defz}
%\end{equation}the marginal measure of $\bar{Z}_t(du,dx)$ on the state space $\R_+$.
\end{exo}

 \subsection{Recovery criterion when the  division rate is constant}

\me  We now consider the case where the infection does not influence the  division rate and $\tau(.)=\tau$. We say that the organism recovers when the proportion of infected cells becomes negligible compared to the population of cells. 
\label{BFDpara}
 \begin{Thm}\label{threcoveryorganismrconstant}
%\textbf{Recovery of the organism with constant division rate}\\
(i) If $r\leq  2\tau\E(\log(1/F))$, then the organism recovers a.s.:
$$ \lim_{t\rightarrow +\infty}\frac{\# \{ i \in V_t : X^i_t>0\}}{N_t} = 0\quad \mbox{ a.s}.$$
(ii) If $r>  2\tau\E(\log(1/F))$ then the parasites proliferate in the cells as soon as  the parasites do not become
extinct in the sense that
\begin{equation}
\big\{\limsup_{t\rightarrow +\infty} \frac{\# \{ i \in V_t : X^i_t\geq e^{\kappa t} \}}{N_t} > 0 \big\}=\{\forall  t>0: \ X_t>0\} \quad \mbox{a.s.}\label{recoveryrctcroissexp}
\end{equation}
for every  $\kappa<r- 2\tau \E(\log(1/\Theta))$. The probability of this event is $\,1-\exp(-rx_0/\gamma)$.
\end{Thm}
%\marginpar{Picture ?}

\me The factor $2$ in the criterion comes from a bias phenomenon in continuous time. It appears in the following result shedding light on an auxiliary Markov process, namely
 the infection process $X$ with catastrophes occurring at the accelerated rate $2\tau$.  This factor $2$ 'increases the probability of recovery' in the sense that
the amount of parasites in a random cell lineage (which can be obtained by keeping one cell at random at each division)  may go to infinity with positive probability whereas the organism recovers a.s. \\ \\
Since the division rate is constant, the process  $(N_t, t\geq 0)$ is a simple linear (branching) birth process (called Yule process). Then $\E(N_t)=\exp(\tau t)$ and we define
%For the recovery criterion of the organism, we are interested in the asymptotic behavior of
%\begin{equation}\mu_t=Z_t(dx)/N_t.\label{def:mut}\end{equation}
%As usual for branching Markov process, it is more convenient to consider the following renormalization :
\begin{equation*}
\gamma_t(dx) := \E(Z_t(dx))/\E(N_t)=e^{-\tau t}\E(Z_t(dx)).\label{defgammat}
\end{equation*}
The evolution of $\gamma_t$ is given  by the following result.
\begin{Lem}\label{propprocessusauxiliaire2}The family of probability measures $(\gamma_t, t\geq 0)$ is the unique solution of the following equation: for $f\in C^{1,2}_b(\R_+^2,\R)$ and $t\in \R_+$ (and $f_{t}(.)=f(t,.)$):
\begin{multline}
\langle \nu_t,f_t\rangle
= f_0(x_0) 
  + \int_0^t \int_{\R_+} \left(\partial_sf_s(x)+rx \partial_x f_s(x)+\gamma x \partial^2_{xx}f_s(x)\right) \nu_s(\ud x)\ud s. \\
  + 2\tau \int_0^t \int_{\R_+}\int_0^1 \left[f_s(\theta x)-f_s(x)\right] \P(F \in \ud \theta)\nu_s(\ud x)\ud s \qquad  \label{nugaltonwatson}
\end{multline}
\end{Lem}
\begin{proof}[Proof of Lemma  \ref{propprocessusauxiliaire2}] Let $t\in \R_+$ and $(f\, : \, (s,x)\mapsto f_s(x))\in C_b^{1,2}(\R_+^2,\R)$.
Using (\ref{martingalegdepopp}) with $(s,x)\mapsto f_s(x)e^{-\tau s}$ entails:
\begin{align*}
&\langle e^{-\tau t}Z_t(\ud x), f_t\rangle  =\langle Z_t(\ud x), e^{-\tau t}f_t\rangle \\
&\ \ = f_0(x_0)
 +  \int_0^t \int_{\R_+} \Big(rx \partial_x f_s(x)+\gamma x \partial^2_{xx}f_s(x)
-\tau f_s(x)+\partial_s f_s(x)\Big)e^{-\tau s} Z_s(\ud x)\, \ud s \\
 &\ \quad +  \int_0^t \int_{\R_+\times \mathfrak{I}\times [0,1]}\ind_{i\in V_{s_-}}\ind_{u\leq \tau}  \Big[f_s(\theta X^i_{s-})+f_s((1-\theta)X^i_{s-})
  -f_s(X^i_{s-})\Big]e^{-\tau s} N(\ud s,\ud u,\ud i,\ud \theta)\\
& \ \quad +M_t^f
\end{align*}where $M_t^f$ is a continuous square integrable martingale started at $0$. Taking the expectation and using 
the symmetry of the distribution of $F$ with respect to $1/2$:
\begin{align}
\langle \gamma_t, f_t\rangle %= & \E(\langle Z_0, f_0\rangle) +\int_0^t \int_{\R_+}\int_0^1 r\left[f_s(\theta x)+f_s((1-\theta)x)-2f_s(x)\right] K(d\theta) \,\E(\gamma_s)(dx)\, ds\nonumber\\
%& + \int_0^t \int_{\R} \left(gx \partial_x f_s(x)+ \sigma^2 x \partial^2_{xx}f_s(x)+\partial_sf_s(x)\right) \E(\gamma_s)(dx)\, ds\nonumber\\
= & f_0(x_0) + \int_0^t \int_{} \left(rx \,\partial_x f_s(x)+\gamma x \,\partial^2_{xx}f_s(x)+\partial_s f_s(x)\right) \gamma_s(\ud x)  \ud s  \nonumber\\
 & \qquad +\int_0^t\int_{\R_+\times [0,1]} 2\tau \left[f_s(\theta x)-f_s(x)\right]    \P(F \in \ud \theta) \gamma_s(\ud x)\ud s.
\label{equationaux}
\end{align}

\me Let us prove that there is a unique solution to (\ref{nugaltonwatson}). We follow Step $1$ of the proof of Theorem~\ref{readif} and let $(\nu^1_t, t\geq 0)$ and $(\nu^2_t, t\geq 0)$ be two probability measures solution of (\ref{nugaltonwatson}). The total variation distance between $\nu^1_t$ and $\nu^2_t$ is
\begin{equation}
\| \nu^1_t-\nu^2_t\|_{TV}=\sup_{\substack{\phi\in C_b(\R_+,\R) \\ \|\phi\|_\infty\leq 1}}|\langle \nu^1_t,\phi\rangle-\langle \nu^2_t,\phi\rangle|.\label{totalvariation}
\end{equation}
Let $t\in \R_+$ and $\varphi\in C^2_b(\R_+,\R)$ with $\|\varphi\|_\infty\leq 1$. We denote by $(P_s : s\geq 0)$ the semi-group associated with the Feller diffusion %(\ref{Fellerdrifted}) 
started at $x\in \R_+$: $P_s\varphi(x)=\E_x(\varphi(X_s))$. Notice that $\|P_{t-s}\varphi\|_\infty\leq \|\varphi\|_\infty\leq 1$.  By (\ref{equationaux}) with $f_s(x)=P_{t-s}\varphi(x)$, the first term equals $0$ and
% the definition of the total variation norm yields:
\begin{align}
\left| \langle \nu^1_t-\nu^2_t, \varphi\rangle\right| = & \left|2\tau \int_0^t\int_{\R_+} \int_0^1  \Big(P_{t-s}\varphi(\theta x)-P_{t-s}\varphi(x)\Big)\P(F \in \ud\theta) (\nu_s^1-\nu^2_s)(\ud x)\, \ud s\right|\nonumber\\
\leq & 4\tau  \int_0^t \|\nu^1_s-\nu^2_s\|_{TV} \ud s. \nonumber
 \end{align}Since $C^2_b(\R_+,\R)$ is dense in $C_b(\R_+,\R)$ for the bounded pointwise topology, taking the supremum in the l.h.s. implies that: $\|\nu^1_t-\nu^2_t\|_{TV}\leq 4\tau \int_0^t \|\nu^1_s-\nu^2_s\|_{TV} \ud s$. Gronwall's Lemma ensures   that $\|\nu^1_t-\nu^2_t\|_{TV}=0$, which ends up the proof.
\end{proof}

\me 
 We can then interpret $\gamma_t$ as the marginal distribution (at time $t$) of an auxiliary process  $(Y_t, t\geq 0)$. 
\begin{Prop}
\label{vincent}For all $f\in C^{2}_b(\R_+,\R)$ and $t\in \R_+,$
 \be
 \langle \gamma_t,f\rangle=e^{-\tau t}\E\big(\sum_{i\in V_t} f(X_t^i) \big)= \mathbb{E}(f(Y_t)),
\label{relationdelmasmarsalle}
\ee
where $(Y_t, t\geq 0)$ is a Feller branching diffusion with parameters $(r,\gamma)$ and catastrophes  with rate $2\tau$ and distribution given by $F$. 
%More precisely, it is defined for $t\geq 0$ by
%\begin{align}
%\xi_t=  x_0+ &   \int_0^t r\xi_s ds+\int_0^t \sqrt{2 \gamma \xi_s}dW_s \nonumber\\
%+  & \int_0^t\int_{\R \times [0,1]} \ind_{u \leq 2\tau} [f(\theta \xi_{s_-})-f(\xi_{s_-})]
 % N(ds,du,d\theta)\label{definitionxi}
%\end{align}
%where $N(ds,dv,d\theta)$ is a PPM of intensity $ds\,dv\,d\theta$ and $W$ a standard BM independent from $N$. \\
Moreover, 
\begin{equation} \E\left(\# \{i \in V_t : X^i_t>0\}\right)=e^{\tau t}\P(Y_t>0).\label{etape4}
\end{equation}
\end{Prop}
\begin{proof}
One can describe the dynamics of  $t\rightarrow \nu_t(f_t)$ where $\nu_t(dx)=\P(Y_t\in dx)$ thanks to  It\^o's formula and check that it satisfies (\ref{nugaltonwatson}). Uniqueness of the solution  of this equation yields  (\ref{relationdelmasmarsalle}). We can then apply \eqref{relationdelmasmarsalle}
 with  $f(x)=\ind_{x>0}$ by taking a monotone limit of $C^2_b$ functions  to get \eqref{etape4}.
\end{proof}
$\newline$
\begin{proof}[Proof of Theorem \ref{threcoveryorganismrconstant}] Let us first prove the convergence in probability in  $(i)$. We denote  by $V_t^*=\{i \in V_t : X^i_t>0\}$ the set of infected 
cells and by $N_t^*=\# V_t^*$ its cardinality.  By Theorem \ref{extlignee}, under the assumption of $(i)$, $(Y_t, t\geq 0)$ dies in finite time a.s. Thus,  Proposition \ref{vincent} ensures that $N_t^*/\exp(\tau t)$  converges in $\mathbb{L}^1$ 
and hence in probability to $0$. Moreover the nonnegative martingale $N_t/\exp(\tau t)$ tends a.s. to a nonnegative random variable  $W$. In addition,  the $\mathbb{L}^2$-convergence  can be
easily obtained and is left to the reader.  Therefore, $\P(W>0)>0$  and then $W>0$ a.s. using the branching  property and $\P(\forall t >0 : N_t>0)=1$. One could actually even show that  $W$ is an exponential random variable with mean $1$. 
 Then
\begin{equation}
\lim_{t\rightarrow +\infty}\frac{N_t^*}{N_t}=\lim_{t\rightarrow +\infty}\frac{N_t^*}{e^{\tau t}}\frac{e^{\tau t}}{N_t}=0\quad \emph{in probability}.\label{etape5}\end{equation}
%Let us now prove the a.s. convergence.
%From (\ref{etape4}):
%\begin{align}
%\int_0^{+\infty}\E\Big(\frac{\# \{ i \in V_t : X^i_t>0\}}{\E(N_t)}\Big)dt=\int_0^{+\infty}\P\big(\xi_t>0\big)dt<+\infty
%\end{align}since an inequality similar to (\ref{majoprobasurvie}) holds for $\xi$, which is a drifted Feller diffusion with multiplicative jumps at a %rate $2r$ instead of $r$. As the integrand in the l.h.s. is nonnegative, we have by Fubini theorem:
%\begin{align}
%\E\Big(\int_0^{+\infty}\frac{\# \{ i \in V_t : X^i_t>0\}}{\E(N_t)}dt\Big)<+\infty
%\end{align}from which we deduce that the integral under the expectation is finite a.s. This implies that the convergence in (\ref{etape5}) holds a.s.
It remains to show that the convergence holds a.s, which is achieved by checking that
 \be
\label{CV}
\sup_{s\geq 0}N^*_{t+s}/N_{t+s} \stackrel{t\rightarrow \infty}{\longrightarrow} 0  \quad \emph{in probability}.
\ee
Indeed, let us denote  by $V_{t,s}(i)$ the set of cells alive at time $t+s$ and whose ancestor at time $t$ is the cell $i\in V_t$. Then $(\#V_{t,s}(i), s\geq 0)$ are i.i.d. random processes for $i\in V_t$ distributed as $(N_{s}, s\geq 0)$. We have
$$N^*_{t+s}\leq \sum_{i\in V_t^*} \# V_{t,s}(i) \leq e^{\tau s} \sum_{i\in V_t^*} M_t(i) \quad \text{a.s.}$$
where $M_t(i)$ are i.i.d random variables distributed like $M:=\sup\{ e^{-\tau s}N_s, s\geq 0\}$. Similarly,
$$N_{t+s} \geq \sum_{i\in V_t} \# V_{t,s}(i) \geq e^{\tau s} \sum_{i\in V_t} I_t(i) \quad \text{a.s.},$$
where $I_t(i)$ are i.i.d. random variables  distributed like $I:=\inf\{ e^{-\tau s}N_s, s\geq 0\}$. We add that $\E(M)<\infty$ since the martingale  $e^{-\tau s}N_s$ is bounded in $\mathbb{L}^2$. Moreover   $\E(I) \in (0,\infty)$   so that %by the law of large numbers that
$$\frac{N_t^{*-1}\sum_{i\in V_t^*} M_t(i) }{N_t^{-1}\sum_{i\in V_t} I_t(i) }$$
is stochastically bounded (or tight) for $t\geq 0$. % : t\geq 0\right\}< \infty \text{ a.s.}$$
Using that  $N_t^*/N_t\rightarrow 0$ when $t\rightarrow +\infty$ in probability yields (\ref{CV}). This ensures the a.s. convergence of $R_t=N_t^*/N_t$ to $0$, using that
$\P(\limsup_{t\rightarrow \infty} R_t \geq \epsilon) \leq \lim_{t \rightarrow \infty} \P(\sup_{s \geq 0} R_{t+s} \geq \epsilon)=0$. 

\me The proof of $(ii)$ is similar. One can first note thanks to Section \ref{cata} that
$\,\P(Y_t\geq \exp(\kappa t))$ has a positive limit and prove that $$\left\{\limsup_{t\rightarrow +\infty} \frac{\# \{ i \in V_t : X^i_t\geq e^{\kappa t} \}}{N_t} > 0 \right\}$$
has a positive probability. To check that this latter event coincides with $\{\forall  t>0: \ X_t>0\}$, 
%\marginpar{details ?}
 a zero-one law is involved, which is inherited from the branching property by a standard argument.
% we get by Lemma \ref{lemtechnq} that
%$$\lim_{t\rightarrow +\infty}\sup_{s\geq 0}\frac{\sum_{i\in V_t^*} \# V_{t,s}(i)}{\sum_{i\in V_t} \# V_{t,s}(i)}=0\qquad \mbox{ in probability.}$$
%This ensures that and we get  the a.s. convergence of $N_t^*/N_t$
%using the following standard argument
%$$\{ \exists \epsilon>0, \forall t\geq 0, \exists s\geq 0, N^*_{t+s}/N_{t+s} \geq  2\epsilon \} \subset \bigcup_{\epsilon > 0}\bigcap_{t\geq 0} \{
%\sup_{s\geq 0} N^*_{t+s}/N_{t+s} \geq  2\epsilon \} \qquad \t{a.s.}$$
%where the probability of the right hand side event is equal to $0$. Th"is ends the proof. 
\end{proof}

\paragraph{Asymptotic regimes for the speed of infection.}  Combining Theorem \ref{equiv} and  Proposition \ref{relationdelmasmarsalle} yield  different asymptotic regimes for the mean number of infected cells $\E(N_t^*)$. They are plotted in Figure \ref{figg3}  when the sharing of parasites is deterministic. We stress that it differs from the discrete analogous model \cite{kim}.
%\marginpar{Notations paramÃƒÂ¨tre}
\begin{figure}
\begin{center}
\label{figg3}
\includegraphics[scale=0.4]{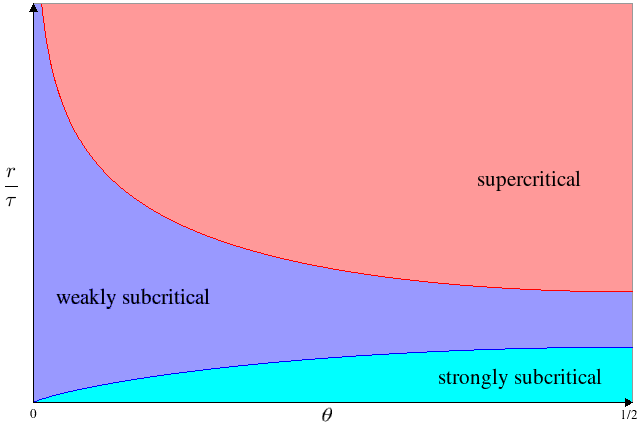}
\end{center}
\caption{Asymptotic regimes for the mean number of infected cells when $\P(F=\theta)=\P(F=1-\theta)=1/2$ and $\theta\in (0,1)$.} 
\label{figg3}
\end{figure}
In the supercritical regime, the number of infected cells and the number of cells are of  the same order. In the strongly subcritical regime, the number of infected cells 
and the parasite loads are of  the same order. In the weakly subcritical regime, the number of infected cells  is negligible compared to the number of cells and the 
amount of parasites.

 \subsection{Some first results for a monotonic rate of division}
\me The asymptotic study of such processes with a non constant rate of division is the object of recent works (see \cite{cloez,DHKR, Bansaye} and works in progress). Let us simply mention some relevant consequences of the previous results for non-decreasing rate $\tau$. 

\subsubsection{A sufficient condition for recovery}
\me We assume either  that $\tau$ is non-decreasing and there exists $x_1>0$ such that $\tau(x_1)<\tau(0)$ or that $\tau(0)>\tau(x)$ for any $x>0$.
The second case means that the non-infected cells divide faster than the infected ones. Let us recall from Proposition \ref{linerincre} in Section \ref{catam} the notation
$$\tau_*:= \inf\{ \tau(x) : x\geq 0\}.$$
\begin{Prop}
If $r\leq \tau_* \E(\log1/(\min(F,1-F)))$, the organism recovers a.s.
\end{Prop}
\me We only give  the idea of  the proof and refer to \cite{MR2754402} for more details. Let us follow a cell lineage by starting from the initial cell and
 choosing the less infected cell at each division. The infection along this lineage is a  Feller diffusion with catastrophes whose distribution is $\min(F,1-F)$ and the catastrophe rate is
 $\tau$. We know from Section \ref{catam} when the infection in   such a cell line becomes extinct a.s. Then one  uses that the population of non-infected cells is growing faster than the population of cells infected by more than $x_1$ parasites.

\subsubsection{An example of moderate infection}\label{section62}

\me We assume here that $\tau(x)$ is an increasing function of the parasite load. This means that the more the cell is infected, the faster it divides. Low infected cells divide slower and may even stop dividing if $\tau(0)=0$. That's why a new regime appears here, between recovery and proliferation of the parasites, where a positive fraction of cells is infected but the quantity of parasites inside remains bounded.
We then say that the infection is $\emph{moderate}$. Let us provide an example where the infection is indeed moderate:
the organism does not recover but the parasites do not proliferate in the cells.
% Actually 'moderate infection' could be defined in several ways and we also consider the average quantity of parasites per cell.
%Then we give two ways to describe this phenomenon using Section \ref{sectionjefflo}. The first one gives a sufficient condition under which the average quantity of parasites per cell remains bounded when time goes to infinity. The criterion  does not depend on the sharing $\Theta$ of the parasites.  The second one  considers when the proportion of  cells infected by  more that $A$ parasites tends to zero as $A$ goes to infinity. We conjecture a criterion for this property.

$$F=1/2 \ \ \t{a.s.}, \quad \tau(x)=0  \ \  \t{if}  \ x< 2 \  \  \ \t{and}  \ \ \ \tau(x)=\infty \ \ \t{if} \ \  x\geq 2.$$
Indeed, as soon as the quantity of parasites in a cell reaches $2$, the cell divides and the  quantity of parasites in each daughter cell is equal to one. The parasites do not proliferate in the cells since the parasite load  in each cell is less than $2$. \\
We now fix the growth rate of parasites $r$ such that the probability that the Feller diffusion $(X_t, t\geq 0)$ reaches $0$ before  $2$ is strictly less than $1/2$. Then the number of infected cells follows a supercritical branching process and grows exponentially
 with positive probability. Conditionally on this event, the proportion of infected cells  doesn't tend to zero since the non-infected cells do not divide. Thus the organism doesn't recover. \\
%This shows that we can get both non recovery and non proliferation of parasites in the cells: this corresponds to a  'moderate infection' .\\

\section{Markov Processes along Continuous Time Galton-Watson Trees}

\me In this section, we consider measure-valued processes associated with a discrete genealogy given by a branching process.   These processes describe a structured population where individuals are characterized by a trait. We focus on the case where the  branching rate is constant and the number of offspring belongs to $\mathbb{N}$. During the life of an individual, its trait dynamics
is modeled by a general Markov process. 
More precisely, the individuals behave independently and
\begin{enumerate}
\item Each individual lives during an independent exponential time of parameter  $\tau$ and
then gives birth to a random number of offspring whose distribution is given by $(p_k, k\geq 0)$.  
\item Between two branching events, the trait dynamics of an individual follows  a c\`adl\`ag
strong Markov process $(X_t)_{t\geq 0}$ with values in a Polish space
$\mathcal X$ and  infinitesimal generator $L$ with domain $D(L)$. Here again, we can assume that $\mathcal X\subset\R^d$. 
\item When an individual with trait $x$ dies, the distribution of the traits of its offspring is given by   $(P^{(k)}(x,dx_1,\ldots, dx_k) : k\geq 1)$, where $k$ is the number of offspring. 
\end{enumerate}
%It h clearly depends on the state $x$ of the mother just before the branching event, and of the number $N=k$ of offspring ; besides, the randomness of these states are modelled \textit{via} the random variable $\Theta$, which is uniform on $[0,1]$.
 Let us note that an individual may die without descendance when $p_0>0$. Moreover,  when $X$ is a Feller diffusion and $p_2=1$, we recover the splitting Feller diffusion of Section \ref{div}. In the general case,  the process $X$ is no longer a branching process and the key property  for the long time study of the measure-valued process
will be  the ergodicity of  a well chosen auxiliary Markov process. 

%\marginpar{ ajouter le livre de Englander ? sylvie R ?}

  \me A vast literature can be found concerning branching Markov processes and  special attention has been payed to Branching Brownian Motion from the pioneering
  work of Biggins \cite{biggins77} about branching random walks,  see  e.g. \cite {Eng, Shi} and references therein.
More recently, non-local branching events  (with jumps occurring
at the branching times) and superprocesses limits corresponding to small and rapidly branching particles have been considered and we  refer e.g. to  the works of Dawson et al.
 %Dawson \cite{dawson}
 and Dynkin \cite{Dynkin}.% In this vein, let us also mention Evans and Steinsaltz \cite{evanssteinsaltz} for a model of  transmission of cell damages. 
%We stick here with the discrete \marginpar{citer aussi Sylvie R. et les deux Sylvie ?}tree in continuous time which we aim characterizing.

\subsection{Continuous Time Galton-Watson Genealogy}
\me The genealogy of the population is  a branching process with reproduction at rate $\tau$ and offspring distribution given by $(p_{k},k\geq 0)$. We assume that the population arises from a single ancestor $\{\emptyset \}$.
Roughly speaking, the genealogy is obtained by adding i.i.d. exponential life lengths (with parameter $\tau$) to a  (discrete)  Galton-Watson tree.
Let us give some details on this construction, which will be useful in the sequel.
We define $\mathbf I:=\{\varnothing\} \cup \bigcup_{n\geq 1} (\N^*)^n$, which we endow with the  order relation $\preceq$  : $u\preceq v$ if there exists $w\in \mathbf I$ such that $v=(u,w)$.  For example, the individual $(2,3,4)$ is the fourth child of the third child of the second child of the root $\emptyset$. 
We denote by $(A(i), i \in \mathbf{I})$ i.i.d. random variables with distribution $p$. The discrete genealogy $\mathcal I$ is the random  subset of
 $\mathbf I$ obtained by keeping the individuals which are indeed born: 
$$\mathcal I:= \cup_{q\geq 0}  \{(i_1,\ldots, i_q) : \forall k=1, \ldots, q, \  i_k\leq A((i_1,\ldots, i_{k-1}))\},$$
where by convention the first set in the right hand side is $\{\emptyset \}$.
 We
 consider now a sequence
 $(l_i, i \in \mathbf{I})$ of exponential random variables, so that    $l_i$ is the life time 
 of the individual $i\in \mathcal{I}$
and  \begin{equation}
\alpha (i)=  \sum_{j \prec i} l_j \quad \mbox{ and }\quad
\beta  (i)=   \sum_{j \preceq i} l_j = \alpha (i)+ l_i,\label{defalphabeta}
\end{equation}
with the convention  $\alpha(\emptyset)=0$, are the birth and death times of $i \in \mathcal{I}$.  
We assume that the offspring  distribution $p$ has a finite
second moment and that
$$
m=\sum_{k \ge 0} k\,p_k> 
% \quad \mbox{ and }\quad  \varsigma^2=\sum_{k\ge 0}(k-m)^2 \,p_k,
%\end{equation}
%its expectation and variance.  
%We assume also that the population survives with positive probability
%$$m
1 \ \text{ (supercriticality)}.$$

\noindent  Let us  denote by $V_t\subset \mathcal I$ the set of individuals alive at time $t$ :
$$V_t:=\{ i \in \mathcal I : \alpha(i)\leq t<\beta(i)\}, \quad \hbox{ and as before } \ N_t=\#V_t.$$

\noindent The supercriticality assumption on the reproduction law  implies the possible persistence of the process. 
\begin{Prop}
\label{moy} The population size process $(N_{t}, t\geq 0)$ survives with positive probability. Moreover, for any $t\geq 0$, \be
\label{moy-sc}\E(N_t)=\exp(\tau(m-1)t)\ee  and
 $$\frac{N_t}{\E(N_t)}\stackrel{t\rightarrow \infty}{\longrightarrow} W \quad \emph{a.s. and in } \mathbb L^2,$$
where $W \in \R_+$ is positive on the survival event. \\
\end{Prop}
\noindent The proof uses the generator of $N$  and the martingale $N_t/\E(N_t)$. It is left to the reader.

\bi The continuous time Galton-Watson genealogy $\mathcal T$ is defined as the  (random) subset  of $\mathcal I \times \R_+$ such that $(i,t) \in \mathcal T$ if and only if $i \in V_t$.

\me 
We define now the branching Markov process  along this genealogy  $\mathcal T$.     We use the shift operator $\theta$ for such trees and $\theta_{(i,t)}\mathcal T$ is the subtree of $\mathcal T$ rooted in $(i,t)$. \begin{Def}
\label{def:XT}
Let $X=(X_t,t\geq 0)$ be a c\`adl\`ag $\mathcal X$-valued strong Markov process and $\mu\in \cp(\mathcal X)$.
Let $(P^{(k)}(x,dx_1\ldots dx_k) ,k\geq 1)$ be a family of
transitions probabilities  from $\mathcal X$ to $\mathcal X^k$. \\ The continuous time branching
Markov  process $X_\mathcal T=(X_t^i,(i,t)\in \mathcal T)$ indexed by $\mathcal T$, the underlying trait dynamics  $X$  and starting
distribution $\mu$, is defined conditionally on $\mathcal T$ recursively as follows:
\begin{itemize}
%\item[(i)] $\mathcal T$  is a continuous time Galton-Watson  tree with offspring
%  distribution $p$ and  exponential lifetimes with mean $1/r$.
\item[(i)] $X^\emptyset=(X^\emptyset_t,t\in
  [0,\beta(\emptyset)))$ is distributed as $(X_t,t\in
  [0,\beta(\emptyset)))$ with  $X_0$ distributed according   $\mu$.
\item[(ii)] Conditionally  on $X^\emptyset$,  the initial traits of
  the first generation  of offspring   $(X^i_{\alpha(i)},1\leq  i\leq
  A({\emptyset}))$               are              distributed
as $P^{(A({\emptyset}))}(x,dx_1\ldots dx_{A({\emptyset})})$. 
\item[(iii)]    Conditionally    on   $X^\emptyset$,    $A(\emptyset)$,
  $\beta_\emptyset$  and $(X^i_{\alpha(i)},1\leq  i\leq A(\emptyset))$,
  the  tree-indexed  Markov  processes $(X^{ij}_{\alpha(i)+t},  (j,t)\in
  \theta_{(i,\alpha(i))}\mathcal{T})$  for $1\leq  i  \leq A(\emptyset)$
  are  independent  and  respectively  distributed  as  $X$ with
  starting distribution  $\delta_{X^i_{\alpha(i)}}$.
\end{itemize}
\end{Def}
%For $x \in E$, we define $\P_x(A)=\P(A |X_0^{\emptyset}=x)$ for all $A \in \mathcal{F}$, and denote by $\E_x$ the corresponding expectation. 

\noindent There is no spatial structure on the genealogical tree and without loss of generality, we assume that the marginal measures of  $P^{(k)}(x,dx_1\ldots dx_k)$ are identical.  It can be achieved simply by a uniform permutation of the traits of the offspring.

%For $u \in \mathcal{T}$, we extend the definition of $X_t^i$ when $t \in [0, \alpha(u))$ as follows:   \\

%Let   $\mathcal{B}_b(E,\R)$  be  the  set  of
%real-valued measurable  bounded functions on $E$ and  $\mathcal M_F(E)$ the set
%of  finite  measures   on  $E$  embedded  with  the   topology  of  weak
%convergence.   
%For $\mu\in  \mathcal M_F(E)$  and $f\in\mathcal{B}_b(E,\R)$  we
%write $\langle \mu ,f\rangle=\int_E f(x)p(dx)$.

%\marginpar{Introduire $\bar{Z}_t$ pour fermer ?}
\bi 
Following the previous sections, we give a pathwise representation of the  point measure-valued process defined  at
time $t$ by
\begin{equation}
%\bar{Z}_t=\sum_{i\in V_t}\delta_{(i,X_t^i)},\label{defz} \quad\text{and}\quad 
{Z}_t=\sum_{i\in V_t}\delta_{X_t^i}.
\end{equation}
%\marginpar{Ref cours Sylvie. Choix de notation  :  $Q$ ou $\rho$, choix de labelisation different : arbre discret plutot que ordre lexicographq ?}
The dynamics  of
$Z$ is given by the following stochastic differential equation. Let
$N(\ud s, \ud i, \ud k, \ud \theta)$ be a
Poisson point measure on  $\mathbb R_+\times \mathbf I \times \mathbb N \times [0,1]$ with intensity
$\tau\ud s n(\ud i) p(\ud k) d\theta$ where 
$n(\ud i)$ is the counting measure on $\mathbf{I}$ and $p(\ud k)=\sum_{l\in
  \N} p_l \delta_l (\ud k)$ is the
offspring distribution.  Let $L$ be the infinitesimal generator of $X$. Then for test functions
$f\,:\,(t,x)\mapsto f_t(x)$ in $\Co^{1,0}_b(\R_+\times \mathcal X,\R)$ such that
$\forall t\in \R_+,f_t\in D(L)$, we have
\begin{align}
\langle Z_t,f_t\rangle= & f_0(X^\emptyset_0) 
  + \int_0^t\int_{\R_+}\left(Lf_s(x) +\partial_s
   f_s(x)\right)\ud sZ_s(\ud x) + M^f_t \label{martingalegdepop} \\
 &   +\int_0^t
\int_{\mathbf{I}\times \N\times [0,1]} \One_{\{i\in V_{s-}\}}
\left(\sum_{j=1}^k
  f_s(F^{(k)}_j(X^i_{s_-},\theta))-f_s(X^i_{s_-})\right)N(\ud s,\ud i,
\ud k,\ud \theta), \nonumber 
\end{align}
where $M^f_t$  is a  martingale and $(F_j^{(k)}(x,\Theta) : j=1\ldots k)$ is a random vector distributed like  $P^{(k)}(x,\ud x_1\ldots \ud x_k)$ when $\Theta$ is uniform in $[0,1]$. %Explicit expressions of this martingale and of the infinitesimal generator of $(Z_t,t\ge 0)$ can be obtained when the form of the generator $L$ is given.

\subsection{Long time behavior}

\me We  are now interested in studying the long time behavior of the branching Markov process $Z$. We will show that it is deduced from the knowledge of  the long time behavior of a well-chosen one-dimensional auxiliary Markov process. In particular, the irreducibility of the auxiliary process will give a sufficient condition in the applications, to obtain a limit as time goes to infinity. 

\me Let us recall from Proposition \ref{moy} that the expectation at time $t$ and the  long time behavior of the population size process $N$ are known. 

\subsubsection{Many-to-one formula}
\me We introduce the auxiliary Markov process $Y$ with  infinitesimal generator given by
$$Af(x)=Lf(x)+\tau m\int_{\mathcal X} \big(f(y)-f(x)\big)Q(x,\ud y)$$ 
for $f \in D(L)$
and
$$Q(x,dy):=\frac{1}{m} \sum_{k\geq 0} kp_kP^{(k)}(x,\ud y \mathcal X^{k-1}).$$
 In  words, $Y$ follows the dynamics of  $X$ with additional jumps at rate $\tau m$ whose distribution is given by the size biased transition probability measure $Q$.

\begin{Prop}
\label{propprocessusauxiliaire}
For $t\ge 0$ and for any non-negative measurable function $f\in \mathcal{B}(\mathbb D([0,t],\mathcal X))$  and $t \ge 0$, we
have
\begin{equation}
\label{defYt}
\E_\mu\left(\sum_{i \in V_t} f(X^i_s, \   s\leq t)
   \right)
= \E(N_t)\, \mathbb{E}_\mu(f({Y}_s, \ s\leq t )) = e^{\tau(m-1)t} \, \mathbb{E}_\mu(f({Y}_s, \ s\leq t )) ,
\end{equation}
where (with a slight abuse) $X^i_s$ is the trait of the ancestor of $i \in V_t$ living at time $s$.
\end{Prop}

\me To prove such a formula in the particular case $f(x_s,  s\leq t)=f(x_t)$, one can use   It\^o's calculus and follow Section \ref{BFDpara} and conclude with a monotone class argument. 
Here we prove the general statement using the following Girsanov type formula.
In the rest of this section, the random jumps $\w{T}_k$ and $T_k$ of the Poisson point processes on $\mathbb R^+$ that we consider, are ranked in  increasing order.

\begin{Lem} Let $\{(\w{T}_k,\w{A}_k) : k \geq 0\}$ be a Poisson point process  with intensity $\tau m \, \ud s \, \w{p}(\ud k)$ on $\R_+\times \N$, where
%Let $\Lambda$  be a  compound Poisson process with rate $rm$ and increments  distributed as $\log S$ where $S$ is an integer valued r.v. whose distribution is
$$\w{p}(\{k\})= kp_k/m.$$ 
Then, for any $t\geq 0$ and $q\geq 0$ and any non-negative measurable function $g$ on $(\mathbb{R}_{+}\times \mathbb{N})^{q+1}$:
\begin{align*}
\E\left(g( (\w{T}_k,\w{A}_k) : k \leq q)\ind_{\w{T}_q\leq t <\w{T}_{q+1}}\right)=\expp{-\tau(m-1)t}\E\left(g(({T}_k,{A}_k) : k \leq q) \One_{\{T_q\leq t <T_{q+1}\}} \prod_{k\leq q}   A_k
  \right),%\label{utgirsano}
\end{align*}
where $\{({T}_k,{A}_k) : k \geq 0\}$  is a Poisson point process on $\R_+\times \N$ with intensity $\tau \, \ud s \, p(\ud k)$. Thus, for any measurable non-negative function $h$,
\begin{align}
\E\left(h( (\w{T}_k,\w{A}_k) : k\geq 0, \ \w{T}_{k}\leq t)\right)=\expp{-\tau(m-1)t}\E\left(h(({T}_k,{A}_k) : k\geq 0, \  T_{k} \leq t)  \prod_{T_{k} \leq t}  A_k
  \right).\label{utgirsanov}
\end{align}
\end{Lem}

\begin{proof} Let $q\geq 0$ and remark that % $g(\Lambda_{[0,t]})$ is a function of $t$, of the times $\tau_q=\inf\{t\ge 0; S_t=q\}-\inf\{t\ge 0; S_t=q-1\}$ and of jump sizes $\log(H_q)$ of $\Lambda$:
\begin{align*}
g((\w{T}_k,\w{A}_k) : k \leq q)=& G_q(\w{T}_0,\w{T}_1-\w{T}_0,\ldots, \w{T}_q-\w{T}_{q-1},\w{A}_0, \w{A}_1,\dots, \w{A}_q),
\end{align*}
for some non-negative functions $(G_q, q\in \N)$. Using that $\w{T}_0$ and $(\w{T}_{k+1}-\w{T}_k : k\geq 0)$ are i.i.d. exponential random variables with parameter
$\tau m$, we  deduce that
\begin{multline*}
%\begin{align*}
\E[g((\w{T}_k,\w{A}_k) : k \leq q) \ind_{\w{T}_q\leq t <\w{T}_{q+1}}]  \\
\begin{aligned}
 =&  \int_{\R_+^{q+2}}\sum_{n_0,\dots, n_q}\!\! (\tau m)^{q+2} \expp{-\tau m(t_0+\ldots+t_{q+1})} G_q(t_0,\dots, t_q,n_0,\dots ,n_q) \\
& \qquad \qquad \qquad \qquad \qquad \qquad \qquad \times\prod_{k=0}^q \frac{p_{n_k}n_k}{m}\, 
 \ind_{\{\sum_{k=0}^{q} t_k\leq t<\sum_{k=0}^{q+1} t_k\}}\ud t_0\ldots \ud t_{q+1}\\
= &  \int_{\R_+^{q+1}}\sum_{n_0,\dots, n_q}\!\!\tau^{q+1}
\expp{-\tau t} G_q(t_0,\dots, t_q,n_0,\dots, n_q)
\expp{-\tau (m-1)t} \\
& \qquad \qquad \qquad \qquad \qquad \qquad \qquad \times \prod_{k=0}^q n_k  \, p_{n_k}\,
\ind_{\{\sum_{k=0}^q t_k\leq t\}} \ud t_0 \ldots \ud t_q,
%= & \E\left[g(\Lambda'_{[0,t]})\expp{-r(m-1)t +
%    \Lambda'_t}  \right].
\end{aligned}
\end{multline*}
which yields the first result. %The second one is obtained by summing over $q$ and forgetting the last point of the PPP.
\end{proof}

\me 
\begin{proof}[Proof of Proposition \ref{propprocessusauxiliaire}]  We give here the main steps of the proof.
We  recall that the random variables  $(A(i),l(i))$ have been defined  for $i\in \mathbf I$. Let us   now introduce the point process describing the birth times and number of offspring of
the ancestral lineage of $i$:
$$\Lambda^i := \{(\beta(j), A(j)) : j \preceq i  \}  \qquad (i \in \mathbf I).$$
 % It relies on a restriction to $i\in \mathcal I$ and compatibility of the PPP. 
% The process
%$\Lambda^i$  describes the ancestral lineage of $i\in V_t$ in the continuous time Galton-Watson tree, which means the number of offspring and date of birth of the ancestors $i$. 
%It means that $T_k^i$ gives the successive
%birth time $\alpha(j)$ of the ancestors $j$ of $i$, whereas $N_k^i$ is the associated number of offspring $N(j)$.
We stress that  for any $q\geq 0$, the processes   $(\Lambda^i , i\in (\mathbb{N}^*)^q)$ corresponding to individual labels with length $q$, are  identically distributed (but dependent).  
We introduce now $(X', \Lambda)$ and $\Lambda_q$ , where 
\begin{itemize}
\item  $\Lambda=\{(T_k,A_k) :  k \geq 0\}$ is
 a Poisson point process   on $\R_+\times \N$ with intensity $\tau  \ud s  p(\ud k)$ (with $T_k$   ranked  in increasing order).
  $\Lambda_q=\{(T_k,A_k) :  k \leq q\}$.
 \item Conditionally on $\Lambda=\{ (t_k,n_k) : k \geq 0\}$ for $n_k \geq 1$, $X'$ is the time non-homogeneous Markov process such that
\begin{itemize}
\item  at time $t_{i}$,  $X'$ jumps and the transition probability is given by $P^{(n_i)}(x,  \ud y \mathcal X^{n_i-1})$.
 \item during the time intervals $[t_i,t_{i+1})$, the infinitesimal generator of $X'$ is $L$; 
\end{itemize}
\end{itemize}
We denote by $|i|$ the length of the label $i$. We can now compute for any $t\geq 0$ and $i \in \mathbf I$ such that $|i|=q$,
\Bea
\E_\mu\big(f(X^i_s : s \leq t)\One_{\{ i\in V_t\}} \vert \Lambda^i\big)&=&
 \One_{\{\alpha(i) \leq t < \beta(i); \ \forall k=1,\ldots,q : \ A(i_1, \ldots ,i_{k-1}) \geq  i_{k} \}}F_t(\Lambda^i)\Eea
where $ F_t(\lambda)= \E_\mu(f(X'_s : s\leq t)| \Lambda \cap ([0,t]\times \N)=\lambda  \cap ([0,t]\times \N) )$. Since $\Lambda^i$ is distribued as $\Lambda_q$, we get
\Bea
&&\sum_{i\in \mathbf I}
\E_\mu\big(f(X^i_s : s \leq t)\One_{\{ i\in V_t\}}\big) \\
 &&=\sum_{i\in \mathbf I } \sum_{q\in \N} \One_{\{|i|=q\}}
\E_\mu\big(F_t(\Lambda^i)\One_{\{\alpha(i) \leq t < \beta(i); \ \forall k=1,\ldots,q : \ A(i_1, \ldots ,i_{k-1}) \geq  i_{k} \}} \big)\\
&&= \sum_{i\in \mathbf I} \sum_{q\in \N} \One_{\{|i|=q\}}
\E_\mu\big( F_t(\Lambda_q) \One_{\{T_{q-1}\leq t < T_{q}; \ \forall k=0,\ldots, q-1: \ A_{k} \geq  i_{k+1} \}} \big) \\
%&&= \sum_{i\in \mathbf I } \sum_{q\in \N} \ind_{\{|i|=q\}}
%\E_\mu\big( F_t(\Lambda) \One_{\{T_{q-1}\leq t < T_{q}; \ \forall k=0,\ldots, q-1: \ A_{k} \geq  i_{k+1} \}} \big), 
\Eea
where we have used the convention $T_{-1}=0$.
Adding that
%\marginpar{one can take $g$ =1 to get (only) the mto formula}
\Bea
&&\sum_{q\in \N}\sum_{i\in \mathbf I }  \One_{\{|i|=q\}} F_t(\Lambda_q)
\ind_{\{T_{q-1}\leq t < T_{q}; \ \forall k=0,\ldots,q-1:  \ A_{k} \geq  i_{k+1}\} } \\
&&\quad =\sum_{q\in \N}F_t(\Lambda_q)\#\{i \in \mathbf{I} : \vert i\vert=q,   \ \forall k=0,\ldots,q-1 :  i_{k+1}\leq A_{k}\} \ind_{T_{q-1}\leq t < T_{q}}\\
&&\quad = F_t((T_k,A_k) : k \geq 0, T_k\leq t)\prod_{T_{k}\leq t} A_{k}
\Eea
and using (\ref{utgirsanov}), we get 
\Bea
\sum_{i\in \mathbf I}
\E_\mu\big(f(X^i_s : s \leq t)\One_{\{ i\in V_t\}}\big)
&=& \E_\mu\left( F_t((T_k,A_k) : k \geq 0, T_k\leq t)\prod_{T_{k}\leq t} A_{k}\right) \\
&=&   e^{\tau (m-1)t} \E_\mu\left(F_t((\w{T}_k,\w{A}_k) : k \geq 0, \w{T}_k\leq t)\right). 
\Eea
Finally,  we  combine the definitions of $X'$ and $Y$ to conclude, recalling that $\{(\w{T}_k,\w{A}_k) : k \geq 0\}$ is a Poisson point process with intensity $\tau m  \ud s  \w{p}(\ud k)$ and \eqref{moy-sc}.
\end{proof}
\subsubsection{Law of large numbers}

\me Let us now  describe the  asymptotic distribution of traits within the population (see \cite{BDMT} for details).

\begin{Thm}\label{thLGNannonceintro}
Assume that for some bounded measurable function $f$, the auxiliary process satisfies  
\be
\label{limf}
\E_x(f(Y_t)) \stackrel{t\rightarrow \infty}{\longrightarrow} \pi(f)
\ee
for every $x\in \mathcal X$ and $\pi$ a probability measure on ${\cal X}$. 

\noindent Then, for every probability distribution $\mu$ on $\mathcal X$,
\begin{equation}
\lim_{t\rightarrow \infty} \frac{\ind_{\{N_t>0\}}}{N_t}\sum_{i\in V_t}f\big(X^i_t\big)=  \ind_{\{W>0\}} \pi(f)\end{equation}
in  $\P_{\mu}$  probability.
\end{Thm}
\noindent This result  implies in particular that for such a function $f$,
\begin{equation}
\lim_{t\rightarrow +\infty} \E\big[f(X_t^{U(t)})\,|\, N_t>0\big]=\pi(f),\label{equationintro2}
\end{equation}where $U(t)$ stands for an individual chosen at random in the
set $V_t$ of individuals alive at time $t$. \\
Condition  $(\ref{limf})$ deals with the ergodic behavior of $Y$ and will be obtained for regular classes of functions $f$, see below for an example.  
%see below for applications and \cite{BDMT} for the convergence of the punctual measure $$\frac{\ind_{\{N_t>0\}}}{N_t}Z_t$$ to $\ind_{\{W>0\}} \pi$.
% bounded and continuous. Then,
%\begin{equation}
%\lim_{t\rightarrow \infty} \frac{\ind_{\{N_t>0\}}}{N_t}\sum_{u\in V_t}\delta_{X^u_t}(\ud x)
%\end{equation}
%converges weakly to
%$$ \ind_{\{W>0\}}\,\pi(\ud x)$$

\begin{proof}[Main ideas of the proof]
Let $f$  be a non negative function  bounded by $1$ and define 
$$G^i_t:=f(X_t^i)- \pi(f)$$
 and let us prove that
$$A_t:=\E\left( \left( \sum_{i \in V_t} G_t^i\right)^2 \right) \ll \E(N_t)^2$$
Indeed we can write
$A_t=B_t+C_t$, where %recalling that $N_t / \E(N_t)$ converges in $L^2$ to a non degenerate r.v.
$$B_t:=\E\left( \sum_{i \in V_t} (G_t^i)^2 \right) \quad \hbox{ and } \quad
C_t:=\E\left(  \sum_{i \ne j \in V_t} G_t^i G_t^j \right).$$
We easily remark  from \eqref{moy-sc} that 
$B_t \leq \E(N_t) \ll \E(N_t)^2$.
Let us now deal with $C_t$ and use the most recent common ancestor of $i$ and $j$:
\Bea
C_t = \E\left(\sum_{\substack{u, (u,k_1), (u,k_2) \in \mathcal I,  \\  k_1\ne k_2}} \One_{\{\beta(u)<t\}}\,\E\left(  \sum_{ i \in V_t : i\succeq (u,k_1)} \sum_{j \in V_t : j\succeq (u,k_2)} \E\left(G_t^i G_t^j \big\vert\, \beta(u), X^{(u,k_1)}_{\beta(u)},X^{(u,k_2)}_{\beta(u)}\right) \right) \right) \\
%&=:& \E\left(\sum_{u \in \mathcal I, \alpha(u)<t}\right)
\Eea
The key point is that on the event $\{\beta(u)<t, \ (u,k_1) \in \mathcal I, (u,k_2) \in \mathcal I\}$,
\Bea
&& \sum_{ i \in V_t : i\succeq (u,k_1)} \sum_{j \in V_t : j\succeq (u,k_2)} \E\left(G_t^i G_t^j \big\vert \,\beta(u), X^{(u,k_1)}_{\beta(u)},X^{(u,k_2)}_{\beta(u)} \right) \\
&&\qquad \qquad = \sum_{ i \in V_t : i\succeq (u,k_1)}  \E\left(G_t^i  \big\vert \,\beta(u),  X^{(u,k_1)}_{\beta(u)}\right) \times \sum_{j \in V_t : j\succeq (u,k_2)} \E\left( G_t^j \big\vert \,\beta(u), X^{(u,k_2)}_{\beta(u)}\right) 
\Eea
by the branching property. Moreover  the many-to-one formula  \eqref{defYt}
%, conditionally on $\beta(u)$  \eqref{moy-sc} 
ensures that
\Bea
\sum_{ i \in V_t : i\succeq uk_1} \E\left(G_t^i  \big\vert\,  \beta(u), X^{(u,k_1)}_{\beta(u)}\right)&=&\E(N_t)\E_{X^{(u,k_1)}_{\beta(u)}} \left(f(Y_{t-\beta(u)})-\mu(f)\right)
\Eea
on the event $\{\beta(u)<t, \ (u,k_1) \in \mathcal I, (u,k_2) \in \mathcal I\}$. The  convergence \eqref{limf} ensures that the second term in the right-hand side tends to zero for $\beta(u)$ fixed. This convergence depends on the initial condition. Nevertheless  this difficulty can be overcome by  proving (see \cite{BDMT}) that  the common ancestor of two individuals lives almost-surely at the beginning of the continuous time Galton-Watson tree. This fact also allows  to sum over $(u,k_1) \in \mathcal I, (u,k_2) \in \mathcal I$ and obtain that
$\, C_t \ll\E(N_t)^2$
by dominated convergence arguments. Recalling that $N_t/\E(N_t)$ converges to $W$ in $\mathbb{L}^2$ yields the result.
\end{proof}

\subsection{Application  to splitting diffusions}
\me For the sake of simplicity, we assume in this section that the branching events are binary ($p(\ud k)=\delta_2(\ud k)$), so that the genealogical tree is the Yule tree. We describe  a population of infected cells undergoing a binary division, as in  the previous section for constant division rates. When a division occurs, 
 a random fraction $F$
is inherited by a daughter cell and the rest by the other daughter cell. But in contrast with the previous section, the process $X$ may not be a branching process, which allows for example to take into account ressources limitation for the parasites living in the cell. Here, $X$ is a diffusion with  infinitesimal generator
$$Lf(x)=r(x) f'(x)+\frac{\sigma(x)^2}{2} f''(x)$$
and we require the ergodicity of the auxiliary process $Y$. We refer to \cite{BDMT} for other applications, such as cellular aging.
%Then the transition kernel is given by $P^{(2)}(x,dx_1dx_2)$ and we denote by 
%$G(x,dx_1)$ its first marginal (recalling that we $P^{(2)}$ has been symmetrized)

\me The  infinitesimal generator of the auxiliary process $Y$ is characterized for $f\in
\Co^2_b(\R,\R)$ by:
\begin{align}
Af(x) %& \gamma(x)f'(x)+\sigma(x)f''(x)+\tau \int_0^1 \left(\frac{1}{2}\big(f(\theta x)-f(x)\big)+\frac{1}{2}\big(fFx)-f(x)\big)\right)\P(F \in d\theta)\nonumber\\
= &  r(x)f'(x)+\frac{\sigma(x)^2}{2} f''(x)+2\tau \int_0^1 \big(f(\theta x)-f(x)\big)\P(F \in d\theta).
\end{align}

\begin{Prop}\label{propapplimeyntweedie}
  Assume that $Y$ is  a Feller process which is  irreducible, i.e. there exists  a probability measure $\nu$ on $\mathbb R$ such that for any measurable set $B$ and $x\in \mathbb R$, 
$$\nu(B)>0 \Rightarrow \int_0^{\infty} \P_x( Y_t\in B) \ud t>0. $$
%(see \cite{meyntweedie} p. 520) and 
Assume also that there exists $K\geq 0$, such that for every $|x|>K$, $r(x)<\tau'\vert x \vert$  for some   $\tau'<\tau$. \\
Then,  $Y$   is
  ergodic  with stationary  probability  $\pi$ and   we have  
$$\frac{\ind_{\{N_t>0\}}}{N_t} \#\{ i \in V_t : X_t^i \in A \} \stackrel{t\rightarrow \infty}{\longrightarrow} \pi (A)$$
for every  Borelian set $A$ such that $\pi(\partial A)=0$ and $\partial A$ is the boundary of $A$.
\end{Prop}

\begin{proof}
  Once    we    check    that    $Y$    is    ergodic,    the second part comes from  Theorem \ref{thLGNannonceintro}.
%and the fact that  $W$  defined  by  \cite{eq:CV-N-W}   
 %readily implies the weak convergence of the Proposition.
The ergodicity of $Y$ is based on Theorems 4.1 of \cite{meyntweedieII} and 6.1 of \cite{meyntweedie}. Since $Y$ is Feller and irreducible, the process $Y$ admits a unique invariant probability measure $\pi$ and is exponentially ergodic provided  there exists a positive measurable function $\,V\,$ such that $\lim_{x\rightarrow \pm \infty}V(x)=+\infty$ and for which:
\begin{align}
\exists c>0,\, d\in \R,\, \forall x\in \R,\, AV(x)\leq -cV(x)+d.\label{conditiondeLyapounov}
\end{align}
For  $V(x)=|x|$  regularized   on  an  $\varepsilon$-neighborhood  of  0
($0<\varepsilon<1$), we have:
\begin{align}
\forall |x|>\varepsilon, \, AV(x)= &\mbox{sign}(x) r(x)+2\tau |x| \E( F -1)=\mbox{sign}(x)r(x)-\tau |x|,\label{applifosterlyapunov}
\end{align}as the distribution of $F$ is symmetric with respect to $1/2$. Then, by assumption, there exist $\eta>0$ and $K>\varepsilon$ such that %(\ref{applifosterlyapunov}) implies:
\begin{align}
\forall x\in \R,\, AV(x)\leq -\eta V(x)+\big(\sup_{|x|\leq K}|r(x)|+\tau K\big)\ind_{\{|x|\leq K\}}.\label{critereLyapounovexemple}
\end{align}This implies (\ref{conditiondeLyapounov}) and  the geometric ergodicity gives us that
\begin{equation*}
\exists \beta>0,\, B<+\infty,\,\forall t\in \R_+,\, \forall x\in \R,\, \sup_{g \,/\, |g(u)|\leq
  1+|u|}\big|\E_x( g(Y_t))-\langle \pi,g\rangle\big|\leq B(1+|x|)\expp{-\beta t}.%\label{exponentielltgeom1}
\end{equation*}
The proof is complete.
\end{proof}

\subsection{Some extensions}

\me Following Section \ref{LPA}, we could consider   a model for cell division with parasites where the growth of parasites is limited by the resources for the cells. The Markovian dynamics 
of the parasite population size could be described by a logistic Feller diffusion process. Since this process  goes to extinction almost surely (or to a finite positive limit if the process is deterministic), Proposition \ref{propapplimeyntweedie} may be applied to derive the asymptotic distribution of the infection among the cell population. The construction of the model and the proofs are left to the reader. 

\bi
In the other hand, let us note  that the many-to-one formula \eqref{defYt} holds for $f$ depending on time. Therefore the  large numbers law (Theorem  \ref{thLGNannonceintro})
can be extended to the case where $Y$ isn't ergodic as soon as  we can find some renormalization  $g_t$ such  $g_{t}(Y_{t})$ satisfies \eqref{limf}. We refer to \cite{BDMT} 
for an application when $X$ is a branching L\'evy process and in particular we recover the classical central limit theorem for  branching  Brownian motions.    
%The Branching Brownian motion has been extensively studied and we refer e.g. to the monographs of Z. Shi and J. Berestycki. (REFERENCES). 
%\begin{exo}[Branching Brownian Motion] Assume that  and the splitting is local and binary.
%What can you say about the repartition of the traits among the population when $t\rightarrow \infty$ ?
%\end{exo}

\section{Appendix : Poisson point measures}
\noindent In this appendix, we summarize  the main definitions and results concerning the Poisson point  measures. The reader can consult the two main books   by Ikeda-Watanabe \cite{SDEsApp1} and  by Jacod-Shiryaev \cite{Jacod} for more details.

\begin{Def}
Let $(E,\EE)$ be a measurable space and $\mu$ a $\sigma$-finite measure  on this space.  A (homogeneous)  Poisson point measure $N$  with intensity $\, \mu(dh) dt\,$ on $\mathbb{R}_{+}\times E$ is a $(\mathbb{R}_{+}\times E, {\cal B}(\mathbb{R}_{+})\otimes \EE)$-random measure  defined on a probability space $(\Omega, {\cal F}, \mathbb{P})$ which satisfies the following properties:
\begin{enumerate}
\item $N$ is a counting measure: $\forall \widehat{A}\in {\cal B}(\mathbb{R}_{+})\otimes \EE$, $\forall \omega\in \Omega$, $N(\omega, \widehat{A}) \in \mathbb{N}\cup \{+\infty\}$.

\item $\forall \omega\in \Omega$, $N(\omega, \{0\}\times E) = 0$: no jump at time $0$.

\item $\forall   \widehat{A}\in {\cal B}(\mathbb{R}_{+})\otimes \EE$, $\E(N( \widehat{A})) = \nu(\widehat{A})$, where $\nu(dt,dh) = \mu(dh) dt\,$.

\item If  $\widehat{A}$ and $ \widehat{B}$ are disjoint in ${\cal B}(\mathbb{R}_{+})\otimes \EE$ and if $\nu(\widehat{A})<+\infty, \nu(\hat{B})<+\infty$, then the random variables $N(\widehat{A})$ and $N(\widehat{B})$ are independent. 
\end{enumerate}
\end{Def}

\me The existence of such a  Poisson point measure with intensity $\mu(dh)dt$ is proven in \cite{Jacod}, for any $\sigma$-finite measure $\mu$ on $(E,\EE)$. 

\me 
Let us remark that for any $A\in \EE$ with $\mu(A)<\infty$ the process defined by
$$N_{t}(A) = N((0,t]\times A)$$ is a Poisson process with intensity 
$\mu(A)$. 

\begin{Def} The filtration $({\cal F}_{t})_{t}$ generated by $N$ is given by
$${\cal F}_{t}= \sigma(N((0,s]\times A), \forall s\leq t, \forall A \in \EE).$$
If $\widehat{A}\in (s,t]\times \EE$ and $\nu(\widehat{A})<\infty$, then $N(\widehat{A})$ is independent of ${\cal F}_{s}$.
\end{Def}

\bi
Let us first assume that the measure $\mu$ is finite on $(E,\EE)$. Then $(N_{t}(E), t\geq 0)$ is a Poisson process with intensity $\mu(E)$. The point measure is associated with a compound Poisson process. Indeed, let us write
$$\mu(dh) = \mu(E)\, {\mu(dh)\over \mu(E)},$$ the decomposition of the measure $\mu$ as the product of the jump rate $\mu(E)$ and the jump amplitude law $ {\mu(dh)\over \mu(E)}$. Let us fix $T>0$ and introduce $T_{1}, \ldots, T_{\gamma}$ the jump times of the process $(N_{t}(E), t\geq 0)$ between $0$ and $T$. We know that the jump number  $\gamma$ is a Poisson variable with parameter $T \mu(E)$. Moreover, conditionally on $\gamma$, $T_{1}, \ldots, T_{\gamma}$, the jumps $(U_{n})_{n=1,\ldots, \gamma}$ are independent with the same law $ {\mu(dh)\over \mu(E)}$. We can write in this case
$$N(dt, dh) = \sum_{n=1}^\gamma
 \delta_{(T_{n}, U_{n})}.$$
 Therefore, one can define for any measurable function $G(\omega, s, h)$ defined on $\Omega\times \mathbb{R}_{+} \times E$ the random variable \ben
 \int_{0}^T\int_{E} G(\omega, s, h) N(\omega, ds,dh) =\, \sum_{n=1}^\gamma
G(\omega, T_{n}, U_{n}).\een
In the following, we will forget the $\omega$. Let us remark that $T\longrightarrow  \int_{0}^T\int_{E} G(s, h) N( ds,dh)$  is a finite variation process which is  increasing if $G$ is positive. A main example is the case where $G(\omega, s, h)=h$. Then $$X_{T}=\int_{0}^T\int_{E}  h\, N(ds,dh) =\, \sum_{n=1}^\gamma
U_{n} = \sum_{s\leq T} \Delta X_{s}$$  is the sum of the jumps between $0$ and $T$. 

\bi Our aim now is to generalize the  definition of the integral of $G$ with respect to $\,N$ when  $\mu(E)=+\infty$. In this case, one can have an accumulation of jumps during the finite time interval $[0,T]$ and the counting measure $N$ is associated with  a countable set of points:
$$N = \sum_{n\geq 1} \delta_{(T_{n},U_{n})}.$$ We need additional properties on the process $G$.  

\noindent Since $\mu$ is $\sigma$-finite, there exists an increasing  sequence $(E_{p})_{p\in \mathbb{N}}$ of subsets of $E$ such that $\mu(E_{p})<\infty$ for each $p$ and $E=\cup_{p} E_{p}$.  As before we can define $\int_{0}^T\int_{E_{p}} G(s, h) N( ds,dh)$ for any $p$.

\me
We introduce the predictable $\sigma$-field ${\cal P}$ on $\Omega\times \mathbb{R}_{+}$ (generated by all left-continuous adapted processes) and define a predictable process $(G(s,h), s\in \mathbb{R}_{+}, h\in E)$ as a ${\cal P}\otimes \EE$ measurable process. 
\begin{Thm} Let us consider a predictable process  G(s, h) and assume that 
\be
\label{H-integ}
\E\left(\int_{0}^T\int_{E} |G(s, h)| \mu(dh)ds\right)<+\infty.
\ee
1) The sequence of random variables $\left(\int_{0}^T\int_{E_{p}} G(s, h) N(ds,dh)\right)_{p}$ is Cauchy in $\mathbb{L}^{1}$ and converges to a  $\mathbb{L}^{1}$-random variable that we denote by $\,\int_{0}^T\int_{E} G(s, h) N(ds,dh)$.  It's an increasing process if $G$ is non-negative. Moreover,
we get 
$$
%\label{H-exp}
\E\left(\int_{0}^T\int_{E} G(s, h) N(ds,dh)\right)=\E\left(\int_{0}^T\int_{E} G(s, h) \mu(dh)ds\right)
$$

\me 2) The process $M=(\int_{0}^t\int_{E} G(s, h) N(ds,dh)-\int_{0}^t\int_{E} G(s, h) \mu(dh)ds, t\leq T)$ is a martingale.

\noindent The random measure 
$$\widetilde N(ds,dh) = N(ds,dh) - \mu(dh)ds$$
is called the compensated martingale-measure of $N$.

\me 3) If we assume moreover that \be
\label{H-integ2}
\E\left(\int_{0}^T\int_{E} G^2(s, h) \mu(dh)ds\right)<+\infty,
\ee
then the martingale $M$ is square-integrable with quadratic variation 
$$
%\label{H-exp2}
\langle M \rangle_{t} = \int_{0}^t\int_{E} G^2(s, h) \mu(dh)ds.$$

\end{Thm}

\noindent Let us remark that when \eqref{H-integ} holds, the random integral $\int_{0}^t\int_{E} G(s, h) N(ds,dh)$
can be defined  without the predictability assumption on $H$ but the martingale property of the stochastic integral  $\int_{0}^t\int_{E} G(s, h)  \widetilde N(ds,dh)$ is  only true  under this assumption.

\bi We can improve the condition under which the martingale $(M_{t})$ can be defined. The proof of the next theorem is tricky and consists in studying the $\mathbb{L}^{2}$-limit of the  sequence of martingales  $\int_{0}^t\int_{E_{p}} G(s, h)  \widetilde N(ds,dh)$ as $p$ tends to infinity. Once again, this sequence is Cauchy in $\mathbb{L}^{2}$ and converges to a limit which is a square-integrable martingale. Let us recall that the quadratic variation of a square-integrable martingale $M$ is the unique predictable process $\langle M \rangle$ such that 
$\ M^2 - \langle M \rangle\,$ is a martingale.

\begin{Thm}Let us consider a predictable process  G(s, h)  satisfying \eqref{H-integ2}. Then the process 
$M=(\int_{0}^t\int_{E} G(s, h) \widetilde N(ds,dh), t\leq T)$ is a square-integrable martingale with quadratic variation 
$$
%\label{H-exp2}
\langle M \rangle_{t} = \int_{0}^t\int_{E} G^2(s, h) \mu(dh)ds.$$

\end{Thm}

\bi If \eqref{H-integ2} is satisfied but not \eqref{H-integ}, the definition of $M$ comes from a $\mathbb{L}^{2}$- limiting argument, as  for the usual stochastic integrals. In this case the quantity $\int_{0}^t\int_{E} G(s, h) N(ds,dh)$ isn't always well defined and we are obliged to compensate.

 \bi
{\bf Example:} Let $\alpha\in (0,2)$. A symmetric $\alpha$-stable process $S$ can be written  
\be
\label{stable}S_{t}= \int_{0}^t\int_{\mathbb{R}} h \One_{\{0<|h|<1\}}\widetilde N(ds,dh) + \int_{0}^t\int_{\mathbb{R}} h \One_{\{|h|\geq 1\}}  N(ds,dh),
\ee
where $N(ds,dh)$ is a Poisson point measure with intensity $\mu(dh) ds = {1\over |h|^{1+\alpha}} dh ds$. 
There is an accumulation of  small jumps and the first term in the r.h.s. of \eqref{stable} is defined as a compensated martingale. The second term corresponds to the big jumps, which are in finite number on any finite time interval. 

\noindent If  $\alpha\in (1,2)$, then $\int h\wedge h^2 \mu(dh)<\infty$ and the process is integrable. If $\alpha\in (0,1)$, we only have that $\int 1\wedge h^2 \mu(dh)<\infty$ and the integrability of the process can fail. 

\bi Let us now consider a stochastic differential equation driven both by a Brownian term and a Poisson point measure. We consider a  random variable $X_{0}$, a Brownian motion $B$ and a Poisson point measure $N(ds,dh)$ on $\mathbb{R}_{+} \times \mathbb{R}$ with intensity $\mu(dh)ds$. Let us fix some measurable functions $b$ and $\sigma$ on $\mathbb{R}$  and   $G(x,h)$   and $K(x,h)$ on $\mathbb{R}\times \mathbb{R}$.

\me We consider a   process $X\in \mathbb{D}(\mathbb{R}_{+}, \mathbb{R})$ such that for any $t>0$, 
\be
\label{SDE-jump}
X_{t}&=& X_{0}+ \int_{0}^t b(X_{s}) ds + \int_{0}^t \sigma(X_{s}) dB_{s} \nonumber \\
&& \qquad \qquad + \int_{0}^t \int_{\mathbb{R}}G(X_{s-}, h) N(ds,dh) + \int_{0}^t \int_{\mathbb{R}}K(X_{s-}, h) \widetilde N(ds,dh).
\ee
\bi To give a sense to the equation, one expects 
  that for any $T>0$,
$$\E\left(\int_{0}^T\int_{\mathbb{R}} |G(X_{s}, h)| \mu(dh)ds\right)<+\infty \ ;\ \E\left(\int_{0}^T\int_{\mathbb{R}} K^2(X_{s}, h) \mu(dh)ds\right)<+\infty.
$$
We refer to \cite{SDEsApp1}  Chapter IV-9 for general existence and uniqueness assumptions (generalizing the Lipschitz continuity assumptions asked in the case without jump).

\bi Let us assume that a solution of  \eqref{SDE-jump} exists. The process $X$ is a left-limited and right-continuous semimartingale. A standard question is to ask when the process $f(X_{t})$ is a semimartingale and to know its Doob-Meyer decomposition. For a smooth function $f
$, there is   an It\^o's formula generalizing the usual one stated for continuous semimartingales.

\me Recall (cf.  Dellacherie-Meyer VIII-25 \cite{Del-Mey}) that for a   function $a(t)$ with bounded variation, the change of variable formula gives that for a $C^1$-function $f$,
$$f(a(t)) = f(a(0)) + \int_{(0,t]}f'(a(s)) da(s) + \sum_{0<s\leq t} (f(a(s) - f(a(s^-
) - \Delta a(s) f'(a(s^-
)).
$$

\me We wish to replace $a$ by a semimartingale.  We have to add smoothness to  $f$ and we will get two additional terms in the formula because of the two martingale terms. As in the continuous case, we assume that the function $f$ is $C^2$. 

 \begin{Thm} (see \cite{SDEsApp1}  Theorem 5.1 in Chapter II).
Let $f$ a $C^2$-function. Then  $f(X)$ is a semimartingale and for any $t$,
\begin{align}
%\label{ito-jump}
f(X_{t}) %&= f(X_{0}) + \int_{0}^t f'(X_{s-}) dX_{s}+  {1\over 2}\int_{0}^t f''(X_{s}) \sigma^2(X_{s}) ds    + \sum_{s\leq t}\left(f(X_{s-}+\Delta X_{s})- f(X_{s-}) - \Delta X_{s}f'(X_{s-})\right)  \nonumber\\
&= f(X_{0}) + \int_{0}^t f'(X_{s}) b(X_{s}) ds + \int_{0}^t f'(X_{s}) \sigma(X_{s}) dB_{s}+  {1\over 2}\int_{0}^t f''(X_{s}) \sigma^2(X_{s}) ds    \nonumber\\
&\hskip 0.5cm +\int_{0}^t \int_{\mathbb{R}}( f(X_{s-}+G(X_{s-}, h))- f(X_{s-})) N(ds,dh)  \nonumber\\
&\hskip 0.5cm + \int_{0}^t \int_{\mathbb{R}}(f(X_{s-}+K(X_{s-}, h))- f(X_{s-})) \widetilde N(ds,dh) \nonumber\\
&\hskip 0.5cm + \int_{0}^t \int_{\mathbb{R}}\left(f(X_{s}+K(X_{s}, h))- f(X_{s}) - K(X_{s}, h) f'(X_{s})\right) \mu(dh)ds.\label{ito-saut}\end{align}
\end{Thm}

\me
\begin{Cor}
\label{GenEt}
Under suitable integrability and regularity conditions on $b$, $\sigma$, $G$, $K$ and $\mu$, the process $X$ is a Markov process with extended generator: for any $C^2$-function $f$, for $x\in \mathbb{R}$, 
\be
\label{gen-jump}
Lf(x)& =& b(x) f'(x) + {1\over 2}\sigma^2(x)\f''(x) + \int_{\mathbb{R}} \left(f(x+G(x,h))-f(x)\right)\mu(dh) \nonumber\\
&&+ \int_{\mathbb{R}} \left(f(x+K(x,h))-f(x) - K(x,h) f'(x)\right)\mu(dh).
\ee
\end{Cor}

\bi
{\bf Example:} let us study the case where
 $$X_{t}=  X_{0}+ \int_{0}^t b(X_{s}) ds + \int_{0}^t \sigma(X_{s}) dB_{s}+ S_{t},$$
where $S$ is the stable process introduced in \eqref{stable}. Let us  consider a $C^2$-function $f$. Then $f(X)$ is a semimartingale and writes
\begin{align}
f(X_{t})& =  f(X_{0}) + M_{t} +\int_{0}^t f'(X_{s}) b(X_{s}) ds +  {1\over 2}\int_{0}^t f''(X_{s}) \sigma^2(X_{s}) ds    \nonumber\\
&\hskip 0.8cm +\int_{0}^t \int_{\mathbb{R}}( f(X_{s-}+h \One_{\{|h|>1\}})- f(X_{s-})){1\over |h|^{1+\alpha}} dhds \nonumber\\
&\hskip 0.8cm + \int_{0}^t \int_{\mathbb{R}}\left(f(X_{s-}+h\One_{\{|h|\leq1\}})- f(X_{s-}) - h\One_{\{|h|\leq1\}}f'(X_{s-})\right) {1\over |h|^{1+\alpha}} dhds \nonumber\\
& =  f(X_{0}) + M_{t} +\int_{0}^t f'(X_{s}) b(X_{s}) ds +  {1\over 2}\int_{0}^t f''(X_{s}) \sigma^2(X_{s}) ds  \nonumber\\
&\hskip 0.8cm + \int_{0}^t \int_{\mathbb{R}}\left(f(X_{s-}+h)- f(X_{s-}) - h\One_{\{|h|\leq1\}} f'(X_{s-})\right) {1\over |h|^{1+\alpha}} dhds, \nonumber
\end{align}
where $M$ is a martingale. 

 \vskip 1cm
 \noindent
Let us come back to the general case 
 and apply It\^o's formula \eqref{ito-saut} to $f(x)=x^2$:
\begin{align}
%\label{ito-jump}
X_{t}^2 
&= X_{0}^2 + \int_{0}^t 2 X_{s} b(X_{s}) ds + \int_{0}^t 2 X_{s-} \sigma(X_{s-}) dB_{s}+  \int_{0}^t  \sigma^2(X_{s}) ds    \nonumber\\
&\hskip 0.5cm +\int_{0}^t \int_{\mathbb{R}}( 2 X_{s-} G(X_{s-}, h)+(G(X_{s-}, h))^2) N(ds,dh)  \nonumber\\
&\hskip 0.5cm + \int_{0}^t \int_{\mathbb{R}}(2 X_{s-} K(X_{s-}, h) + (K(X_{s-}, h))^2) \widetilde N(ds,dh) \nonumber\\
&\hskip 0.5cm + \int_{0}^t \int_{\mathbb{R}} (K(X_{s-}, h))^2 \mu(dh)ds.\label{ito-saut}\end{align}

\bi In the other hand, since 
$$X_{t} = X_{0}+ M_{t} + A_{t},$$
where $M$ is square-integrable and $A$ has finite variation, then
$$X_{t}^2 
= X_{0}^2 + N_{t} + \int_{0}^t 2 X_{s-} dA_{s}  + \langle M\rangle_{t}
.$$
 Doob-Meyer's decomposition allows us to identify the martingale parts and the finite variation parts in the two previous decompositions and therefore
 \ben \langle M\rangle_{t} &=& \int_{0}^t  \sigma^2(X_{s}) ds   + \int_{0}^t \int_{\mathbb{R}} (G^2(X_{s-}, h)+ K^2(X_{s-}, h)) \mu(dh)ds.\een

\end{document}